%% file: PhaseTransHystHyper.tex
\documentclass[a4paper,12pt]{article}
 \usepackage[utf8]{inputenc}
 \usepackage{lmodern}
 \usepackage{amsmath}
 \usepackage{amssymb}
 \usepackage{amsthm}
 \usepackage{mathtools}
 \usepackage{vmargin} 
 \usepackage{graphicx}
 \usepackage{color}
 \usepackage[only,llbracket,rrbracket]{stmaryrd}
 \usepackage{hyperref}
 \usepackage{cite}

\usepackage{unicodeintex}

\newif\ifscreen
\ifscreen
\usepackage{cmbright}
\renewcommand{\sfdefault}{cmss}
\usepackage[italic,defaultmathsizes,basic,selfGreek]{mathastext}
\DeclareSymbolFont{letterg}{OML}{cmbrm}{m}{it}
\DeclareMathSymbol{g}{\mathalpha}{letterg}{`g}
\setmarginsrb{20mm}{10mm}{20mm}{10mm}{0mm}{10mm}{0mm}{10mm}
\fi

\renewcommand{\d}{\mathrm d}

\newtheorem{remark}{Remark}[section]
\newtheorem{definition}{Definition}[section]
\newtheorem{hypothesis}{Hypothesis}[section]
\newtheorem{theorem}{Theorem}
\newtheorem{proposition}{Proposition}[section]
\newtheorem{conjecture}{Conjecture}[section]
\newtheorem{lemma}{Lemma}

\numberwithin{equation}{section}

\title{Phase transitions, hysteresis, and hyperbolicity for self-organized alignment dynamics}

\author{Pierre Degond$^{(1,2)}$, Amic Frouvelle$^{(3)}$,  Jian-Guo Liu$^{(4)}$}
\date{}
\begin{document}
\maketitle

\begin{center}
1- Université de Toulouse; UPS, INSA, UT1, UTM;\\ 
Institut de Mathématiques de Toulouse; \\
F-31062 Toulouse, France. \\
2- CNRS; Institut de Mathématiques de Toulouse UMR 5219;\\ 
F-31062 Toulouse, France.\\
email: pierre.degond@math.univ-toulouse.fr
\end{center}

\begin{center}
3- CEREMADE, UMR CNRS 7534\\
Université Paris-Dauphine\\
75775 Paris Cedex 16, France\\
email: frouvelle@ceremade.dauphine.fr
\end{center}

\begin{center}
4- Department of Physics and Department of Mathematics\\
Duke University\\
Durham, NC 27708, USA\\
email: jliu@phy.duke.edu
\end{center}

\begin{abstract}
We provide a complete and rigorous description of phase transitions for kinetic models of self-propelled particles interacting through alignment.
These models exhibit a competition between alignment and noise.
Both the alignment frequency and noise intensity depend on a measure of the local alignment.
We show that, in the spatially homogeneous case, the phase transition features (number and nature of equilibria, stability, convergence rate, phase diagram, hysteresis) are totally encoded in how the ratio between the alignment and noise intensities depend on the local alignment.
In the spatially inhomogeneous case, we derive the macroscopic models associated to the stable equilibria and classify their hyperbolicity according to the same function.
\end{abstract}

\medskip
\noindent
{\bf Acknowledgements:} This work has been supported by KI-Net NSF RNMS grant No. 1107291.
The research of JGL was partially supported by NSF grant DMS 10-11738.
JGL is grateful for the opportunity to stay and work at the Institut de Mathématiques de Toulouse in fall 2012, under sponsoring of Centre National de la Recherche Scientifique and University Paul Sabatier and at University Paris-Dauphine, under the sponsoring of the french ‘Agence Nationale pour la Recherche (ANR)’ in the frame of the contract ‘CBDif-Fr’ (ANR-08-BLAN-0333-01).
AF wants to acknowledge partial support from the FP7-REGPOT-2009-1 project “Archimedes Center for Modeling, Analysis and Computation”. PD acknowledges support from the ‘Agence Nationale pour la Recherche (ANR)’ in the frame of the contract  ‘MOTIMO’ (ANR-11-MONU-009-01).


\vskip 0.4cm

\section{Introduction}
\label{sec-intro}

In this work we provide a complete and rigorous description of phase transitions in a general class of kinetic models describing self-propelled particles interacting through alignment.
These models have broad applications in physics, biology and social sciences for instance for the description of animal swarming behavior or opinion consensus formation.
Their essential feature is the competition between the alignment process which provides self-organization, and noise which destroys it.
An important point is that both the alignment frequency and noise intensity depend on a measure of the local alignment denoted by~$|J|$.
The phase transition behavior in the spatially homogeneous case is totally encoded in the ratio between these two functions denoted by~$k(|J|)$.
Namely we have the following features:
\begin{itemize}
\item[(i)] The function~$k$ gives rise to an algebraic compatibility relation whose roots provide the different branches of equilibria of the kinetic model.
One distinguished branch is given by isotropic or uniform distributions which correspond to no alignment at all, i.e.~$|J|=0$.
The other branches are associated to non-isotropic von Mises--Fisher distributions associated to non-zero~$|J|$.
\item[(ii)] The stability of these various equilibria is completely determined by the monotonicity of a function derived from~$k$ around these roots and there exists an exponential rate of local convergence of the solution to one of these stable equilibria. 

\item[(iii)] The global shape of this function~$k$ provides the phase diagram which encodes the order of the associated phase transitions.
According to its monotonicity, these can be second-order phase transitions, first-order phase transitions with hysteresis behavior or even be more complex.
For second-order phase transition, we give an explicit formula for the critical exponent in terms of the local behavior of~$k$.
The involved phase transitions are spontaneous symmetry breaking phase transitions between isotropic and non-isotropic equilibria.
Such phase transitions appear in many branches of physics, such as spontaneous magnetization in ferromagnetism, nematic phase transition in liquid crystals and polymers, Higgs mechanism of mass generation for the elementary particles.
\item[(iv)] In the spatially inhomogeneous case, we can derive the hydrodynamic equations associated to both the isotropic and non-isotropic stable equilibria (the former leading to diffusion behavior, the latter to hyperbolic models).
The hyperbolicity is again completely determined by this function, and is linked to the critical exponent in the case of a second-order phase transition.
\end{itemize}
To our knowledge, this is the first time that a complete mathematical theory of phase transitions in a physics system can be rigorously derived and related to one single object with high physical significance: this function~$k$.
One of the main achievement of this work is Theorem~\ref{thm-strong-stability-anisotropic}, which provides part of point~$(ii)$ above, namely the nonlinear stability of the non-isotropic equilibria (the von Mises--Fisher distributions) when the function associated to~$k$ is increasing.
To be more precise, let us write this set of equilibria as~$\{f^{eq}_Ω,Ω∈𝕊\}$ (it has the same symmetries as the unit sphere~$𝕊$ of~$ℝ^n$,~$n$ being the dimension of the model). Then, we have a rate of convergence~$λ$ and two positive constants~$δ$ and~$C$ such that, if the initial condition~$f_0$ satisfies~$∥f_0-f_{Ω}^{eq}∥<δ$ for some~$Ω∈𝕊$, then there exist~$Ω_∞∈𝕊$ such that for all~$t>0$, the solution~$f$ of the spatially homogeneous model satisfies
\[∥f(t)-f_{Ω_∞}^{eq}∥⩽C∥f_0-f_{Ω}^{eq}∥ \,e^{-λt}.\]
This stability result takes place in the Sobolev space~$H^s$ as long as~$s>\frac{n-1}2$.
In previous works (in the case where the function~$k$ is linear) such as~\cite{frouvelle2012dynamics}  or~\cite{giacomin2012global} (for the Kuramoto model in dimension~$n=2$, where a precise study of the attractor is performed), the exponential convergence with rate~$β$ was only proven for all~$β<λ$, and the existence of such a constant~$C$ independent of~$f_0$ was lacking. 

Self-propelled particle systems interacting through alignment have been widely used in the modeling of animal swarms (see e.g.
the review~\cite{vicsek2012collective} and~\cite{bertin2009hydrodynamic, chuang2007state, cucker2007emergent, toner2005hydrodynamics, vicsek1995novel}).
Kinetic models of self-propelled particles have been introduced and studied in~\cite{ bolley2012meanfield, carrillo2010asymptotic, degond2008continuum, ha2009simple, ha2008particle}.
Here, specifically, we are interested in understanding phase transitions and continuum models associated to the Vicsek particle system~\cite{vicsek1995novel}.
Phase transitions in the Vicsek system have been widely studied in the physics literature~\cite{aldana2009emergence, chate2008collective}.
There have been some controversies whether the involved phase transitions were first or second order.
In some sense, this paper provides a complete answer to this question, at least in the kinetic framework.

The passage from the kinetic to macroscopic descriptions of the Vicsek system has first been  proposed in~\cite{degond2008continuum}.
Further elaboration of the model can be found in~\cite{degond2012hydrodynamic, frouvelle2012continuum}.
The resulting continuum model is now referred to as the Self-Organized Hydrodynamic (SOH) model.
In these derivations of the SOH, the noise and alignment intensities are functions of the local densities and not of the local alignment.
No phase transition results from this choice but the resulting SOH models are hyperbolic.
In~\cite{frouvelle2012dynamics, degond2012macroscopic}, alignment intensity has been made proportional to the local alignment.
Second-order phase transition have been obtained.
However, the resulting SOH model is not hyperbolic.
In the present paper, we investigate general relations between the noise and alignment intensities upon the local alignment~$|J|$.
As described above, the phase diagram becomes extremely complex and its complexity is fully deciphered here.
The kind of alignment phase transition that we find here is similar to nematic phase transitions in liquid crystals, polymer dynamics and ferromagnetism~\cite{constantin2004asymptotic, doi1999theory, fatkullin2005critical, maier1958eine, onsager1949effects}.

The organization of the paper is as follows.
In section~\ref{sec-kinetic}, we derive the kinetic model from the particle system and determine its equilibria.
In section~\ref{sec-rates-convergence}, we study the stability of these equilibria in the spatially homogeneous case and find the rates of convergences of the solution to the stable ones.
Then, in section~\ref{sec-phase-transition}, we use these results to study two examples respectively leading to second order and first order phase transitions, and in the case of first order phase transitions, to the hysteresis phenomenon.
Finally, in section~\ref{sec-macro-models}, we return to the spatially inhomogeneous case and investigate the macroscopic limit of the kinetic model towards hydrodynamic or diffusion models according to the considered type of equilibrium.
For the hydrodynamic limit, we provide conditions for the model to by hyperbolic.
Finally, a conclusion is drawn in section~\ref{sec-conclu}.
We supplement this paper with appendix~\ref{sec-numeric} which provides elements on the numerical simulation of the hysteresis phenomenon.

\section{Kinetic model and equilibria}
\label{sec-kinetic}

In this section, we derive the mean-field kinetic model from the particle system, and determine its equilibria.
We begin with the particle model in the next section.
Then, in section~\ref{subsec-mean-field} we derive the mean-field limit.
The space-homogeneous case will be highlighted in section~\ref{subsec-space-homogeneous} and the equilibria will be determined in section~\ref{subsec-equilibria}.

\subsection{The particle model}
\label{sec-discrete}

We consider a system of a large number~$N$ of socially interacting agents defined by their positions~${X_i∈ℝ^n}$ and the directions of their velocities~$ω_i∈𝕊$ (where~$𝕊$ is the unit sphere of~$ℝ^n$).
They obey the following rules, which are a time continuous version of those of the Vicsek model~\cite{vicsek1995novel}:
\begin{itemize}
\item[-] they move at constant speed~$a$,
\item[-] they align with the average direction of their neighbors, as a consequence of the social interaction.
\item[-] the directions of their velocities are subject to independent random noises, which expresses either some inaccuracy in the computation of the social force by the subject, or some trend to move away from the group in order to explore the surrounding environment.
\end{itemize}
These rules are expressed by the following system of stochastic differential equations:
\begin{align}
&\d X_i=a\, ω_i\d t, \label{particle-model-dX}\\
&\d ω_i=ν(|\mathcal J_i|) P_{ω_i^⊥}\,\bar{ω}_i \d t + √{2τ(|\mathcal J_i|)}P_{ω_i^⊥}∘\d B^i_t,\label{particle-model-domega}\\
&\bar{ω}_i=\frac{\mathcal J_i}{|\mathcal J_i|}, \quad \mathcal J_i = \frac{a}N\sum_{ℓ=1}^N K(|X_ℓ-X_i|) \,ω_ℓ.
\label{particle-model-omega-J}
\end{align}
 Eq.~\eqref{particle-model-dX} simply translates that particle~$i$ moves with velocity~$a\,ω_i$.
The first term at the right-hand side of~\eqref{particle-model-domega} is the social force, which takes the form of a relaxation of the velocity direction towards the mean direction of the neighbors~$\bar{ω}_i$, with relaxation rate~$ν$ (the operator~$P_{ω_i^⊥}$ is the projection on the tangent space orthogonal to~$ω_i$, ensuring that~$ω_i$ remains a unit vector).
Eq.~\eqref{particle-model-omega-J} states that the mean direction is obtained through the normalization of the average current~$\mathcal J_i$, itself computed as the average of the velocities of the particles.
This average is weighted by the observation kernel~$K$, which is a function of the distance between the test particle~$i$ and its considered partner~$ℓ$.
Without loss of generality, we can assume that~$∫ K(|ξ|) \, \d ξ = 1$.
The second term of~\eqref{particle-model-domega} models the noise in the velocity direction.
Eq.~\eqref{particle-model-domega} must be understood in the Stratonovich sense (as indicated by the symbol~$∘$), with~$N$ independent standard Brownian motions~$B^i_t$ in~$ℝ^n$.
The quantity~$τ>0$ is the variance of the Brownian processes.

In this paper, we assume that the relaxation rate~$ν$ and the noise intensity~$τ$ are functions of the norm of the current~$|\mathcal J|$.
The present hypothesis constitutes a major difference with previous works.
Indeed, the case where~$ν$ and~$τ$ are constant has been investigated in~\cite{degond2008continuum}, while the case where~$ν(|\mathcal J|)=|\mathcal J|$ and~$τ=1$ has been treated in~\cite{degond2012macroscopic}.
We recall that no phase transition appears at the macroscopic level in the first case while in the second case, a phase transition appears.
This phase transition corresponds to a change in the number of equilibria as the density crosses a certain threshold called critical density.
The critical exponent is~$1/2$ in this case.
Here, we investigate the more general case of almost arbitrary dependences of~$ν$ and~$τ$ upon~$|\mathcal J|$, and show that the phase transition patterns can be much more complex than those found in~\cite{degond2012macroscopic}.
For later convenience, we will denote by~$τ_0>0$ the value of~$τ(0)$.

To understand why~$|\mathcal J|$ is the crucial parameter in this discussion, let us introduce the local density~$ρ_i$ and order parameter (or mean alignment)~$c_i$ as follows:  
\[c_i=\frac{|\mathcal J_i|}{a\, ρ_i}, \quad ρ_i = \frac1N\sum_{ℓ=1}^N K(|X_ℓ-X_i|),\]
and we note that~$0⩽ c_i⩽1$.
The value~$c_i\sim 0$ corresponds to disorganized motion, with an almost isotropic distribution of velocity directions, while~$c_i\sim 1$ characterizes a fully organized system where particles are all aligned.
Therefore~$|\mathcal J_i|$ appears as the ``density of alignment'' and increases with both particle density and order parameter.
This paper highlights that the dependence of~$ν$ and~$τ$ upon~$|\mathcal J_i|$ acts as a positive feedback which triggers the phase transition.
Besides, in~\cite{frouvelle2012continuum}, it has been shown that making~$ν$ and~$τ$ depend on the density~$ρ$ only does not produce any phase transition, and that the recovered situation is qualitatively similar to that of~\cite{degond2008continuum}.
The present work could be extended to~$ν$ and~$τ$ depending on both~$ρ$ and~$|\mathcal J|$ at the expense of an increased technicality, which will be omitted here.
The present framework is sufficient to cover all interesting situations that can be desirable at the macroscopic scale.

\subsection{Mean-field derivation of the kinetic model}
\label{subsec-mean-field}

The first step in the study of the macroscopic behaviour of this system consists in considering a large number of particles.
In this limit, we aim at describing the evolution of the density probability function~$f^N(x,ω,t)$ of finding a particle with direction~$ω$ at position~$x$.
This has been studied in~\cite{bolley2012meanfield} in the case where~$ν(|\mathcal J|)=|\mathcal J|$ and~$τ=1$.
It is nearly straightforward to perform the same study in our more general case.

For convenience, we will use the following notation for the first moment of a function~$f$ with respect to the variable~$ω$ (the measure on the sphere is the uniform measure such that~$∫_𝕊\dω=1$):
\begin{equation}
\label{def-J}
J_f(x,t)=∫_{ ω ∈ 𝕊 }  ω \, f(x, ω,t) \, \dω.
\end{equation}
For the following, we will assume that:
\begin{hypothesis}~
\label{hyp-prop-chaos}
\begin{enumerate}
\item[(i)] The function~$K$ is a Lipschitz bounded function with finite second moment.
\item[(ii)] The functions~$|J|↦\frac{ν(|J|)}{|J|}$ and~$|J|↦τ(|J|)$ are Lipschitz and bounded.
\end{enumerate}
\end{hypothesis}

In these conditions the mean-field limit of the particle model is the following kinetic equation, called Kolmogorov--Fokker--Planck equation:
\begin{gather}
\label{KFP-meanfield}
∂_t f + a\, ω · ∇_x f + ν(|\mathcal J_f|) ∇_ω · (P_{ω^⊥} \bar{ω}_f \, f) = τ(|\mathcal J_f|) Δ_ω f
\intertext{with}
\label{def-mcJ-omega-bar}
\mathcal J_f(x,t) = a\, (K*J_f) (x,t)\,,\quad \bar{ω}_f=\frac{\mathcal J_f}{|\mathcal J_f|} ,
\end{gather}
where~$*$ denotes the convolution in~$ℝ^n$ (only on the~$x$ variable),~$Δ_ω$ and~$∇_ω·$ stand for the Laplace-Beltrami and divergence operators on the sphere~$𝕊$.

More precisely, the following statements hold:

\begin{proposition}
\label{prop-meanfield-limit}
If~$f_0$ is a probability measure on~$ℝ^n\times𝕊$ with finite second moment in~$x∈ℝ^n$, and if~$(X_i^0,ω_i^0)_{i∈⟦1,N⟧}$ are~$N$ independent variables with law~$f_0$, then: 

\begin{enumerate}
\item[(i)] There exists a pathwise unique global solution~$f$ to the particle system~\eqref{particle-model-dX}-\eqref{particle-model-omega-J} with initial data~$(X_i^0,ω_i^0)$.
\item[(ii)] There exists a unique global weak solution of the kinetic equation~\eqref{KFP-meanfield} with initial data~$f_0$.
\item[(iii)] The law~$f^N$ at time~$t$ of any of one of the processes~$(X_i,ω_i)$ converges to~$f$ as~$N→∞$.
\end{enumerate}
\end{proposition}

The proof of this proposition follows exactly the study performed in~\cite{bolley2012meanfield}, using auxiliary coupling processes as in the classical Sznitman’s theory (see~\cite{sznitman1991topics}), and is omitted here.
Let us make some comment on the structure of the kinetic equation~\eqref{KFP-meanfield}.
The first two terms of the left hand side of~\eqref{KFP-meanfield} correspond to the free transport with speed given by~$a\,ω$.
It corresponds to~\eqref{particle-model-dX} in the particle model.
The last term of the left hand side corresponds to the alignment mechanism towards the target orientation~$\bar{ω}_f$, with intensity~$ν(|\mathcal J_f|)$, while the term at the right hand side is a diffusion term in the velocity variable, with intensity~$τ(|\mathcal J_f|)$.
These two terms correspond to~\eqref{particle-model-domega} in the particle model.
We will see in~\eqref{KFP-homogeneous} and~\eqref{KFP-meanfield-eps} that these two terms, under certain assumptions (spatially homogeneous case, or expansion in terms of a scaling parameter~$η$), behave as a local collision operator~$Q$, only acting on the velocity variable~$ω$.
Finally, the convolution with~$K$ in~\eqref{def-mcJ-omega-bar} expresses the fact that~$\mathcal J_f$ is a spatial averaging of the local momentum~$J_f$ defined in~\eqref{def-J}, it corresponds to the definition~\eqref{particle-model-omega-J} in the particle model.

\subsection{The space-homogeneous kinetic model}
\label{subsec-space-homogeneous}

The hydrodynamic limit involves an expansion of the solution around a local equilibrium (see section~\ref{subsec-hydro-scaling}).
Therefore, local equilibria of the collision operator~$Q$ are of key importance.
We will see that such equilibria are not unique.
The existence of multiple equilibria requires an a priori selection of those equilibria which make sense for the hydrodynamic limit.
Obviously, unstable equilibria have to be ignored because no actual solution will be close to them.
In order to make this selection, in the present section, we consider the spatially homogeneous problem.
To the most possible exhaustive way, in section~\ref{sec-rates-convergence}, we will determine the stable equilibria and characterize the convergence rate of the solution of the space-homogeneous problem to one of these equilibria.
In section~\ref{sec-phase-transition}, we will illustrate these results on two examples.
Finally, in section~\ref{sec-macro-models}, we will deal with the spatially non-homogeneous case and apply the conclusions of the spatially homogeneous study.

The spatially homogeneous version of this model consists in looking for solutions of the kinetic equation~\eqref{KFP-meanfield} depending only on~$ω$ and~$t$.
Obviously, such solutions cannot be probability measures on~$ℝ^n×𝕊$ any more, so we are looking for solutions which are positive measures on~$𝕊$.
In that case,~${\mathcal{J}_f= a J_f}$, and (up to writing~$\hat{ν}(|J_f|)=ν(a|J_f|)$ and~$\hat{τ}(|J_f|)=τ(a|J_f|)$) the kinetic equation~\eqref{KFP-meanfield} reduces to
\begin{equation}
\label{KFP-homogeneous}
∂_t f  = Q(f),
\end{equation}
where the operator~$Q$ is defined by
\begin{equation}
\label{def-Q}
Q(f)= - ν(|J_f|)∇_ω · (P_{ω^⊥} Ω_{f} \, f) + τ(|J_f|)Δ_ω f,
\end{equation}
where~$Ω_f=\frac{J_f}{|J_f|}$ and where we have dropped the ``hats'' for the sake of clarity.
Let us remark that by hypothesis~\ref{hyp-prop-chaos}, we do not have any problem of singularity of~$Q$ as~$|J_f|→0$: if~$|J_f|=0$, we simply have~$Q(f)=τ_0Δ_ω f$.

The investigation of the properties of the operator~$Q$ is of primary importance, as we will see later on.
For later usage, we define
\begin{equation}
\label{def-k-phi}
k(|J|)=\frac{ν(|J|)}{τ(|J|)}, \quad Φ(r)=∫_0^r k(s)\d s,
\end{equation}
so that~$Φ(|J|)$ is an antiderivative of~$k$:~$\frac{\d Φ}{\d |J|}=k(|J|)$. The space-homogeneous dynamics corresponds to the gradient flow of the following free energy functional:
\begin{equation}
\label{def-free-energy}
\mathcal F(f)=∫_𝕊f\ln f\, \d ω - Φ(|J_f|).
\end{equation}
Indeed, if we define the dissipation term~$\mathcal D(f)$ by
\begin{equation}
\label{def-dissipation}
\mathcal D(f)=τ(|J_f|)∫_𝕊f\,|∇_ω(\ln f-k(|J_f|)\,ω\cdotΩ_f)|^2\, \d ω,
\end{equation}
we get the following conservation relation:
\begin{equation}
\label{dissipation-free-energy}
\frac{\d}{\d t} \mathcal F(f)=-\mathcal D(f)⩽0.
\end{equation}
The main ingredient to derive this relation is the identity~$P_{ω^⊥}Ω_f=∇_ω(ω·Ω_f)$.
Therefore, the collision operator~$Q$ defined in~\eqref{def-Q} can be written:
\begin{equation}
\label{eq-Q-div1}
Q(f)=τ(|J_f|)∇_ω·\big[f\,∇_ω(\ln f - k(|J_f|)\,ω·Ω_f)\big].
\end{equation}
Finally, since
\[\frac{\d}{\d t}\mathcal F=∫_𝕊 ∂_t f (\ln f - k(|J_f|)\,ω·Ω_f) \,\d ω\, ,\]
using~\eqref{KFP-homogeneous},~\eqref{eq-Q-div1} and integrating by parts, we get~\eqref{dissipation-free-energy}.

We first state results about existence, uniqueness, positivity and regularity of the solutions of~\eqref{KFP-homogeneous}.
Under hypothesis~\ref{hyp-prop-chaos}, we have the following
\begin{theorem}
\label{theorem-existence-uniqueness}
Given an initial finite nonnegative measure~$f_0$ in~$H^s(𝕊)$, there exists a unique weak solution~$f$ of~\eqref{KFP-homogeneous} such that~$f(0)=f_0$.
This solution is global in time.
Moreover,~$f∈ C^1(ℝ^*_+,C^∞(𝕊))$, with~$f(ω,t)>0$ for all positive~$t$.

Finally, we have the following instantaneous regularity and uniform boundedness estimates (for~$m∈ℕ$, the constant~$C$ being independent of~$f_0$):
\[∥f(t)∥^2_{H^{s+m}}⩽ C\left(1+\frac1{t^m}\right)∥f_0∥^2_{H^{s}}.\]
\end{theorem}

The proof of this theorem follows exactly the lines of the proof given in~\cite{frouvelle2012dynamics} for the case where~$ν(|J|)=|J|$, and will be omitted here.
Let us remark that here we do not need the bounds on~$\frac{ν(|J|)}{|J|}$ and on~$τ$ provided by hypothesis~\ref{hyp-prop-chaos}, since the positivity ensures that~$|J|$ takes values in~$[0,ρ_0]$, where~$ρ_0$ is the total mass of~$f_0$ (a conserved quantity).
Therefore~$τ$ is uniformly bounded from below in time, by a positive quantity~$τ_{min}$, and~$\frac{ν(|J|)}{|J|}$ is also uniformly bounded from above in time.
Finally, the fact that~$f$ is only~$C^1$ in time comes from the fact that the proof only gives~$f∈C([0,T],H^s(𝕊))$ for all~$s$, and we use the equation to get one more derivative.
We could obtain a better time regularity at the price of a better regularity for the functions~$\frac{ν(|J|)}{|J|}$ and on~$τ$.

\subsection{Equilibria}
\label{subsec-equilibria}

We now define the von Mises--Fisher distribution which provides the general shape of the non-isotropic equilibria of~$Q$.

\begin{definition}
\label{def-vonMises}
The von Mises--Fisher distribution of orientation~$Ω∈𝕊$ and concentration parameter~$κ⩾0$ is given by:
\begin{equation}
\label{eq-vonMises}
M_{κΩ}(ω) =\frac{e^{κ\, ω · Ω}}{∫_𝕊 e^{κ\, υ · Ω}\, \dυ}.
\end{equation}
The order parameter~$c(κ)$ is defined by the relation
\begin{equation}
\label{eq-currentvM}
J_{M_{κΩ}}=c(κ)Ω,
\end{equation}
and has expression:
\begin{equation}
\label{eq-c-kappa}
c(κ)=\frac{∫_0^π \cosθ \, e^{κ \cosθ}\sin^{n-2}θ \, \dθ}{∫_0^π e^{κ \cosθ}\sin^{n-2}θ \, \dθ}.
\end{equation}
\end{definition}
The function~$c:κ∈[0,∞)↦ c(κ)∈[0,1)$ defines an increasing one-to-one correspondence.
The case~$κ=c(κ)=0$ corresponds to the uniform distribution, while when~$κ$ is large (or~$c(κ)$ is close to~$1$), the von Mises--Fisher distribution is close to a Dirac delta mass at the point~$Ω$.

For the sake of simplicity, we will assume the following:
\begin{hypothesis}
\label{hyp-increasing}
The function~$|J|↦ k(|J|)=\frac{ν(|J|)}{τ(|J|)}$ is an increasing function.
We denote by~$j$ its inverse, i.e.
\begin{equation}
\label{eq-inverse-j}
κ=k(|J|)⇔|J|=j(κ).
\end{equation}
\end{hypothesis}
This assumption is not critical.
It would be easy to remove it at the price of an increased technicality.
Additionally, it means that when the alignment of the particles is increased, the relative intensity of the social force compared to the noise is increased as well.
This can be biologically motivated by the existence of some social reinforcement mechanism.
It bears analogies with Diffusion Limited Aggregation~(see~\cite{yates2009inherent}), in which the noise intensity is decreased with larger particle density.
This can also be related with what is called ``extrinsic noise'' in~\cite{aldana2009emergence}, where the noise corresponds to some uncertainty in the particle-particle communication mechanism.
Indeed in this case, the intensity of the noise increases when~$|J|$ decreases.

The equilibria are given by the following proposition:

\begin{proposition}
\label{prop-equilibria}
The following statements are equivalent:
\begin{enumerate}
\item[(i)]~$f∈ C^2(𝕊)$ and~$Q(f)=0$.
\item[(ii)]~$f∈ C^1(𝕊)$ and~$\mathcal D(f)=0$.
\item[(iii)] There exists~$ρ⩾0$ and~$Ω∈𝕊$ such that~$f=ρ M_{κΩ}$, where~$κ⩾0$ satisfies the compatibility equation:
\begin{equation}
\label{eq-compatibility}
j(κ)=ρ c(κ).
\end{equation}
\end{enumerate}
\end{proposition}
\begin{proof}[Sketch of the proof] The proof is identical to that of~\cite{frouvelle2012dynamics}, and we just summarize the main ideas here.
The main ingredient is to observe that~$Q(f)$ (or~$\mathcal D(f)$) is equal to zero if and only if~$f$ is proportional to~$M_{k(|J_f|)\,Ω_f}$. This is quite straightforward for~$\mathcal D$ using~\eqref{def-dissipation}. For~$Q$, it follows from the following expression:
\begin{equation}
\label{eq-Qdiv}
Q(f)=τ(|J_f|)∇_ω·\Big[M_{k(|J_f|)\,Ω_f}∇_ω\Big(\frac{f}{M_{k(|J_f|)\,Ω_f}}\Big)\Big].
\end{equation}
This expression comes from Definition~\ref{def-vonMises}, which gives first
\[∇_ω\Big(\frac{1}{M_{k(|J_f|)\,Ω_f}}\Big)=\frac{-k(|J_f|)\,∇_ω(ω·Ω_f)}{M_{k(|J_f|)\,Ω_f}}=-\frac{k(|J_f|)}{M_{k(|J_f|)\,Ω_f}}\,P_{ω^⊥}Ω_f,\]
and therefore, applying the chain rule to the right-hand side of~\eqref{eq-Qdiv}, we recover the definition of~$Q$ given in~\eqref{def-Q}.
Hence, we obtain
\[∫_𝕊Q(f)\frac{f}{M_{k(|J_f|)\,Ω_f}}\, \d ω=-τ(|J_f|)∫_𝕊\Big|∇_ω\Big(\frac{f}{M_{k(|J_f|)\,Ω_f}}\Big)\Big|^2M_{k(|J_f|)\,Ω_f}\,\d ω.\]
So if~$Q(f)=0$, we get that~$\frac{f}{M_{k(|J_f|)\,Ω_f}}$ is equal to a constant. Conversely if~$f$ is proportional to~$M_{k(|J_f|)\,Ω_f}$, we directly get with~\eqref{eq-Qdiv} that~$Q(f)=0$.

Now if~$f$ is proportional to~$M_{k(|J_f|)\,Ω_f}$, we write~$f=ρM_{κΩ}$, with~$κ=k(|J_f|)$, which corresponds to~$|J_f|=j(κ)$ thanks to~\eqref{eq-inverse-j}. But then by~\eqref{eq-currentvM}, we get that~$|J_f|=ρc(κ)$, which gives the compatibility equation~\eqref{eq-compatibility}. Conversely, if we have~$(iii)$, we also get that~$|J_f|=ρc(κ)=j(κ)$ and so~$κ=k(|J_f|)$, which gives that~$f$ is proportional to~$M_{k(|J_f|)\,Ω_f}$.
\end{proof}

We now make comments on the solutions of the compatibility equation~\eqref{eq-compatibility}.
Let us first remark that the uniform distribution, corresponding to~$κ=0$ is always an equilibrium.
Indeed, we have~$c(0)=j(0)=0$ and~\eqref{eq-compatibility} is satisfied.
However, Proposition~\ref{prop-equilibria} does not provide any information about the number of the non-isotropic equilibria.
The next proposition indicates that two values,~$ρ_*$ and~$ρ_c$, that can be expressed through the function~$k$ only, are important threshold values for the parameter~$ρ$, regarding this number of non-isotropic equilibria.

\begin{proposition}
\label{prop-two-thresholds}
Let~$ρ>0$.
We define
\begin{align}
\label{def-rho-c}
ρ_c=\lim_{\phantom{κ∈(0,}\mathclap{κ→0}\phantom{κ_{max})}} \frac{j(κ)}{c(κ)}&=\lim_{|J|→0}\frac{|J|}{c(k(|J|))}=\lim_{|J|→0}\frac{n|J|}{k(|J|)}, \\
\label{def-rho-star}
ρ_*=\inf_{κ∈(0,κ_{max})}\frac{j(κ)}{c(κ)}&=\inf_{|J|>0}\frac{|J|}{c(k(|J|))},
\end{align}
where~$ρ_c>0$ may be equal to~$+∞$, where~$κ_{max}=\lim_{|J|→∞}k(|J|)$, and where we recall that~$n$ denotes the dimension.
Then we have~$ρ_c⩾ρ_*$, and
\begin{itemize}
\item[(i)] If~$ρ<ρ_*$, the only solution to the compatibility equation is~$κ=0$ and the only equilibrium with total mass~$ρ$ is the uniform distribution~$f=ρ$.
\item[(ii)] If~$ρ>ρ_*$, there exists at least one positive solution~$κ>0$ to the compatibility equation~\eqref{eq-compatibility}.
It corresponds to a family~$\{ρ M_{κΩ},Ω∈𝕊\}$ of non-isotropic equilibria.
\item[(iii)] The number of families of nonisotropic equilibria changes as~$ρ$ crosses the threshold~$ρ_c$ (under regularity and non-degeneracy hypotheses that will be precised in the proof, in a neighborhood of~$ρ_c$, this number is even when~$ρ<ρ_c$ and odd when~$ρ>ρ_c$).
\end{itemize}
\end{proposition}

\begin{proof}
Some comments are necessary about the definitions of~$ρ_c$ and~$ρ_*$.
First note that, under hypotheses~\ref{hyp-prop-chaos} and~\ref{hyp-increasing},~$k$ is defined from~$[0,+∞)$, with values in an interval~$[0,κ_{max})$, where we may have~$κ_{max}=+∞$. 
So~$j$ is an increasing function from~$[0,κ_{max})$ onto~$ℝ_+$, and this gives the equivalence between the two terms of~\eqref{def-rho-star}.
Thanks to hypothesis~\ref{hyp-prop-chaos}, we have
\[k(|J|)=\frac{ν_1}{τ_0}|J|+o(|J|)\text{ as }|J|→0,\]
with~$τ_0=τ(0)$ and~$ν_1=\lim_{|J|→0}\frac{ν(|J|)}{|J|}$, and the last term of~\eqref{def-rho-c} is well defined in~$(0,+∞]$ (we have~$ρ_c=\frac{nτ_0}{ν_1}$ if~$ν_1>0$ and~$ρ_c=+∞$ if~$ν_1=0$). 
The last equality in~\eqref{def-rho-c} comes from the fact that~$c(κ)\sim\frac1n κ$ as~$κ→∞$ (see~\cite{frouvelle2012dynamics} for instance), and the first equality comes from the correspondence~\eqref{eq-inverse-j}.
 
To investigate the positive solutions of equation~\eqref{eq-compatibility}, we recast it into:
\begin{equation}
\label{eq-compat2}
\frac{j(κ)}{c(κ)}=ρ,
\end{equation}
which is valid as long as~$κ≠0$, since~$c$ is an increasing function.
This gives points~$(i)$ and~$(ii)$: there is no solution to~\eqref{eq-compat2} if~$ρ<ρ_*$, and at least one solution if~$ρ>ρ_*$, since~$κ↦\frac{j(κ)}{c(κ)}$ is a continuous function, and its infimum is~$ρ_*$.

Let us precise now the sense of point~$(iii)$.
We fix~$ε>0$, and we suppose that~$\frac{j}{c}$ is differentiable and that for~$ρ∈(ρ_c-ε,ρ_c)∪(ρ_c,ρ_c+ε)$, all the solutions of the compatibility equation satisfy~$(\frac{j}{c})'(κ)≠0$.
Then, the number of solutions of the compatibility equation~\eqref{eq-compat2}, if finite, is odd for~$ρ∈(ρ_c,ρ_c+ε)$ and even for~$ρ∈(ρ_c-ε,ρ_c)$.

Indeed, under these assumptions, by the intermediate value theorem, the sign of~$(\frac{j}{c})'$ must be different for two successive solutions of the compatibility equation~\eqref{eq-compat2}.
Moreover, since~$j$ is unbounded (it maps its interval of definition~$[0,κ_{max})$ onto~$[0,+∞)$), we have
\begin{equation}
\label{lim-j-c-inf}
\lim_{κ→κ_{max}}\frac{j(κ)}{c(κ)}=+∞,
\end{equation}
so the sign of~$(\frac{j}{c})'$ must be positive for the greatest solution of the compatibility equation~\eqref{eq-compat2}.
Finally for the smallest solution, this sign must be the same as the sign of~$ρ-ρ_c$.
\end{proof}

Except from these facts, since~$c$ and~$j$ are both increasing, we have no further direct information about this function~$κ↦ j(κ)/c(κ)$.

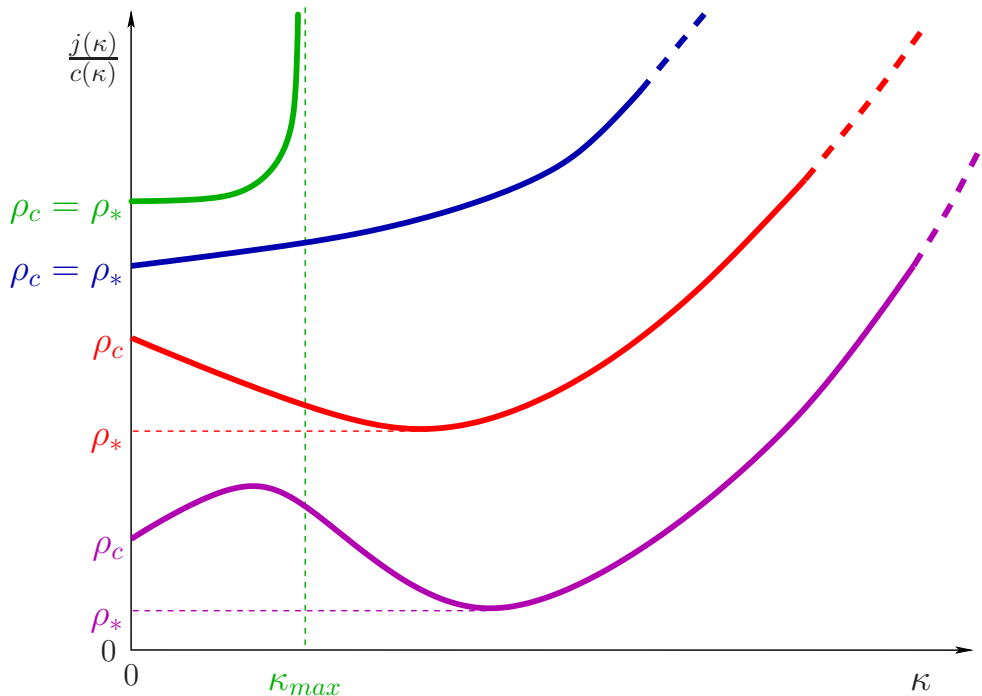
\begin{figure}[h]
\centering
\input{shapes_rho_kappa.pspdftex}
\caption{The green, blue, red and purple curves correspond to various possible profiles for the function~$κ↦\frac{j(κ)}{c(κ)}$.}
\label{fig-shapes}
\end{figure} 

\begin{remark}
The results of Proposition~\ref{prop-two-thresholds} are illustrated by Figure~\ref{fig-shapes}: the number of families of non-isotropic equilibria is given by the cardinality of the level set at~$ρ$ of the function~$κ↦\frac{j(κ)}{c(κ)}$.
We see that depending on the value of~$ρ$, this number can be zero, one, two or even more.
The minimum of this function and its limiting value at~$κ=0$ provide a direct visualization of the thresholds~$ρ_*$ and~$ρ_c$ thanks to~\eqref{def-rho-star}-\eqref{def-rho-c}.

We will see later on that the importance of the threshold~$ρ_c$ is above all due to a loss of stability of the uniform equilibrium, more than a change in the number of families of nonisotropic equilibria.
And we will see that the sign of~$(\frac{j}{c})'(κ)$ which played a role in counting this number in the proof of point~$(iii)$ will actually play a stronger role to determine the stability of the nonisotropic equilibria.
\end{remark}

We now turn to the study of the stability of these equilibria, through the study of the rates of convergence.

\section{Stability and rates of convergence to equilibria}
\label{sec-rates-convergence}

\subsection{Main results}
\label{subsec-main}

We provide an overview of the most important results of this section.
We emphasize that the results of this section are concerned with the space-homogeneous model as reviewed in section~\ref{subsec-space-homogeneous} and~\ref{subsec-equilibria}.

The first result deals with the stability of uniform equilibria.
We prove that the critical density~$ρ_c$ defined previously at~\eqref{def-rho-c} acts as a threshold: 
\begin{itemize}
\item[(i)] if~$ρ < ρ_c$, then the uniform distribution is locally stable and we show that the solution associated to any initial distribution close enough to it converges  with an exponential rate to the uniform distribution.
\item[(ii)] if~$ρ > ρ_c$, then the uniform distribution is unstable, in the sense that no solution (except degenerate cases that we specify) can converge to the uniform distribution.
\end{itemize}

The second result deals with the stability of anisotropic equilibria.
As seen in the previous section, the anisotropic equilibria are given by the von Mises--Fisher distributions~$f = ρ M_{κ Ω}$, defined in~\eqref{eq-vonMises}, of concentration parameter~$κ$ and associated order parameter~$c(κ)$, given by the formula~\eqref{eq-c-kappa}.
Recall that~$j(κ)$ is the inverse function of~$|J| ↦ k(|J|) = \frac{ν(|J|)}{τ(|J|)}$.
We also recall that, for a von Mises--Fisher distribution to be an equilibrium, the compatibility equation~\eqref{eq-compatibility} i.e. the relation~$\frac{j(κ)}{c(κ)} = ρ$ must be satisfied.
Then:
\begin{itemize}
\item[(i)] the von Mises--Fisher equilibrium is stable if~$\big(\frac{j}{c}\big)'>0$ where the prime denotes derivative with respect to~$κ$.
Then, we have an exponential rate of convergence of the solution associated to any initial distribution close enough to one of the von Mises--Fisher distributions, to a (may be different) von Mises--Fisher distribution (with the same~$κ$ but may be different~$Ω$).
\item[(ii)] the von Mises--Fisher equilibrium is unstable if ~$\big(\frac{j}{c}\big)'<0$.
Here, the proof for instability relies on the fact that on any neighborhood of an unstable von Mises--Fisher distribution there exists a distribution which has a smaller free energy than the equilibrium free energy, which only depends on~$κ$ but not on~$Ω$.
The instability follows from the time decay of the free energy.
\end{itemize}

The main tool to prove convergence of the solution to a steady state is LaSalle’s principle.
We recall it in the next section and only sketch its proof.
Indeed, the proof follows exactly the lines of~\cite{frouvelle2012dynamics}.
Then, in section~\ref{subsec-uniform}, we consider stability and rates of convergence near uniform equilibria.
Finally, in section~\ref{subsec-anisotropic}, we investigate the same problem for non-isotropic equilibria.

\subsection{LaSalle’s principle}
\label{subsec-Lasalle}

By the conservation relation~\eqref{dissipation-free-energy}, we know that the free energy~$\mathcal F$ is decreasing in time (and bounded from below since~$|J|$ is bounded).
LaSalle’s principle states that the limiting value of~$\mathcal F$ corresponds to an~$ω$-limit set of equilibria:

\begin{proposition}
\label{prop-lasalle}LaSalle’s invariance principle: let~$f_0$ be a positive measure on the sphere~$𝕊$, with mass~$ρ$.
We denote by~$\mathcal F_∞$ the limit of~$\mathcal F(f(t))$ as~$t→∞$, where~$f$ is the solution to the mean-field homogeneous equation~\eqref{KFP-homogeneous} with initial condition~$f_0$.
Then
\begin{itemize}
\item[(i)] the set~$\mathcal E_∞=\{f ∈ C^∞(𝕊)\text{ with mass } ρ\text{ and s.t.
} \mathcal D(f)=0 \text{ and } \mathcal F(f)=\mathcal F_∞ \}$ is not empty.
\item[(ii)]~$f(t)$ converges in any~$H^s$ norm to this set of equilibria (in the following sense):
\[\lim_{t→∞}d_{H^s}(f,\mathcal E_∞)=0,\text{ where }\, d_{H^s}(f,\mathcal E_∞)=\inf_{g∈\mathcal E_∞}∥f(t)-g∥_{H^s}.\]
\end{itemize}
\end{proposition}

This result has been proved in~\cite{frouvelle2012dynamics}.
Since the different types of equilibria are known, we can refine this principle to adapt it to our problem:

\begin{proposition}
\label{prop-lasalle-refined}

Let~$f_0$ be a positive measure on the sphere~$𝕊$, with mass~$ρ$.
If no open interval is included in the set~$\{κ,ρ c(κ)=j(κ)\}$, then there exists a solution~$κ_∞$ to the compatibility solution~\eqref{eq-compatibility} such that we have:
\begin{gather}
\label{eq-limJ}
\lim_{t→∞} |J_f(t)|=ρ c(κ_∞)
\intertext{and}
\label{eq-limf}
∀ s∈ℝ, \lim_{t→∞}\,∥f(t)-ρ M_{κ_∞Ω_f(t)}∥_{H^s}=0.
\end{gather}
\end{proposition}
This proposition helps us to characterize the~$ω$-limit set by studying the single compatibility equation~\eqref{eq-compatibility}.
Indeed, when~$κ_∞=0$ is the unique solution, Proposition~\ref{prop-lasalle-refined} implies that~$f$ converges to the uniform distribution.
Otherwise, two cases are possible: either~$κ_∞=0$, and~$f$ converges to the uniform distribution, or~$κ_∞>0$, and the~$ω$-limit set consists in the family of von Mises--Fisher equilibria~$\{ρ M_{κ_∞Ω},Ω∈𝕊\}$, but the asymptotic behavior of~$Ω_{f(t)}$ is unknown.

\begin{proof}
We first recall some useful formulas regarding functions on the sphere. Any function~$g$ in~$H^s$ can be decomposed~$g=\sum_ℓ g_ℓ$ where~$g_ℓ$ is a spherical harmonic of degree~$ℓ$ (an eigenvector of~$-Δ_ω$ for the eigenvalue~$ℓ(ℓ+n-2)$, which has the form of a homogeneous polynomial of degree~$ℓ$), and this decomposition is orthogonal in~$H^s$. The spherical harmonics of degree~$1$ are the functions~$ω↦ ω· A$ for~$A∈ℝ^n$, and we have
\begin{equation}
\label{eq-omega-omega}
∫_𝕊ω⊗ω\,\dω=\tfrac1n\mathrm I_n, \quad\text{i.e.}\quad∀A∈ℝ^n,∫_𝕊(A·ω)\, ω\,\d ω=\tfrac1nA.
\end{equation}
which gives that the first mode~$g_1$ of~$g$ is given by the function~$ω↦ n\,ω· J_g$, where the first moment~$J_g$ is defined in~\eqref{def-J}. We refer to the appendix of~\cite{frouvelle2012dynamics} for more details on these spherical harmonics. Another useful formula is
\begin{equation}
\label{eq-divA}
∫_𝕊ω\,∇_ω·\mathcal A(ω)\, \d ω = -∫\mathcal A(ω) \d ω,
\end{equation}
where~$\mathcal A$ is any tangent vector field (satisfying~$\mathcal A(ω)·ω=0$).

Since the decomposition in spherical harmonics is orthogonal in~$H^s$, we have a lower bound on the norm of~$f(t)-ρ M_{κΩ}$ (for~$κ⩾0$ and~$Ω∈𝕊$) with the norm of its first mode: 
\begin{align*}
∥f(t)-ρ M_{κΩ}∥^2_{H^s}&⩾∫_𝕊n\,ω·(J_f-J_{ρ M_{κΩ}})(-Δ_ω)^s[n\,ω·(J_f-J_{ρ M_{κΩ}})]\,\d ω\\
&⩾(n-1)^s∫_𝕊n^2\,[ω·(J_f-J_{ρ M_{κΩ}})]^2\,\d ω,
\end{align*}
and using~\eqref{eq-omega-omega}, we get
\begin{align}
∥f(t)-ρ M_{κΩ}∥^2_{H^s}\label{est-1stmoment}&⩾ n(n-1)^s|J_f-ρ c(κ)Ω|^2\\
\label{est-1stmomentnorm}&⩾ n(n-1)^s\bigl||J_f|-ρ c(κ)\bigr|^2.
\end{align}
Since~$\mathcal E_∞$ consists in functions of the form~$ρ M_{κΩ}$ with~$Ω∈𝕊$ and~$κ$ a solution of~\eqref{eq-compatibility} (and such that~$\mathcal{F}(ρ M_{κΩ})=\mathcal F_∞$), if we define~$S_∞=\{ρ c(κ),κ\text{ s.t.
}ρ c(κ)=j(κ)\}$, we get that the distance~$d_{H^s}(f,\mathcal E_∞)$ is greater than~$√ n(n-1)^{s/2}d(|J_f|,S_∞)$, where the notation~$d(|J_f|,\mathcal J_∞)$ denotes the usual distance in~$ℝ$ between~$|J_f|$ and the set~$S_∞$.
By LaSalle’s principle, we then have~$\lim_{t→∞}d(|J_f|,S_∞)=0$.
Since~$|J_f|$ is a continuous function, bounded in time, its limit points consist in a closed interval, which is included in~$S_∞$.
Obviously, if no open interval is included in the set of solutions to the compatibility equation~\eqref{eq-compatibility}, then no open interval is included in~$S_∞$, and the limit points of~$|J_f|$ are reduced to a single point~$ρ c(κ_∞)$.
Since~$|J_f|$ is bounded, this proves~\eqref{eq-limJ}.

Let us now suppose that~\eqref{eq-limf} does not hold.
We can find an increasing and unbounded sequence~$t_n$ such that~$∥f(t_n)-ρ M_{κ_∞Ω_f(t_n)}∥_{H^s}⩾ε$.
By LaSalle’s principle, we can find~$g_n∈\mathcal E_∞$ such that~${∥f(t_n)-g_n∥→0}$ when~$n→∞$.
Since~$g_n$ is of the form~$ρ M_{κ_nΩ_n}$, we then have by the estimation~\eqref{est-1stmomentnorm} that~$\bigl||J_{f(t_n)}|-ρ c(κ_n)\bigr|→0$, and so~$c(κ_n)→c(κ_∞)$, consequently~${κ_n→κ_∞}$.
If~$κ_∞≠0$, then we also get by~\eqref{est-1stmoment} that~$|Ω_{f(t_n)}-Ω_n|→0$, so in any case, that gives that~${∥g_n-ρ M_{κ_∞Ω_f(t_n)}∥_{H^s}→0}$ (it is equal to~$∥ρ M_{κ_nΩ_n}-ρ M_{κ_∞Ω_f(t_n)}∥_{H^s}$).
But then we obtain the convergence of~$∥f(t_n)-ρ M_{κ_∞Ω_f(t_n)}∥_{H^s}$ to~$0$, which is a contradiction.
\end{proof}

From this proposition, the asymptotic behavior of a solution can be improved in two directions.
First, as pointed above, the behavior of~$Ω_{f(t)}$ is unknown and we are left to comparing the solution to a von Mises--Fisher distribution with asymptotic concentration parameter~$κ_∞$ but local mean direction~$Ω_f(t)$, varying in time.
If we are able to prove that~$Ω_f→Ω_∞∈𝕊$, then~$f$ would converge to a fixed non-isotropic steady-state~$ρ M_{κ_∞Ω_∞}$.
The second improvement comes from the fact that Proposition~\ref{prop-lasalle-refined} does not give information about quantitative rates of convergence of~$|J_f|$ to~$ρ c(κ_∞)$, and of~$∥f(t)-ρ M_{κ_∞Ω_f(t)}∥_{H^s}$ to~$0$, as~$t→∞$.

So we now turn to the study of the behavior of the difference between the solution~$f$ and a target equilibrium~$ρ M_{κ_∞Ω_f(t)}$.
There are two tools we will use.
First, a simple decomposition in spherical harmonics will give us an estimation in~$H^s$ norm near the uniform distribution.
Then we will expand the free energy~$\mathcal F$ and its dissipation~$\mathcal D$ around the nonisotropic target equilibrium~$M_{κ_∞Ω_f(t)}$.
In case of stability, we will see that it gives us control on the displacement of~$Ω_f(t)$, allowing to get actual convergence to a given steady-state.
We split the stability analysis into two cases: stability about uniform equilibrium, and stability about anisotropic equilibrium.

\subsection{Local analysis about the uniform equilibrium}
\label{subsec-uniform}

We first state the following proposition, about the instability of the uniform equilibrium distribution for~$ρ$ above the critical threshold~$ρ_c$.
\begin{proposition}
\label{prop-unstability-uniform}
Let~$f$ be a solution of~\eqref{KFP-homogeneous}, with initial mass~$ρ$.
If~$ρ>ρ_c$, and if~$J_{f_0}≠0$, then we cannot have~$κ_∞=0$ in Proposition~\ref{prop-lasalle-refined}.
\end{proposition}
This proposition tells that the uniform equilibrium is unstable, in the sense that no solution of initial mass~$ρ$ and with a nonzero initial first moment~$J_{f_0}$ can converge to the uniform distribution.

\begin{proof}
We first derive an estimation for the differential equation satisfied by~$J_f$ which will also be useful for the next proposition.

We expand~$f$ under the form~$f=ρ+n\, ω· J_f +g_2$~($g_2$ consists only in spherical harmonics modes of degree~$2$ and more), and we get~$∫_𝕊g_2\,\dω=0$ and~$∫_𝕊g_2\,ω\dω=0$. Let us first expand the alignment term~$∇_ω·(P_{ω^⊥}Ω_f f)$ of the operator~$Q$ defined in~\eqref{def-Q}, using the fact that~${∇_ω·(P_{ω^⊥}Ω_f)=Δ_ω(Ω_f·ω)=-(n-1)\,Ω_f·ω}$. We get
\begin{equation}
\label{expansion-alignment}
∇_ω·(P_{ω^⊥}Ω_ff)=-ρ(n-1)\,Ω_f·ω - n^2 |J_f| \big[(Ω_f·ω)^2-\tfrac1n] + ∇_ω·(P_{ω^⊥}Ω_fg_2),
\end{equation}
and we remark that the term in brackets is a spherical harmonic of degree~$2$, associated to the eigenvalue~$2n$ of~$-Δ_ω$.
Multiplying~\eqref{KFP-homogeneous} by~$ω$ and integrating on the sphere, we obtain, using~\eqref{expansion-alignment},~\eqref{eq-divA} and~\eqref{eq-omega-omega} (and observing that the terms~$∫_𝕊ω\,\d ω$ and~$∫_𝕊(ω·Ω_f)^2\,ω\,\d ω$ are both zero):
\begin{align}
\label{ODE-J}
\frac{\d}{\d t}J_f&=\frac{n-1}{n}\,ρ\,ν(|J_f|)\,Ω_f+ν(|J_f|)∫_𝕊P_{ω^⊥}Ω_f\, f\,\d ω-(n-1)τ(|J_f|)\,J_f\\
\nonumber&=-(n-1)τ(|J_f|)\bigl[1-\frac{ρ\,k(|J_f|)}{n|J_f|}\bigr]J_f+ν(|J_f|)∫_𝕊P_{ω^⊥}Ω_f g_2\,\d ω.
\end{align}
Using~\eqref{def-rho-c} and hypothesis~\ref{hyp-prop-chaos}, we can write:
\begin{equation}
\label{ODE-J-reduced}
\frac{\d}{\d t}J_f=-(n-1)τ_0\bigl(1-\frac{ρ}{ρ_c}\bigr)J_f +R(|J_f|)J_f +\frac{ν(|J_f|)}{|J_f|}\Big(∫_𝕊P_{ω^⊥} g_2\,\d ω\Big)J_f,
\end{equation}
with the remainder estimation, with an appropriate constant~$C>0$.
\begin{equation}
\label{estimate-R}
R(|J|) ⩽ C|J|.
\end{equation}
Equation~\eqref{ODE-J-reduced} can be seen as~$\frac{\d}{\d t}J_f=M(t)J_f$, the matrix~$M$ being a continuous function in time.
Therefore we have uniqueness of a solution of such an equation (even backwards in time), and if~$J_{f_0}≠0$, then we cannot have~$J_{f(t)}=0$ for~$t>0$.
Now if we suppose that~$∥f-ρ∥_{H^s}→0$, then we have~$|J_f|→0$ and~$∫_𝕊P_{ω^⊥} g_2\,\d ω→0$ (as a matrix).
So, for any~$ε>0$, and for~$t$ sufficiently large, taking the dot product of~\eqref{ODE-J-reduced} with~$J_f$, we get that
\[\frac12\frac{\d}{\d t}|J_f|^2⩾\big[(n-1)τ_0\big(\frac{ρ}{ρ_c}-1\big)-ε\big]|J_f|^2,\]
which, for~$ε$ sufficiently small, leads to an exponential growth of~$|J_f|$, and this is a contradiction.
\end{proof}

We now turn to the study of the stability of the uniform distribution when~$ρ$ is below the critical threshold~$ρ_c$.
We have the
\begin{proposition}
\label{prop-stability-uniform}
Suppose that~$ρ<ρ_c$.
We define
\[λ=(n-1)τ_0(1-\frac{ρ}{ρ_c})>0.\]
Let~$f_0$ be an initial condition with mass~$ρ$, and~$f$ the corresponding solution of~\eqref{KFP-homogeneous}.
There exists~$δ>0$ independent of~$f_0$ such that if~$∥f_0-ρ∥_{H^s}<δ$, then for all~$t⩾0$
\[∥f(t)-ρ∥_{H^s}⩽\frac{∥f_0-ρ∥_{H^s}}{1-\frac{1}{δ}∥f_0-ρ∥_{H^s}}e^{-λ t}.\]
\end{proposition}

\begin{proof}
We multiply~\eqref{KFP-homogeneous} by~$(-Δ_ω)^s g_2$ and integrate by parts on the sphere. Using~\eqref{expansion-alignment},~\eqref{eq-divA}, and the fact that~$g_2$ is orthogonal to the spherical harmonics of degree~$1$, we get
\begin{align*}
\frac12\frac{\d}{\d t}∥g_2∥^2_{H^s}=&-ν(|J_f|)\,n^2|J_f|∫_𝕊 \big[(Ω_f·ω)^2-\tfrac1n\big] (-Δ_ω)^s g_2\,\d ω \\
&+ν(|J_f|)∫_𝕊[Ω_f·∇_ω(-Δ_ω)^s g_2] g_2\,\d ω\\
&-τ(|J_f|)∫_𝕊g_2(-Δ_ω)^{s+1} g_2\,\d ω.
\end{align*}
Using the fact that the second eigenvalue of~$-Δ_ω$ is~$2n$, we get
\begin{equation}
\label{est-g2-Hs}
\begin{split}
\frac12\frac{\d}{\d t}∥g_2∥^2_{H^s}=-τ(|J_f|)∥g_2∥^2_{H^{s+1}}&-n^2ν(|J_f|)|J_f|∫_𝕊(2n)^s (Ω_f·ω)^2 g_2 \d ω\\
&+ν(|J_f|)∫_𝕊[Ω_f·∇_ω(-Δ_ω)^s g_2] g_2\,\d ω.
\end{split}
\end{equation}
We can directly compute the~$H^s$ norm of the first mode of~$f-ρ$ as in~\eqref{est-1stmoment}, and we get by orthogonal decomposition that
\begin{equation}
\label{norm-f-rho}
∥f-ρ∥^2_{H^s}=n(n-1)^s|J_f|^2+∥g_2∥^2_{H^s}.
\end{equation}
Taking the dot product of~\eqref{ODE-J-reduced} with~$n(n-1)^s J_f$ and summing with~\eqref{est-g2-Hs}, we get the time derivative of~$∥f-ρ∥^2_{H^s}$:
\begin{equation}
\label{eq-deriv-f-rho}
\begin{split}
\frac12\frac{\d}{\d t}∥f-ρ∥^2_{H^s}= {} &-n(n-1)^{s+1}τ_0\bigl(1-\frac{ρ}{ρ_c}\bigr)|J_f|^2 -
τ(|J_f|)∥g_2∥^2_{H^{s+1}} \\
&+ n(n-1)^s R(|J_f|)|J_f|^2 +ν(|J_f|)∫_𝕊g_2 \,Ω_f·∇(-Δ_ω)^s g_2 \,\d ω\\
&+ [n(n-1)^s-n^2(2n)^s]ν(|J_f|)|J_f|∫_𝕊Ω_f· P_{ω^⊥}Ω_f g_2 \d ω.
\end{split}
\end{equation}
Using the Poincaré inequality, and again, that the second eigenvalue of~$-Δ_ω$ is~$2n$, we get that
\begin{equation}
\label{poincare-g2}
∥g_2∥^2_{H^{s+1}}⩾2n ∥g_2∥^2_{H^s}⩾(n-1)(1-\frac{ρ}{ρ_c})∥g_2∥^2_{H^s}.
\end{equation}
We combine the first two terms of the right-hand side of~\eqref{eq-deriv-f-rho} with~\eqref{poincare-g2} to get an estimation of~$\frac12\frac{\d}{\d t}∥f-ρ∥^2_{H^s}$ in terms of a constant times~$∥f-ρ∥^2_{H^s}$ and a remainder that we expect to be of smaller order:
\begin{equation}
\label{est-Hs2}
\frac12\frac{\d}{\d t}∥f-ρ∥^2_{H^s}⩽-(n-1)τ_0\bigl(1-\frac{ρ}{ρ_c}\bigr)∥f-ρ∥^2_{H^s}+\mathcal R_s,
\end{equation}
where
\begin{equation}
\label{def-Rs}
\begin{split}
\mathcal R_s= {} &n(n-1)^s R(|J_f|)|J_f|^2 +ν(|J_f|)∫_𝕊g_2 \,Ω_f·∇(-Δ_ω)^s g_2 \,\d ω\\
&+ [n(n-1)^s-(2n)^s]ν(|J_f|)|J_f|∫_𝕊(Ω_f·ω)^2 g_2 \d ω\\
&+ [τ_0-τ(|J_f|)](n-1)(1-\frac{ρ}{ρ_c})∥g_2∥^2_{H^s}.
\end{split}
\end{equation}
Using Lemma~2.1 of~\cite{frouvelle2012dynamics}, there exists a constant~$C_1$ (independent of~$g_2$) such that
\[\big|∫_𝕊g_2 \,Ω_f·∇(-Δ_ω)^s g_2 \,\d ω\big|⩽ C_1∥g_2∥^2_{H^s}.\]
Together with the estimates~$R$,~$ν$ and~$τ$ given by~\eqref{estimate-R} and hypothesis~\eqref{hyp-prop-chaos}, and the fact that the function~$ω↦(Ω_f·ω)^2$ belongs to~$H^{-s}$, we can estimate every term of~\eqref{def-Rs}, giving existence of constants~$C_2, C_3$, such that
\[\mathcal R_s⩽ C_2 \big[|J_f|^3+|J_f|^2∥g_2∥_{H^s} +|J_f|∥g_2∥^2_{H^s}\big]⩽C_3 ∥f-ρ∥^3_{H^s},\]
the last inequality coming from equation~\eqref{norm-f-rho}.
Solving the differential inequality~$y'⩽-λy+C_3y^2$ which corresponds to~\eqref{est-Hs2} with~$y=∥f-ρ∥_{H^s}$, we get that
\[\frac{y}{λ-C_3y}⩽\frac{y_0}{λ-C_3y_0}e^{-λt},\]
provided that~$y<δ=\frac{λ}{C_3}$. If~$y_0<δ$, the differential inequality ensures that~$y$ is decreasing and the condition~$y<δ$ is always satisfied. In this case, we get
\[y⩽\frac{y}{1-\frac{y}{δ}}⩽\frac{y_0}{1-\frac{y_0}{δ}}e^{-λt},\] 
which ends the proof.
\end{proof}

\begin{remark}
\label{remark-special-entropy}
We can indeed remove this condition of closeness of~$f_0$ to~$ρ$ by using the method of~\cite{frouvelle2012dynamics} in the case where~$ρ<\hat{ρ}$, where the critical threshold~$\hat{ρ}$ is defined as follows:~$\hat{ρ}=\inf_{|J|} \frac{n|J|}{k(|J|)}$ (since we have~$c(κ)⩽\frac{κ}{n}$ for all~$κ$, compared to the definition~\eqref{def-rho-c}-\eqref{def-rho-star} of~$ρ_c$ and~$ρ_*$, we see that~$\hat{ρ}⩽ρ_*⩽ρ_c$, with a possible equality if for example~$|J|↦\frac{k(|J|)}{|J|}$ is nonincreasing).

We can use the special cancellation presented in~\cite{frouvelle2012dynamics}:
\[∫∇ g\,\widetilde{Δ_{n-1}}g=0,\]
where~$\widetilde{Δ_{n-1}}$ is the so-called conformal Laplacian on~$𝕊$, a linear operator defined, for any spherical harmonic~$Y_ℓ$ of degree~$ℓ$, by
\[\widetilde{Δ}_{n-1}\, Y_ℓ = ℓ(ℓ+1)…(ℓ+n-2)Y_ℓ.\]
Multiplying~\eqref{KFP-homogeneous} by~$\widetilde{Δ}_{n-1}^{-1}(f-ρ)$ and integrating by parts, we get the following conservation relation:
\begin{equation}
\label{new-entropy}
\tfrac12\tfrac{\d}{\d t} \big(\tfrac{n}{(n-1)!}|J_f|^2+∥g_2∥^2_{\widetilde{H}^{-\frac{n-1}2}}\big)=
-τ(|J_f|)\big[\tfrac{n}{(n-2)!}\big(1-\tfrac{ρ k(|J_f|)}{n|J_f|}\big)\,|J_f|^2
+∥g_2∥^2_{\widetilde{H}^{-\frac{n-3}2}}\big],
\end{equation}
where the norms~$∥·∥_{\widetilde{H}^{-\frac{n-1}2}}$ and~$∥·∥_{\widetilde{H}^{-\frac{n-3}2}}$ are modified Sobolev norms respectively equivalent to~$∥·∥_{{H}^{-\frac{n-1}2}}$ and~$∥·∥_{{H}^{-\frac{n-3}2}}$.

So if~$ρ<\hat{ρ}$, equation~\eqref{new-entropy} can be viewed as a new entropy dissipation for the system, and we have global exponential convergence with rate~$\hat{λ}=(n-1)τ_{min}(1-\frac{ρ}{\hat{ρ}})$, where~$τ_{min}=\min_{|J|⩽ρ}τ(|J|)$:
\begin{equation}
\label{eq-global-decay}
∥f-ρ∥_{\widetilde{H}^{-\frac{n-1}2}}⩽∥f_0-ρ∥_{\widetilde{H}^{-\frac{n-1}2}}e^{-\hat{λ}t},
\end{equation}
valid for any initial condition~$f_0∈ H^{-\frac{n-1}2}(𝕊)$ with initial mass~$ρ$, whatever its distance to~$ρ$.

Let us also remark that if~$\hat{ρ}⩽ρ<ρ_*$, where~$ρ_*$ is defined in~\eqref{def-rho-star}, any solution with initial mass~$ρ$ converges to the uniform distribution (the unique equilibrium), but we do not have an a priori global rate.
We can just locally rely on Proposition~\ref{prop-stability-uniform}.
\end{remark}

\subsection{Local analysis about the anisotropic equilibria}
\label{subsec-anisotropic}

\label{subsec-local-analysis-anisotropic}

We fix~$κ>0$ and let~$ρ$ be such that~$κ$ is a solution of the compatibility equation~\eqref{eq-compatibility}, i.e.~$ρ=\frac{j(κ)}{c(κ)}$.
In this subsection, to make notations simpler, we will not write the dependence on~$κ$ when not necessary.

We make an additional hypothesis on the function~$k$:
\begin{hypothesis}
\label{hyp-differentiable}
The function~$|J|↦ k(|J|)$ is differentiable, with a derivative~$k'$ which is itself Lipschitz.
\end{hypothesis}

We can then state a first result about the stability or instability of a non-isotropic solution~$ρ M_{κΩ}$, depending on the sign of~$(\frac{j}{c})'$.
In summary, if the function~$κ↦\frac{j}{c}$ is (non-degenerately) increasing then the corresponding equilibria are stable, while if it is (non-degenerately) decreasing the equilibria are unstable.
For example, for the different cases depicted in Figure~\ref{fig-shapes}, it is then straightforward to determine the stability of the different equilibria.

\begin{proposition}
\label{prop-unstability-stability-anisotropic}
Let~$κ>0$ and~$ρ=\frac{j(κ)}{c(κ)}$.
We denote by~$\mathcal F_κ$ the value of~$\mathcal F(ρ M_{κΩ})$ (independent of~$Ω∈𝕊$).

\begin{itemize}
\item[(i)] Suppose~$(\frac{j}{c})'(κ)<0$.
Then any equilibrium of the form~$ρ M_{κΩ}$ is unstable, in the following sense:
in any neighborhood of~$ρ M_{κΩ}$, there exists an initial condition~$f_0$ such that~$\mathcal F(f_0)<\mathcal F_κ$.
Consequently, in that case, we cannot have~$κ_∞=κ$ in Proposition~\ref{prop-lasalle-refined}.
\item[(ii)] Suppose~$(\frac{j}{c})'(κ)>0$.
Then the family of equilibria~$\{ρ M_{κΩ},Ω∈𝕊\}$ is stable, in the following sense: for all~$K>0$ and~$s>\frac{n-1}2$, there exists~$δ>0$ and~$C$ such that for all~$f_0$ with mass~$ρ$ and with~$∥f_0∥_{H^s}⩽ K$, if~$∥f_0-ρ M_{κΩ}∥_{L^2}⩽δ$ for some~$Ω∈𝕊$, then for all~$t⩾0$, we have
\begin{gather*}
\mathcal F(f)⩾\mathcal F_κ,\\
∥f-ρ M_{κΩ_f}∥_{L^2}⩽ C∥f_0-ρ M_{κΩ_{f_0}}∥_{L^2}.
\end{gather*}
\end{itemize}
\end{proposition}

\begin{proof}
We first make some preliminary computation which will also be useful for the following theorem.
We expand the solution~$f$ of~\eqref{KFP-homogeneous} (with initial mass~$ρ$) around a ``moving'' equilibrium~$ρ M_{κΩ_f(t)}$.
Let us use the same notations as in~\cite{frouvelle2012dynamics}: we write~$⟨ g⟩_M$ for~$∫_𝕊g(ω)M_{κΩ_f}\d ω$, we denote~$ω\cdotΩ_f$ by~$\cosθ$ and we write:
\[f=M_{κΩ_f}(ρ+g_1)=M_{κΩ_f}(ρ+α(\cosθ-c)+g_2),\]
where
\begin{equation}
\label{def-alpha}
α=\frac{|J_f|-ρ c}{⟨(\cosθ-c)^2⟩_M}.
\end{equation}
We have~$⟨ g_1⟩_M=⟨ g_2⟩_M=0$, and definition of~$α$ ensures that~$⟨ω\,g_2⟩_M=0$.
The derivative of~$c$ with respect to~$κ$ is given by
\begin{equation}
\label{eq-c-prime}
c'(κ)=⟨\cos^2θ⟩_M-⟨\cosθ⟩^2_M=⟨(\cosθ-c)^2⟩_M.
\end{equation}
We are now ready to estimate the difference between the free energy of~$f$ and of the equilibrium~$ρ M_{κΩ_f}$.
We have a first expansion, for the potential term of the free energy~\eqref{def-k-phi}:
\begin{align*}
Φ(|J_f|)&=Φ(ρ c)+k(ρ c)α⟨(\cosθ-c)^2⟩_M+k'(ρ c)\frac{α^2}2⟨(\cosθ-c)^2⟩^2_M+O(α^3)\\
&=Φ(j)+κ c'(κ)α+\frac{(c'(κ))^2}{j'(κ)}\frac{α^2}2 + O(α^3).
\end{align*}
Now, we will use the following estimation, valid for any~$x∈(-1,+∞)$:
\begin{equation}
\label{estimate-xlnx}
|(1+x)\ln(1+x) - x -\tfrac12x^2|⩽\tfrac12|x|^3.
\end{equation}
To get this estimation, we note that~$h_2(x)=(1+x)\ln(1+x) - x -\tfrac12x^2$ is such that~$h_2$,~$h_2'$ and~$h_2''$ cancel at~$x=0$, and that~$h_2^{(3)}(x)=\frac{-1}{(1+x)^2}∈(-1,0)$ for~$x>0$.
Therefore Taylor’s formula gives~$-\frac16x^3<h_2(x)<0$ for~$x>0$. 
For~$x<0$ we have by the same argument~$h_2(x)>0$, but Taylor’s formula is not sufficient to have a uniform estimate on~$(-1,0)$.
We introduce~$h_3=h_2+\frac12x^3$.
By induction from~$i=3$ to~$i=1$ we have that~$h^{(i)}_3$ as a unique root~$γ_i$ in~$(-1,0)$, with~$γ_3>γ_2>γ_1$.
Since~$h_3'(x)→-∞$ as~$x→-1$,~$h_3$ is decreasing on~$(-1,γ_1)$ and increasing on~$(γ_1,0)$, but we have~$h_3(-1)=h_3(0)=0$ so~$h_3<0$ on~$(-1,0)$, which ends the derivation of~\eqref{estimate-xlnx}.

Using~\eqref{estimate-xlnx} with~$x=\frac{g_1}{ρ}$, we have that
\begin{align*}
∫_𝕊f\ln f\d ω &=⟨(ρ+g_1)[\ln(1+\frac{g_1}{ρ})+\ln(ρ M_{κΩ_f})]⟩_M\\
&=⟨ρ\ln(ρ M_{κΩ_f})⟩_M + ⟨κ\cosθ g_1⟩_M + \frac1{2ρ}⟨ g_1^2⟩_M+ O(⟨|g_1|^3⟩_M)\\
&=∫_𝕊ρ M_{κΩ_f}\ln(ρ M_{κΩ_f})\d ω + ακ c' + \frac1{2ρ}[α^2c'+⟨ g_2^2⟩_M] + O(⟨|g_1|^3⟩_M).
\end{align*}
Finally we get
\begin{align}
\nonumber\mathcal F(f)-\mathcal F(ρ M_{κΩ_f})&= \frac{α^2}2\,c'\big(\frac1{ρ}-\frac{c'}{j'}\big) + \frac1{2ρ}⟨ g_2^2⟩_M + O(⟨|g_1|^3⟩_M)\\
\label{estimate-free-energy}&= \frac1{2ρ}[\frac{c' c}{j'}\big(\frac{j}{c}\big)' α^2 + ⟨ g_2^2⟩_M] + O(⟨|g_1|^3⟩_M).
\end{align}
Now, we prove~(i).
We simply take~$α$ sufficiently small and~$g_2=0$, and the estimation~\eqref{estimate-free-energy} gives the result.
Indeed, since~$c$ and~$j$ are increasing functions of~$κ$, the leading order coefficient in~\eqref{estimate-free-energy}, which is~$\frac1{2ρ}\frac{c' c}{j'}\big(\frac{j}{c}\big)'$, is negative by the assumption.

We now turn to point~(ii).
We will use the following simple lemma, the proof of which is left to the reader.
\begin{lemma}
\label{lemma-x-y}
Suppose~$x(t)⩾0$ is a continuous function and~$y(t)$ is a decreasing function satisfying
\[|x(t)-y(t)|⩽ C x(t)^{1+ε}, ∀ t⩾0,\]
for some positive constants~$C$ and~$ε$.
Then there exist~$δ>0$ and~$\widetilde C$ such that, if~$x(0)⩽δ$, then
\[y(t)⩾0, \quad \text{and} \quad |x(t)-y(t)|⩽\widetilde C y(t)^{1+ε}, ∀ t⩾0.\]
\end{lemma}

By Sobolev embedding, Sobolev interpolation, and the uniform bounds of Theorem~\ref{theorem-existence-uniqueness}, we have
\begin{equation}
\label{bound-infty-L2}
∥g_1∥_∞⩽C∥g_1∥_{H^{\frac{n-1}2}}⩽C ∥g_1∥^{1-ε}_{H^s}∥g_1∥^{ε}_{L^2}⩽ C_1(⟨ g_1^2⟩_{M})^ε,
\end{equation}
for some~$ε>0$, and where the constant~$C_1$ depends only on~$K$ (the constant in the statement of the proposition, which is an upper bound for~$∥f_0∥_{H^s}$),~$s$,~$κ$ and the coefficients~$ν$ and~$τ$ of the model. We will denote by~$C_i$ such a constant in the following of the proof.

We define~$x(t)=\frac1{2ρ}[\frac{cc'}{j'}(\frac{j}{c})'\,α^2 + ⟨ g_2^2⟩_M]$ and~$y(t)=\mathcal F(f)-\mathcal F_κ$.
Together with the estimate~\eqref{estimate-free-energy}, since~$⟨ g_1^2⟩_M=c'α^2+⟨ g_2^2⟩_M$, and~$(\frac{j}{c})'>0$, we can apply Lemma~\ref{lemma-x-y}.
It gives us that if~$⟨ g_1^2⟩_M$ is initially sufficiently small, then~$\mathcal F(f)⩾\mathcal F_κ$ and we have
\[x(t)=\frac1{2ρ}[\frac{cc'}{j'}(\frac{j}{c})'\,α^2 + ⟨ g_2^2⟩_M] = \mathcal F(f)-\mathcal F_κ + O((\mathcal F(f)-\mathcal F_κ)^{1+ε}).\]
Now, using the fact that~$x(t)$,~$⟨ g_1^2⟩_M$ and~$∥f-ρM_{κΩ_f}∥_{L^2}^2$ are equivalent quantities (up to a multiplicative constant) and the estimate~\eqref{estimate-free-energy}, we get that
\begin{equation}
\label{est-L2-L2}
∥f-ρM_{κΩ_f}∥_{L^2}^2⩽ C_2x(t)⩽ C_3(\mathcal F(f)-\mathcal F_κ).
\end{equation}
Using the fact that~$\mathcal F(f)-\mathcal F_κ$ is decreasing in time, and the same equivalent quantities, we finally get
\[∥f-ρM_{κΩ_f}∥_{L^2}^2⩽ C_3(\mathcal F(f_0)-\mathcal F_κ)⩽ C_4∥f_0-ρM_{κΩ_{f_0}}∥_{L^2}^2.\]
This completes the proof, with the simple remark that, as in the proof of proposition~\ref{prop-lasalle-refined}, we can control~$|Ω-Ω_{f_0}|$ by~$∥f_0-ρ M_{κΩ}∥_{L^2}$ (using the formula~\eqref{est-1stmoment}).
Then we can also control the quantities~$∥ρ(M_{κΩ}-M_{κΩ_{f_0}})∥_{L^2}$ and~$∥f_0-ρ M_{κΩ_{f_0}}∥_{L^2}$, and finally the initial value of~$⟨ g_1^2⟩_M$, by this quantity~$∥f_0-ρ M_{κΩ}∥_{L^2}$.
\end{proof}

We can now turn to the study of the rate of convergence to equilibria when it is stable (in the case~$(\frac{j}{c})'>0$).
The main result is the following theorem, which also gives a stronger stability result, in any Sobolev space~$H^s$ with~$s>\frac{n-1}2$.
Let us remark that this theorem is an improvement compared to the results of~\cite{frouvelle2012dynamics}, in the case where~$τ$ is constant and~$ν(|J|)$ is proportional to~$|J|$.
In what follows, we call constant a quantity which does not depend on the initial condition~$f_0$ (that is to say, it depends only on~$s$,~$κ$,~$n$ and the coefficients of the equation~$ν$ and~$τ$).

\begin{theorem}
\label{thm-strong-stability-anisotropic}
Suppose~$(\frac{j}{c})'(κ)>0$.
Then, for all~$s>\frac{n-1}2$, there exist constants~$δ>0$ and~$C>0$ such that for any~$f_0$ with mass~$ρ$ satisfying~$∥f_0-ρ M_{κΩ}∥_{H^s}<δ$ for some~$Ω∈𝕊$, there exists~$Ω_∞∈𝕊$ such that
\[∥f-ρ M_{κΩ_∞}∥_{H^s}⩽ C∥f_0-ρ M_{κΩ}∥_{H^s}e^{-λ t},\]
where the rate is given by
\begin{equation}
\label{def-lambda}
λ=\frac{cτ(j)}{j'}Λ_κ(\frac{j}{c})'.
\end{equation}
The constant~$Λ_κ$ is the best constant for the following weighted Poincaré inequality (see the appendix of~\cite{degond2012macroscopic} for more details on this constant, which does not depend on~$Ω$):
\begin{equation}
\label{poincare-lambda}
⟨|∇_ω g|^2⟩_{M}⩾Λ_κ⟨(g-⟨ g⟩_{M})^2⟩_{M}.
\end{equation}
\end{theorem}

\medskip
\noindent
We first outline the key steps.
Firstly, we want to get a lower bound for the dissipation term~$\mathcal D(f)$ in terms of~$\mathcal F(f)-\mathcal F_κ$, in order to get a Grönwall inequality coming from the conservation relation~\eqref{dissipation-free-energy}.
After a few computations, we get
\[\mathcal D(f)⩾2λ(\mathcal F(f)-\mathcal F_κ)+O((\mathcal F(f)-\mathcal F_κ)^{1+ε}).\]
With this lower bound, we obtain exponential decay of~$\mathcal F(f)-\mathcal F_κ$ (with rate~$2λ$), which also gives exponential decay of~$∥f-M_{κΩ_f}∥_{L^2}$ (with rate~$λ$) in virtue of~\eqref{est-L2-L2}.
We also prove that we can control the displacement~$\dot{Ω}_f$ by~$√{⟨ g_2^2⟩_M}$.
Hence we get that~$Ω_f$ is also converging exponentially fast towards some~$Ω_∞∈𝕊$ (with the same rate~$λ$).
After linearizing the kinetic equation~\eqref{KFP-homogeneous} around this equilibrium~$ρ M_{κΩ_∞}$, an energy estimate for a norm equivalent to the~$H^s$ norm gives then the exponential convergence for~$∥f-M_{κΩ_∞}∥_{H^s}$ with the same rate~$λ$.

\medskip
\noindent
We now give the detailed proof.

\begin{proof}[Proof of Theorem~\ref{thm-strong-stability-anisotropic}.] We fix~$s>\frac{n-1}2$ and we suppose~$(\frac{j}{c})'(κ)>0$.
We recall the notations of the proof of Proposition~\ref{prop-unstability-stability-anisotropic}:
\[f=M_{κΩ_f}(ρ+g_1)=M_{κΩ_f}(ρ+α(\cosθ-c)+g_2),\]
where~$\cosθ=ω\cdotΩ_f$ and~$α$, defined in~\eqref{def-alpha}, is such that
\begin{equation}
\label{def-alpha-appendix}
|J_f|=ρ c + α ⟨(\cosθ-c)^2⟩_M= j + α\, c',
\end{equation}
thanks to~\eqref{eq-c-prime}.
We have that~$⟨g_1⟩_M=⟨g_2⟩_M=0$, and~$⟨ω\,g_2⟩_M=0$.

The proof will be divided in three propositions. 
\begin{proposition}
\label{prop-step-1}
There exist constants~$δ>0$,~$ε>0$ and~$C$ such that, if initially, we have~$⟨ g_1^2⟩_M<δ$ and~$∥f_0-M_{κΩ_{f_0}}∥_{H^s}⩽1$, then for all time, we have 
\begin{gather*}
\mathcal F(f)⩾\mathcal F_κ,\\
\mathcal D(f)⩾2λ(\mathcal F(f)-\mathcal F_κ)-C(\mathcal F(f)-\mathcal F_κ)^{1+ε},
\end{gather*}
where the rate is given by~\eqref{def-lambda}:~$λ=\frac{cτ(j)}{j'}Λ_κ(\frac{j}{c})'$.
\end{proposition}

\begin{proof}
We apply the stability results of the second part of Proposition~\ref{prop-unstability-stability-anisotropic}, with the constant~$K$ being~$1+∥ρ M_{κΩ_{f_0}}∥_{H^s}$ (this does not depend on~$Ω_{f_0}$).
This gives us constants~$δ_1>0$,~$ε>0$,~$C_1$,~$C_2$ such that if we have initially~$⟨ g_1^2⟩_M<δ_1$, then (see formulas~\eqref{bound-infty-L2}-\eqref{est-L2-L2})
\begin{gather}
\nonumber\mathcal F(f)⩾\mathcal F_κ,\\
\label{est-infty-L2-app}∥g_1∥_∞⩽ C_1⟨ g_1^2⟩^{ε}_{M},\\
\label{est-L2-energy-app}\big|\frac1{2ρ}[\frac{cc'}{j'}(\frac{j}{c})'\,α^2 + ⟨ g_2^2⟩_M] -(\mathcal F(f)-\mathcal F_κ)\big| ⩽ C_2(\mathcal F(f)-\mathcal F_κ)^{1+ε},\\
\label{est-g1-energy-app}⟨ g_1^2⟩_{M}⩽ C_3(\mathcal F(f)-\mathcal F_κ).
\end{gather}
We get, using the definition~\eqref{def-dissipation}:
\begin{align*}
\mathcal D(f)&=τ(|J_f|)⟨(ρ+g_1)|∇_ω[\ln(ρ+g_1)-(k(|J_f|)-κ)ω\cdotΩ_f]|^2⟩_{M}\\
&=τ(|J_f|)⟨\frac1{ρ+g_1}|∇_ω g_1|^2+(ρ+g_1)(k(|J_f|)-κ)^2|∇_ω(ω\cdotΩ_f)|^2⟩_{M}\\
&\hspace{2em} - 2 τ(|J_f|)⟨∇_ω g_1·(k(|J_f|)-κ)∇_ω(ω\cdotΩ_f)⟩_{M}.
\end{align*}
Using the fact that~$\frac1{ρ+g_1}⩾\frac1{ρ^2}(ρ-∥g_1∥_∞)$, we obtain
\begin{align*}
\mathcal D(f)&⩾τ(|J_f|)(ρ-∥g_1∥_∞)⟨\frac1{ρ^2}|∇_ω g_1|^2+(k(|J_f|)-κ)^2|∇_ω(ω\cdotΩ_f)|^2⟩_{M}\\
&\hspace{2em} - 2 τ(|J_f|)⟨∇_ω g_1·(k(|J_f|)-κ)∇_ω(ω\cdotΩ_f)⟩_{M}\\
\mathcal D(f)&⩾τ(|J_f|)(ρ-∥g_1∥_∞)⟨|∇_ω[\frac{g_1}{ρ}-(k(|J_f|)-κ)ω\cdotΩ_f]|^2⟩_{M}\\
&\hspace{2em}+τ(|J_f|)\frac2{ρ}∥g_1∥_∞(k(|J_f|)-κ)⟨ g_1(κ|∇_ω(ω\cdotΩ_f)|^2-(n-1)ω\cdotΩ_f)⟩_{M}.
\end{align*}
where we used Green’s formula to evaluate~$⟨∇_ω g_1·∇_ω(ω·Ω_f)⟩_M$.

First of all, using the definition~\eqref{def-alpha-appendix} we can get that~$|k(|J_f|)-κ-α\frac{c'}{j'}|⩽ C_4α^2$, for a constant~$C_4$.
Then we use the Poincaré inequality~\eqref{poincare-lambda}:
\[⟨|∇_ω g|^2⟩_{M}⩾Λ_κ⟨(g-⟨ g⟩_{M})^2⟩_{M}.\]
Hence, since~$|α|$ is controlled by~$\sqrt{⟨ g_1^2⟩_M}$ (we recall that~$⟨ g_1^2⟩_M=c'α^2+⟨ g_2^2⟩_M$), and since we also have~$||J_f|-j|⩽ C_5|α|$ for a constant~$C_5$, we get
\begin{align*}
\begin{split}\mathcal D(f)&⩾Λ_κτ(|J_f|)(ρ-∥g_1∥_∞)⟨|\frac{g_1}{ρ}-(k(|J_f|)-κ)(\cosθ-c)|^2⟩_{M}\\
&\hspace{2em}-C_6∥g_1∥_∞⟨ g_1^2⟩_{M}\end{split}\\
&⩾Λ_κτ(j)ρ⟨|\frac{g_2}{ρ}+α(\frac1{ρ}-\frac{c'}{j'})(\cosθ-c)|^2⟩_{M}-C_7∥g_1∥_∞⟨ g_1^2⟩_{M}\\
&=\frac{Λ_κτ(j)}{ρ}[\frac{c^2c'}{(j')^2}((\frac{j}{c})')^2\,α^2 +⟨ g_2^2⟩_{M}] -C_7∥g_1∥_∞⟨ g_1^2⟩_{M},
\end{align*}
where~$C_6$ and~$C_7$ are constants.
Together with the fact that~$\frac{c}{j'}(\frac{j}{c})'⩽1$ (this is equivalent to~$j c'⩾0$), and with equations~\eqref{est-infty-L2-app}-\eqref{est-g1-energy-app}, this ends the proof.
\end{proof}

\begin{proposition}
\label{prop-step-2}
There exist positive constants~$C,\widetilde C$ and~$δ$ such that if initially, we have~$⟨ g_1^2⟩_M<δ$ and~$∥f_0-ρ M_{κΩ_{f_0}}∥_{H^s}⩽1$, then for all time, we have
\[∥f-ρM_{κΩ_f}∥_{L^2}⩽ C∥f_0-ρ M_{κΩ_{f_0}}∥_{L^2}\,e^{-λ t},\]
and furthermore, there exists~$Ω_∞∈𝕊$ such that for all time, we have
\begin{equation}
\label{decay-Omega}
|Ω_f-Ω_∞|⩽\widetilde C ∥f_0-ρ M_{κΩ_{f_0}}∥_{L^2}\,e^{-λ t}.
\end{equation}
\end{proposition}

\begin{proof}
By Proposition~\ref{prop-step-1}, using the expression~$⟨ g_1^2⟩_M=c'α^2+⟨ g_2^2⟩_M$  and inequalities~\eqref{est-L2-energy-app} and~\eqref{est-g1-energy-app}, we get that there exist constants~$δ_1>0$ and~$C_1$,~$C_2$, and~$\widetilde C_2>0$ such that if~$⟨ g_1^2⟩_M<δ_1$, then~$\mathcal F(f)⩾\mathcal F_κ$, and for all time,
\begin{gather}
\label{gronwall-energy}\frac{\d}{\d t}(\mathcal F(f)-\mathcal F_κ)=-\mathcal D(f)⩽-2λ(\mathcal F(f)-\mathcal F_κ)+C_1(\mathcal F(f)-\mathcal F_κ)^{1+ε},\\
\label{equiv-g1-energy}\widetilde C_2(\mathcal F(f)-\mathcal F_κ)⩽⟨ g_1^2⟩_M⩽ C_2(\mathcal F(f)-\mathcal F_κ).
\end{gather}
Solving the differential inequality~\eqref{gronwall-energy} for~$F(f)-\mathcal F_κ$ sufficiently small, we get that, up to taking~$δ_2<δ_1$, if~$⟨ g_1^2⟩_M<δ_2$, we get a constant~$C_3$ such that
\[\mathcal F(f)-\mathcal F_κ⩽ C_3(\mathcal F(f_0)-\mathcal F_κ) e^{-2λ t}.\]
This gives the first part of the proposition, with~\eqref{equiv-g1-energy}, and the fact that there exists constants~$C_4$, and~$\widetilde C_4$ such that
\[\widetilde C_4∥f-ρM_{κΩ_f}∥_{L^2}⩽√{⟨ g_1^2⟩_M}⩽ C_4∥f-ρM_{κΩ_f}∥_{L^2}.\]
Now we compute the time derivative of~$Ω_f$, using~$\frac{\d}{\d t}Ω_f=\frac1{|J_f|}P_{Ω_f^⊥}\frac{\d}{\d t}J_f$ and~\eqref{ODE-J}:
\[\frac{\d}{\d t}Ω_f=\frac{ν(|J_f|)}{|J_f|}P_{Ω_f^⊥}⟨ P_{ω^⊥} Ω_f (ρ+ α(\cos θ - c) + g_2)⟩_M=-\frac{ν(|J_f|)}{|J_f|}P_{Ω_f^⊥}⟨\cos θ\, ω g_2⟩.\]
So there exist constants~$C_5$ and~$C_6$ such that
\[|\dot{Ω}_f|⩽ C_5√{⟨ g_2^2⟩_M}⩽C_5√{⟨ g_1^2⟩_M}⩽ C_6 ∥f_0-ρM_{κΩ_{f_0}}∥_{L^2}\,e^{-λ t},\]
which, after integration in time, gives the second part of the proposition.
\end{proof}

We can now prove the last step which leads to Theorem~\ref{thm-strong-stability-anisotropic}.

\begin{proposition}
\label{prop-step-3}
There exist constants~$δ>0$ and~$C>0$, such that for any initial condition~$f_0$ with mass~$ρ$ satisfying~$∥f_0-ρ M_{κΩ_{f_0}}∥_{H^s}<δ$, there exists~$Ω_∞∈𝕊$ such that
\[∥f-ρ M_{κΩ_∞}∥_{H^s}⩽ C∥f_0-ρ M_{κΩ_{f_0}}∥_{H^s}e^{-λ t}.\]
\end{proposition}
\begin{proof}
All along this proof we will use the symbol~$\asymp$ to denote quantities of the same order: for~$a$ and~$b$ two nonnegative quantities,~$a\asymp b$ means that there exist two positive constants~$C_1,C_2$ such that~$C_1a⩽ b⩽ C_2a$.

By the estimation~$⟨ g_1^2⟩_M\asymp ∥f-ρ M_{κΩ_f}∥^2_{L^2}$ (since the weight~$M_{κΩ}$ is bounded above and below), and by a simple Sobolev embedding ($L^2⊂ H^s$), there exists a constant~$δ_1>0$ such that if~${∥f_0-ρ M_{κΩ_{f_0}}∥_{H^s}<δ_1}$, then we are in the hypotheses of Proposition~\ref{prop-step-2}.
We suppose we are in that case and we can then go back to the original equation and perform a linear analysis around~$ρ M_{κΩ_∞}$.
We will now write~$⟨ g⟩_M$ for~$∫_𝕊g M_{κΩ_∞}\d ω$.

If we write~$f=(ρ+g)M_{κΩ_∞}$, then the equation becomes
\begin{equation}
\label{KFP-expand-g}
∂_t g= - τ(|J_f|) L g - A(t)·∇_ω g + B(ω)· A(t)(ρ+g),
\end{equation}
where
\begin{gather*}
Lg=-\frac{1}{M_{κΩ_∞}}∇_ω·(M_{κΩ_∞}∇_ω g)=-(Δ_ω g+κΩ_∞·∇_ω g),\\
A(t)=ν(|J_f|)Ω_f-τ(|J_f|)κΩ_∞,\\
B(ω)=(n-1)ω-κ P_{ω^⊥}Ω_∞.
\end{gather*}
Let us remark that the linear operator~$L$ is a coercive selfadjoint operator for the inner product~$(g_1,g_2)↦⟨ g_1g_2⟩_M$ (also denoted~$⟨ g_1,g_2⟩_M$ in the following), on the space~$\dot{L}_M^2⊂ L^2$ of functions~$g$ such that~$⟨ g⟩_M=0$ (thanks to the Poincaré inequality~\eqref{poincare-lambda}).
Indeed we have
\[⟨ g_1,Lg_2⟩_M=⟨∇_ω g_1·∇_ω g_2⟩_M.\]
It is classical to prove that the inverse of~$L$ is a positive selfadjoint compact operator of~$\dot{L}_M^2$.
Hence, by spectral decomposition, we can define the operator~$L^s$, and use it to define a new Sobolev norm by
\[∥g∥^2_{\dot{H}_M^s}=⟨ g,L^s g⟩_M.\]
We will use a lemma (the proof of which is postponed at the end of this section) about estimations for this norm, and about a commutator estimate:
\begin{lemma}
\label{lemma-modified-Hs}
For~$s⩾0$, we have~$∥g∥_{\dot{H}_M^s}\asymp ∥g∥_{H^s}$, for functions~$g$ in~$\dot{H}_M^s=H^s∩\dot{L}_M^2$.

Furthermore, for~$g∈\dot{H}_M^s$, the (vector valued) quantity~$⟨ L^s g ∇_ω g⟩_{M}$ is well defined and there is a constant~$C$ such we have:
\begin{equation}
|⟨ L^s g ∇_ω g⟩_M|⩽ C∥g∥^2_{\dot{H}_M^s}.
\label{commutator-estimate-lemma}
\end{equation}
\end{lemma}

We will also use the following Poincaré estimate, for~$g∈\dot{H}_M^s$, with the same constant~$Λ_κ$ as in~\eqref{poincare-lambda}:
\[⟨g,L^{s+1}g⟩_M=⟨|∇(L^{\frac{s}2}g)|^2⟩_M⩾Λ_κ⟨(L^{\frac{s}2}g)^2⟩_M=Λ_κ∥g∥^2_{\dot{H}_M^s}\]
We now multiply the equation~\eqref{KFP-expand-g} by~$L^s g$ and integrate with respect to the measure~$M_{κΩ_∞}\dω$. We get
\[\frac12\frac\d{\d t}∥g∥^2_{\dot{H}_M^s}⩽-τ(|J_f|)Λ_κ∥g∥^2_{\dot{H}_M^s}+ |A(t)|(C_1\,∥g∥^2_{\dot{H}_M^s}+∥g∥_{\dot{H}_M^s}∥B(ω)(ρ+g)∥_{\dot{H}_M^s}),\]
where~$∥B(ω)(ρ+g)∥_{\dot{H}_M^s}$ denotes the maximum of~$∥e·B(ω)(ρ+g)∥_{\dot{H}_M^s}$ for~$e∈𝕊$. Since~$ω↦ e· B(ω)$ is smooth, the multiplication by~$e· B(ω)$ is a continuous operator from~$\dot{H}_M^s$ to~$H^s$ when~$s$ is an integer, so by interpolation this is true for all~$s$.
Therefore, we get a constant~$C_2$ such that for all~$g∈\dot{H}_M^s$, we have
\begin{equation}
\label{estimate-multB}
∥B(ω)g∥_{\dot{H}_M^s}⩽ C_2\,∥g∥_{\dot{H}_M^s}.
\end{equation}
We finally get
\[\frac\d{\d t}∥g∥_{\dot{H}_M^s}⩽-τ(|J_f|)Λ_κ∥g∥_{\dot{H}_M^s}+ |A(t)|((C_1+C_2)\,∥g∥_{\dot{H}_M^s}+∥B(ω)ρ∥_{\dot{H}_M^s}).\]
Now, applying Proposition~\ref{prop-step-2},  there exist constants~$C_3,C_4,C_5$ such that
\begin{align*}
|A(t)|&⩽ν(|J_f|)|Ω_f-Ω_∞|+[ν(|J_f|)-τ(|J_f|)κ]|Ω_∞|\\
&⩽ν(|J_f|)|Ω_f-Ω_∞|+τ(|J_f|)[k(|J_f|)-k(j(κ))]\\
&⩽ C_3|Ω_f-Ω_∞|+C_4||J_f|-j(κ)|\\
&⩽ C_5 ∥f_0-ρ M_{κΩ_{f_0}}∥_{L^2}\,e^{-λ t}⩽ C_5 ∥f_0-ρ M_{κΩ_{f_0}}∥_{\dot{H}_M^s}\,e^{-λ t}.
\end{align*}
in virtue of~\eqref{est-1stmomentnorm}.
By the same argument, we get, for a constant~$C_6$, that
\[|τ(|J_f|)-τ(j)|⩽ C_6∥f_0-ρ M_{κΩ_{f_0}}∥_{\dot{H}_M^s}\,e^{-λ t},\]
so we finally obtain, together with a uniform bound on~$∥g∥_{\dot{H}_M^s}$ coming from Theorem~\ref{theorem-existence-uniqueness} (and independent of~$f_0$ since~$∥f_0-ρ M_{κΩ_{f_0}}∥_{H^s}<δ_1$), a constant~$C_7$ such that
\[\frac\d{\d t}∥g∥_{\dot{H}_M^s}⩽-τ(j)Λ_κ∥g∥_{\dot{H}_M^s}+C_7∥f_0-ρ M_{κΩ_{f_0}}∥_{\dot{H}_M^s}\,e^{-λ t} .\]
We solve this inequality and we get
\[∥g∥_{\dot{H}_M^s}⩽∥g_0∥_{\dot{H}_M^s}\exp(-τ(j)Λ_κ t) +C_7∥f_0-ρ M_{κΩ_{f_0}}∥_{\dot{H}_M^s}\frac{e^{-λ t}-e^{-τ(j)Λ_κ t}}{τ(j)Λ_κ-λ},\]
and this gives the final estimation, using the fact that~$λ<τ(j)Λ_κ$ (this is equivalent, by definition~\eqref{def-lambda} of~$λ$, to~$(\frac{j}{c})'<\frac{j'}c$, and we indeed have~$j \,c'>0$), and that
\begin{align*}
∥g_0∥_{\dot{H}_M^s}&\asymp ∥f_0-ρ M_{κΩ_∞}∥_{\dot{H}_M^s}\\
&⩽∥f_0-ρ M_{κΩ_{f_0}}∥_{\dot{H}_M^s}+C_8|Ω_{f_0}-Ω_∞|\\
&⩽ C_9 ∥f_0-ρ M_{κΩ_{f_0}}∥_{H^s},
\end{align*}
in virtue of~\eqref{est-1stmomentnorm}, Lemma~\ref{lemma-modified-Hs} and~\eqref{decay-Omega} (we have~$s>\frac{n-1}2$ so~$L^2⊂H^s$ is a continuous embedding).
\end{proof}
Finally, Proposition~\ref{prop-step-3} can be refined since, thanks to the estimation~\eqref{est-1stmomentnorm}, we only need to control~$∥f_0-ρ M_{κΩ}∥_{\dot{H}_M^s}$ for a given~$Ω∈𝕊$ in order to ensure that~$∥f_0-ρ M_{κΩ_{f_0}}∥_{\dot{H}_M^s}$ is sufficiently small, and this ends the proof of Theorem~\ref{thm-strong-stability-anisotropic}.
\end{proof}

\begin{proof}[Proof of Lemma~\ref{lemma-modified-Hs}] We first define the space~$\dot{H}_M^s$ as the completion of~$C^∞(𝕊)∩\dot{L}_M^2$ for~$∥·∥_{\dot{H}_M^s}$.
The first estimate (which amounts to prove that~$∥g∥_{\dot{H}_M^s}\asymp ∥g∥_{H^s}$ for smooth functions~$g∈ C^∞(𝕊)∩\dot{L}_M^2$) is true when~$s$ is an integer: indeed~$L^s$ and~$(-Δ_ω)^s$ are simple differential operators (of degree~$2s$), and these estimates can be done by induction on~$s$: when~$s=2p$ is even, we write
\[\begin{cases}
⟨ g,L^s g⟩_M=∥L^pg∥_{L^2_M}^2\asymp ∥L^pg∥_2^2\\ ∥g∥_{H^s}^2= ∥(-Δ)^pg∥_2^2\asymp ∥(-Δ)^pg∥_{L^2_M}^2.
\end{cases}\]
In the first case,~$L$ is decomposed as~$(-Δ_ω)-κΩ_∞·∇_ω$ to estimate~$∥L^pg∥_2^2$ in terms of~$∥g∥_{H^s}^2$, and in the second case~$-Δ_ω$ is decomposed as~$L+κΩ_∞·∇_ω$ to estimate~$∥(-Δ)^pg∥_{L^2_M}^2$ in terms of~$⟨ g,L^s g⟩$.
When~$s=2p+1$ is odd, the same argument applies, writing
\[\begin{cases}
⟨ g,L^s g⟩_M=\big{∥}|∇_ω(L^pg)|\big{∥}_{L^2_M}^2\asymp \big{∥}|∇_ω(L^pg)|\big{∥}_2^2\\ ∥g∥_{H^s}^2= \big{∥}|∇_ω(-Δ)^pg|\big{∥}_2^2\asymp \big{∥}|∇_ω(-Δ)^pg|\big{∥}_{L^2_M}^2.
\end{cases}\]
Finally, the general case is done by interpolation, for~$s=n+θ$, with~$θ∈(0,1)$. We refer the reader to~\cite{tartar2007introduction} for an introduction to interpolation spaces, and we will denote~$(F_1,F_2)_{(θ,p)}$ the interpolation space between~$F_1$ and~$F_2$ using the real interpolation method. 
Using the so-called~$K$-method (see~\cite[Lecture~$22$]{tartar2007introduction}), it consists in the space of elements~$u∈F_1+F_2$ such that~$∥u∥_{θ,p}<+∞$, together with the norm~$∥·∥_{θ,p}$, where
\[∥u∥_{θ,p}=\Big(∫_0^∞[t^{-θ}K(t,u)]^q\frac{\d t}{t}\Big)^\frac1q,\quad\text{with}\quad K(t,u)=\inf_{\substack{u=u_1+u_2,\\u_1∈F_1,u_2∈F_2}}∥u_1∥_{F_1}+t∥u_2∥_{F_2}.\]
We will use the following result (see~\cite[Lemma~$23.1$]{tartar2007introduction}): 
if~$(X,μ)$ is a measured space and~$w_0,w_1$ are two weight functions, we have
\begin{equation}
\label{weighted-L2-interpolation}
(L^2(w_0\dμ),L^2(w_1\dμ))_{(θ,2)}=L^2(w_0^{1-θ}w_1^θ\dμ),
\end{equation}
where, for a weight function~$w⩾0$, the weighted space~$L^2(w\dμ)$ denotes the functions~$h$ such that~$∥h∥_{L^2(w\dμ)}^2=∫_Xh^2(x)w(x)\dμ(x)$ is finite.
Now if~$(g_i)_{i∈ℕ}$ is an orthonormal basis (for the dot product~$⟨·,·⟩_M$) of eigenvectors of~$L$ (associated to the eigenvalues~$(λ_i)$), it is easy to see that the map~$h↦(⟨h,g_i⟩_M)_{i∈ℕ}$ is an isometry between~$\dot{H}^s_M$ and the weighted~$ℓ^2$ space with weight~$(λ_i)_{i∈ℕ}$ (it corresponds to~$L^2(w\dμ)$ where~$X=ℕ$,~$μ$ is the counting measure, and~$w(i)=λ_i$). Therefore, we obtain with~\eqref{weighted-L2-interpolation} that~$\dot{H}_M^s=(\dot{H}_M^n,\dot{H}_M^{n+1})_{(θ,2)}$, and by the same argument~${H^s=(H^n,H^{n+1})_{(θ,2)}}$. So we finally get, with equivalence of norms:
\[H^s∩\dot{L}_M^2=(H^n∩\dot{L}_M^2,H^{n+1}∩\dot{L}_M^2)_{(θ,2)}=(\dot{H}_M^n,\dot{H}_M^{n+1})_{(θ,2)}=\dot{H}_M^s.\]
To get the estimation~\eqref{commutator-estimate-lemma}, we first observe that it is a commutator estimate.
Indeed, by integration by parts for a given~$e∈𝕊$, we get that the adjoint operator of~$e·∇_ω$ (for~$⟨·,·⟩_M$) is~$-e·∇_ω + e· B(ω)$, where~$B(ω)=(n-1)ω-κ P_{ω^⊥}Ω_∞$ (the same expression as in the proof of Proposition~\ref{prop-step-3}).
So, splitting the left part of~\eqref{commutator-estimate-lemma} in two halves, we are led to show that for~$g∈\dot{H}_M^s$, we have
\[\frac12|⟨ g[L^s,∇_ω]g⟩_M + ⟨ B(ω)g L^s g⟩_M|⩽ C⟨ g,L^s g⟩_M.\]
Using~\eqref{estimate-multB}, it is equivalent to find a constant~$\widetilde C$ such that for all~$g∈\dot{H}_M^s$, we have
\begin{equation}
|⟨ g[L^s,∇_ω]g⟩_M|⩽ C⟨ g,L^s g⟩_M.
\label{estimate-commutator}
\end{equation}
In the case~$s=1$, by using Schwartz Theorem, we see that~$[L,∇_ω]=[-Δ_ω,∇_ω]$.
It is proven in Lemma~$2.1$ of~\cite{frouvelle2012dynamics} that~\eqref{commutator-estimate-lemma} is true in the limit case where~$κ=0$.
This means that~$[(-Δ_ω)^s,∇_ω]$ is an operator of degree~$2s$.
In particular~$[-Δ_ω,∇_ω]$ is a differential operator of degree~$2$.
Actually, using Lemma~$A.5$ of~\cite{frouvelle2012dynamics}, it is possible to get that
\[[-Δ_ω,∇_ω]=2ωΔ_ω-(n-3)∇_ω.\]
This directly gives the estimate~\eqref{estimate-commutator} when~$s=1$.
We obtain the estimate when~$s$ is an integer with the formula~$[L^{p+1},∇_ω]=\sum_{q=0}^p L^{p-q}[L,∇_ω]L^q$.

The proof in the general case relies on a resolvent formula for the operator~$A^θ$, when~$θ$ belongs to~$(0,1)$, and~$A:\mathcal D(A)⊂ H→ H$ is a strictly positive operator of a Hilbert space~$H$ with a complete basis of eigenvectors (see~\cite[Remark~V-$3.50$]{kato1995perturbation}):
\[A^θ=\frac{\sin{\piθ}}{π}∫_0^∞ t^θ(t^{-1}-(t+A)^{-1})\d t.\]
This formula can be checked on an orthonormal basis of eigenvectors of~$A$, and relies on the fact that, for~$λ>0$, we have
\[∫_0^∞ t^θ\big(\frac1t-\frac1{t+λ}\big)\d t=λ^θ∫_0^∞\frac{t^{θ-1}\d t}{1+t}.\]
The fact that this last integral is equal to~$\frac{π}{\sin{\piθ}}$ for~$0<θ<1$ is classical, and can be done by the method of residues.

We then have, for another operator~$B$ (with dense domain for~$(t+A)^{-1}B$ and~$B(t+A)^{-1}$ for~$t>0$) 
\begin{align*}
[A^θ,B]&=\frac{\sin{\piθ}}{π}∫_0^∞ t^θ[B,(t+A)^{-1}]\d t\\
&=\frac{\sin{\piθ}}{π}∫_0^∞ t^θ(t+A)^{-1}[A,B](t+A)^{-1}\d t.
\end{align*}
We can apply this result to~$A=L^m$ with~$H=\dot{L}^2_M$,~$B=e·∇_ω$ for a fixed~$e∈𝕊$, and~$θ=\frac{s}m$ for~$0<s<m$ and we get, using the fact that~$(t+L^m)^{-1}$ is self-adjoint in~$H$ (and bounded, so all smooth functions are in the domain of~$(t+A)^{-1}B$ and~$B(t+A)^{-1}$),
\begin{align*}
|⟨ g\,[L^s,e·∇_ω]g⟩_M|&⩽\frac{\sin{\piθ}}{π}∫_0^∞ t^θ|⟨ g\,(t+L^m)^{-1}[L^m,e·∇_ω](t+L^m)^{-1}g⟩_M|\d t\\
&⩽\frac{\sin{\piθ}}{π}∫_0^∞ t^θ|⟨(t+L^m)^{-1}g\,[L^m,e·∇_ω](t+L^m)^{-1}g⟩_M|\d t\\
&⩽ C_m\frac{\sin{\piθ}}{π}∫_0^∞ t^θ⟨(t+L^m)^{-1}g,L^m(t+L^m)^{-1}g⟩_M\d t\\
&⩽ C_m\frac{\sin{\piθ}}{π}∫_0^∞ t^θ⟨ g,(t+L^m)^{-1}L^m(t+L^m)^{-1}g⟩_M\d t.
\end{align*}
But as before, it is easy to see that
\[∫_0^∞ t^θ\frac{λ}{(t+λ)^2}\d t=λ^θ∫_0^∞\frac{t^{θ}\d t}{(1+t)^2}=θλ^θ\frac{π}{\sin{\piθ}},\]
and then
\[θ A^θ=\frac{\sin{\piθ}}{π}∫_0^∞ t^θ(t+A)^{-1}A(t+A)^{-1}\d t.\]
Finally, we get
\[|⟨ g\,[L^s,e·∇_ω]g⟩_M|⩽ C_m\frac{s}{m}⟨ g,L^s g⟩_M,\]
wich ends the proof of Lemma~\ref{lemma-modified-Hs}
\end{proof}

\section{Phase transitions}
\label{sec-phase-transition}

\subsection{Application of the previous theory to two special cases}
\label{subsec-appli}

In the previous section, we have stated results which are valid for all possible behaviors of the function~$κ ↦ \frac{j(κ)}{c(κ)}$.
In particular, the number of branches of equilibria can be arbitrary.

In this section, we apply the previous theory to two typical examples: 
\begin{itemize}
\item[(i)] The function~$κ ↦ \frac{j(κ)}{c(κ)}$ is increasing.
In this case, there exists only one branch of stable von Mises--Fisher equilibria.
The uniform equilibria are stable for~${ρ<ρ_c}$, where~$ρ_c=\lim_{κ → 0} \frac{j(κ)}{c(κ)}$, and become unstable for~$ρ>ρ_c$.
The von Mises--Fisher equilibria only exist for~$ρ>ρ_c$ and are stable.
This corresponds to a second-order phase transition.
We will provide details and a determination of the critical exponent of this phase transition in section~\ref{subsec-second}.
\item[(ii)] The function~$κ ↦ \frac{j(κ)}{c(κ)}$ is unimodal, i.e.
there exists~$κ_*$ such that this function is decreasing on~$[0,κ_*]$ and increasing on~$[κ_*, ∞)$.
Then, another critical density is defined by~$ρ_*=\frac{j(κ_*)}{c(κ_*)}$.
Then we have the following situation: 
\begin{itemize}
\item[a.] if~$ρ ∈ (ρ_*, ρ_c)$, there exist two branches of von Mises--Fisher equilibria, and therefore, three types of equilibria if we include the uniform distribution.
Both the uniform distribution and the von Mises--Fisher distribution with the largest~$κ$ are stable while the von Mises--Fisher distribution with intermediate~$κ$ is unstable.
\item[b.] if~$ρ<ρ_*$, there exists only one equilibrium, the uniform one, which is stable.
\item[c.] if~$ρ>ρ_c$, there exist two types of equilibria, the uniform equilibrium which is unstable and the von Mises--Fisher equilibria which are stable.
\end{itemize}
This situation corresponds to a first-order phase transition and is depicted in section~\ref{subsec-first}, where phase diagrams for both the two-dimensional and three-dimensional cases are given.
The major feature of first-order phase transitions is the phenomenon of hysteresis, which will be illustrated by numerical simulations in section~\ref{subsec-first}.
\end{itemize}
For references to phase transitions, we refer the reader to~\cite{goldenfeld2005lectures}.

\subsection{Second order phase transition}
\label{subsec-second}

Let us now focus on the case where we always have~$(\frac{j}{c})'>0$ for all~$κ>0$ (this corresponds for example to the upper two curves of Figure~\ref{fig-shapes}).
In this case, the compatibility equation~\eqref{eq-compat2} has a unique positive solution for~$ρ>ρ_c$.
With the results of the previous subsection about stability and rates of convergence, we obtain the behavior of the solution.
\begin{proposition}
\label{prop-second-order}
Let~$f_0$ be an initial condition with mass~$ρ$, and~$f$ the corresponding solution of~\eqref{KFP-homogeneous}.
We suppose that~$(\frac{j}{c})'>0$ for all~$κ>0$.
Then:
\begin{itemize}
\item[(i)] If~$ρ<ρ_c$, then~$f$ converges exponentially fast towards the uniform distribution~$f_∞=ρ$.
\item[(ii)] If~$ρ=ρ_c$, then~$f$ converges to the uniform distribution~$f_∞=ρ$.
\item[(iii)] If~$ρ>ρ_c$ and~$J_{f_0}≠0$, then there exists~$Ω_∞$ such that~$f$ converges exponentially fast to the von Mises--Fisher distribution~$f_∞=ρ M_{κΩ_∞}$, where~$κ>0$ is the unique positive solution to the equation~$ρ c(κ)=j(κ)$.
\end{itemize}
\end{proposition}
\begin{proof}
This is a direct application of Propositions~\ref{prop-lasalle-refined}-\ref{prop-stability-uniform} and Theorem~\ref{thm-strong-stability-anisotropic}.
\end{proof}
\begin{remark}~
\begin{itemize}
\item[(i)] When~$ρ>ρ_c$, the special case where~$J_{f_0}=0$ leads to the study of heat equation~$∂_t f=τ_0Δ_ω f$.
Its solution converges exponentially fast to the uniform distribution, but this solution is not stable under small perturbation of the initial condition.
\item[(ii)] For some particular choice of the coefficients, as in~\cite{frouvelle2012dynamics}, it is also possible to get a polynomial rate of convergence in the second case~$ρ=ρ_c$.
For example when~$j(κ)=κ$, we have~${∥f-ρ∥⩽ Ct^{-\frac12}}$ for~$t$ sufficiently large.
\end{itemize}
\end{remark}
We now describe the phase transition phenomena by studying the order parameter of the asymptotic equilibrium~$α=\frac{|J_{f_∞}|}{ρ}$, as a function of the initial density~$ρ$.

We have~$α(ρ)=0$ if~$ρ⩽ρ_c$, and~$α(ρ)=c(κ)$ for~$ρ>ρ_c$, where~$κ>0$ is the unique positive solution to the equation~$ρ c(κ)=j(κ)$.
This is a positive continuous increasing function for~$ρ>ρ_c$.
This is usually described as a continuous phase transition, also called second order phase transition.

\begin{definition}
\label{def-critical-exponent}
We say that~$β$ is the critical exponent of the phase transition if there exists~$α_0>0$ such that
\[α(ρ)\sim α_0(ρ-ρ_c)^β,\text{ as }ρ\overset{>}{→}ρ_c.\]
\end{definition}

This critical exponent~$β$ can take arbitrary values in~$(0,1]$, as can be seen by taking~$k$ such that~$j(κ)=c(κ)(1+κ^{\frac1{β}})$ (we recall that~$k$ is the inverse function of~$j$, see Hypothesis~\ref{hyp-increasing}).
Indeed in this case, the function~$k$ is well defined (its inverse~$j$ is increasing), and satisfies Hypothesis~\ref{hyp-differentiable} (if~${β⩽1}$).
We then have~$(\frac{j}{c})'=\frac1{β}κ^{\frac1{β}-1}>0$, and the conclusions of Proposition~\ref{prop-second-order} apply, with~${ρ_c=1}$.
Finally, the compatibility equation~$ρc(κ)=j(κ)$ becomes~$ρ=(1+κ^{\frac1{β}})$, i.e.~${κ=(ρ-1)^β}$.
And since~$c(κ)\sim\frac1n κ$ when~$κ→0$, we get:
\[α(ρ)=c\big((ρ-1)^β\big)\sim\frac1{n}(ρ-1)^β\text{ as }ρ\overset{>}{→}1.\]
More generally, we can give the expression of the critical exponent in terms of the expansion of~$k$ in the neighborhood of~$0$.
\begin{proposition}
\label{prop-critical-exponent}
We suppose, as in Proposition~\ref{prop-second-order}, that~$(\frac{j}{c})'>0$ for all~$κ>0$.
We assume an expansion of~$k$ is given under the following form:
\begin{equation}
\label{expansion-k1}
\frac{k(|J|)}{|J|}=\frac{n}{ρ_c}-a|J|^q+o(|J|^q)\text{ as }|J|→0,
\end{equation}
with~$q⩾1$ (see Hypothesis~\ref{hyp-differentiable}) and~$a∈ℝ$.
\begin{itemize}
\item[(i)] If~$q<2$ and~$a≠0$, then~$a>0$ and we have a critical exponent given by~$β=\frac1q$.
\item[(ii)] If~$q>2$, the critical exponent is given by~$β=\frac12$.
\item[(iii)] If~$q=2$ and~$a≠-\frac{n^2}{ρ_c^3(n+2)}$, then~$a>-\frac{n^2}{ρ_c^3(n+2)}$ and the critical exponent is given by~$β=\frac12$.
In the special case where
\[\frac{k(|J|)}{|J|}=\frac{n}{ρ_c}+\frac{n^2}{ρ_c^3(n+2)}|J|^2 - a_2|J|^p+o(|J|^p)\text{ as }|J|→0,\]
with~$2<p<4$ and~$a_2≠0$, then~$a_2>0$ and we have a critical exponent given by~$β=\frac1p$.
\end{itemize}
\end{proposition}
It is also possible to give more precise conditions for a higher order expansion of~$k$ in order to have a critical exponent less than~$\frac14$, the point~(iii) of this proposition is just an example of how to get an exponent less than~$\frac12$.
We will only detail the proofs of the first two points, the last one can be done in the same way, with more computations, which are left to the reader.
\begin{proof}
We only detail the first two points, the last one is done in the same way, with more complicate computations.
We recall that~$k(j(κ))=κ$ by definition of~$j$.
So we get that~$κ\sim\frac{n j(κ)}{ρ_c}$ as~$κ→0$.
And using~\eqref{expansion-k1}, we obtain
\[\frac{κ}{j(κ)}=\frac{k(j(κ))}{j(κ)}=\frac{n}{ρ_c}-a\Big(\frac{κρ_c}{n}\Big)^q + o(κ^q).\]
Furthermore, we have~$\frac{c(κ)}{κ}=\frac1n-\frac1{n^2(n+2)}κ^2+O(κ^4)$ (see~\cite{frouvelle2012dynamics}, Remark~$3.5$) as~$κ→0$.
So we get, as~$κ→0$:
\[\frac1{ρ}=\frac{c(κ)}{j(κ)}=\frac{κ}{j(κ)}\,\frac{c(κ)}{κ}=\frac1{ρ_c}\Big(1
 -a(\tfrac{ρ_c}{n})^{q+1}\, κ^q -\tfrac1{n(n+2)}κ^2\Big) + o(κ^{\min(q,3)}).\]
So since~$κ↦\frac{c(κ)}{j(κ)}$ is decreasing, if~$q<2$ and~$a≠0$ we have~$a>0$.
In this case, we get that~$ρ=ρ_c(1+a(\tfrac{ρ_c}{n})^{q+1}\, κ^q) + o(κ^{q})$ as~$κ→0$.
Hence, as~$ρ\overset{>}{→}ρ_c$, we have~$κ\sim\frac{n^{1+\frac1q}}{a^{\frac1q}(ρ_c)^{1+\frac2q}}(ρ-ρ_c)^{\frac1q}$.
Since~$c(κ)\sim\frac{κ}n$ as~$κ→0$, we obtain~(i).

For the same reason, if~$q>2$, we get~$ρ=ρ_c(1+\tfrac1{n(n+2)}κ^2) + o(κ^{\min(q,3)})$ as~$κ→0$, and then~$κ\sim √{\frac{n(n+2)}{ρ_c}(ρ-ρ_c)}$ as~$κ→0$, which proves point~(ii).
\end{proof}

The hypothesis in Proposition~\ref{prop-second-order} is not explicit in terms of the alignment and diffusion rates~$ν$ and~$τ$.
We have a more direct criterion in terms of~$k$ which is given below (but which is more restricted in terms of the critical exponents that can be attained).

\begin{lemma}
\label{lemma-2nd-order}
If~$\frac{k(|J|)}{|J|}$ is a non-increasing function of~$|J|$, then we have~$(\frac{j}{c})'>0$ for all~$κ>0$.
In this case, the critical exponent, if it exists, can only take values in~$[\frac12,1]$.
\end{lemma}

\begin{proof}
We have that~$\frac{\d}{\d κ}(\frac{c(κ)}{κ})<0$ for~$κ>0$ (see~\cite{frouvelle2012dynamics}).
Then
\[\Big(\frac{j}{c}\Big)'=\Big(\frac{κ}{c} \frac{j}{k(j)}\Big)'=\frac{κ}{c}\Big(\frac{j}{k(j)}\Big)'+\Big(\frac{κ}{c}\Big)'\frac{j}{k(j)}<0,\]
since~$(\frac{j}{k(j)})'⩾0$ ($j$ is an increasing function of~$κ$ and~$\frac{k(|J|)}{|J|}$ is a non-increasing function of~$|J|$).
Now if we suppose that there is a critical exponent~$β$ according to Definition~\ref{def-critical-exponent}, we get, using the fact that~$α(ρ)=c(κ)\sim\frac{κ}n$ as~$κ→0$, that
$\frac1{ρ}=\frac1{ρ_c}-aκ^{\frac1{β}}+o(κ^{\frac1{β}})$ as~$κ→0$, with~$a=(ρ_c)^{-2}(nα_0)^{-\frac1{β}}$.
We then have
\begin{align*}
\frac{k(j)}{j}&=\frac{κ}c\,\frac{c}{j}=(n+\frac1{n+2}κ^2+O(κ^4))(\frac1{ρ_c}-aκ^{\frac1{β}}+o(κ^{\frac1{β}}))\\
&=\frac{n}{ρ_c} + \frac1{(n+2)ρ_c}κ^2-n aκ^{\frac1{β}}+o(κ^{\min(2,\frac1{β})}).
\end{align*}
Then~$β$ cannot be less than~$\frac12$, otherwise the function~$\frac{k(j(κ))}{j(κ)}$ could not be a nonincreasing function of~$κ$ in the neighborhood of~$0$.
\end{proof}
\begin{remark}
\label{remark-global-rate-2nd-order}
When this criterion is satisfied, we can also use the result of Remark~\ref{remark-special-entropy}.
Indeed, in that case we get easily that~$\hat{ρ}=ρ_c$, and we obtain that for any~$ρ<ρ_c$, there is a global rate of decay for the modified~$H^{-\frac{n-1}2}$ norm: for all~$f_0∈ H^{-\frac{n-1}2}(𝕊)$, we have the estimation~\eqref{eq-global-decay}.
\end{remark}

\subsection{First order phase transition and hysteresis}
\label{subsec-first}

We now turn to a specific example, where all the features presented in the stability study can be seen.
We focus on the case where~$ν(|J|)=|J|$, as in~\cite{frouvelle2012dynamics}, but we now take~$τ(|J|)=1/(1+|J|)$.
From the modeling point of view, this can be related to the Vicsek model with vectorial noise (also called extrinsic noise)~\cite{aldana2009emergence,chate2008collective}, since in that case the intensity of the effective noise is decreasing when the neighbors are well aligned.

In this case, we have~$k(|J|)=|J|+|J|^2$, so the assumptions of Lemma~\ref{lemma-2nd-order} are not fulfilled, and the function~$j$ is given by~$j(κ)=\frac12(√{1+4κ}-1)$.

Expanding~$\frac{j}{c}$ when~$κ$ is large or~$κ$ is close to~$0$, we get
\[\frac{j}{c}=\begin{cases}
n\,(1-κ)+O(κ^2)&\text{ as }κ→0,\\
\sqrt{κ}+O(1)&\text{ as }κ→∞.
\end{cases}\]
Consequently, there exist more than one family of non-isotropic equilibria when~$ρ$ is close to~$ρ_c=n$ (and~$ρ<ρ_c$).

The function~$κ↦\frac{j(κ)}{c(κ)}$ can be computed numerically.
The results are displayed in Figure~\ref{fig-2d-3d-hysteresis} in dimensions~$n=2$ and~$n=3$.
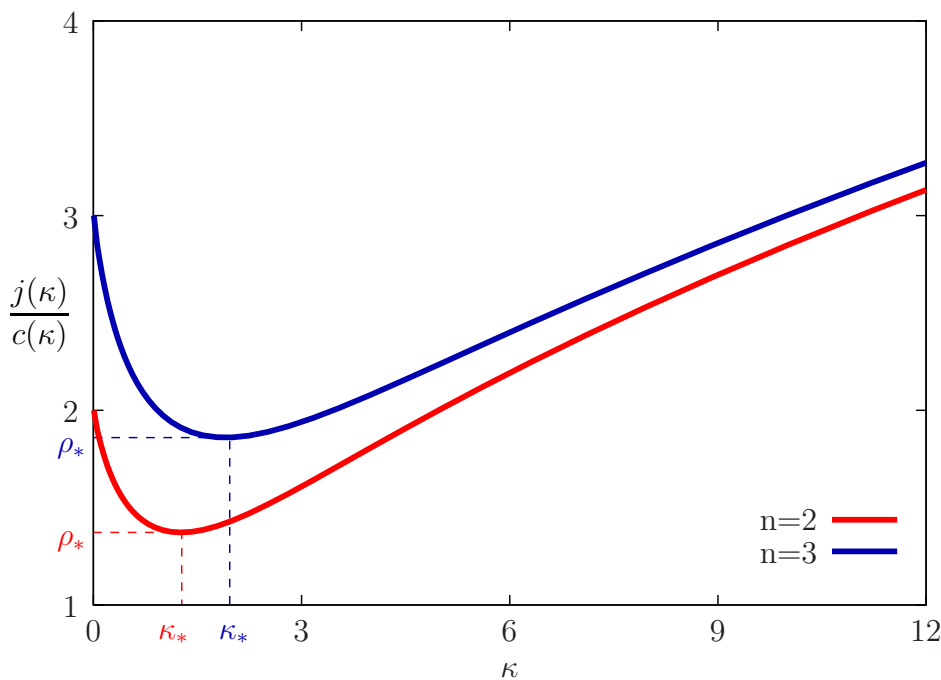
\begin{figure}[h]
\centering
\input{rho_kappa_hysteresis.pspdftex}
\caption{The function~$κ↦\frac{j(κ)}{c(κ)}$, in dimensions~$2$ and~$3$.}
\label{fig-2d-3d-hysteresis}
\end{figure}

We observe the following features:
\begin{itemize}
\item There exists a unique critical point~$κ_*$ for the function~$\frac{j}{c}$, corresponding to its global minimum~$ρ_*$ (we obtain numerically~$ρ_*≈ 1.3726$ and~$κ_*≈ 1.2619$ if~$n=2$, and~$ρ_*≈ 1.8602$ and~${κ_*≈ 1.9014}$ if~$n=3$).
\item The function~$\frac{j}{c}$ is strictly decreasing in~$[0,κ_*)$ and strictly increasing in~$(κ_*,∞)$.
\end{itemize}
We conjecture that this is the exact behavior of the function~$\frac{j}{c}$, called unimodality.
From these properties, it follows that the solution associated to an initial condition~$f_0$ with mass~$ρ$ can exhibit different types of behavior, depending on the three following regimes for~$ρ$.
\begin{proposition}
\label{prop-hysteresis}
We assume that the function~$\frac{j}{c}$ is unimodal, as described above.
Then we have the following hysteresis phenomenon:
\begin{itemize}
\item[(i)] If~$ρ<ρ_*$, then the solution converges exponentially fast to the uniform equilibrium~$f_∞=ρ$.
\item[(ii)] If~$ρ_*<ρ<n$, there are two families of stable solutions: either the uniform equilibrium~$f=ρ$ or the von Mises--Fisher distributions of the form~$ρ M_{κΩ}$ where~$κ$ is the unique solution with~$κ>κ_*$ of the compatibility equation~\eqref{eq-compatibility} and~$Ω∈𝕊$.
If~$f_0$ is sufficiently close to one of these equilibria, there is exponential convergence to an equilibrium of the same family.

The von Mises--Fisher distributions of the other family (corresponding to solution of~\eqref{eq-compatibility} such that~$0<κ<κ_*$) are unstable in the sense given in Proposition~\ref{prop-unstability-stability-anisotropic}.
\item[(iii)] If~$ρ>n$ and~$J_{f_0}≠0$, then there exists~$Ω_∞∈𝕊$ such that~$f$ converges exponentially fast to the von Mises--Fisher distribution~$ρ M_{κΩ_∞}$, where~$κ$ is the unique positive solution to the compatibility equation~$ρ c(κ)=j(κ)$.
\end{itemize}
\end{proposition}
\begin{proof}
Again, it is a direct application of Propositions~\ref{prop-lasalle-refined}-\ref{prop-stability-uniform} and Theorem~\ref{thm-strong-stability-anisotropic}.
\end{proof}

\begin{remark}~
\begin{itemize}
\item[(i)] At the critical point~$ρ=ρ_*$, the uniform equilibrium is stable (and for any initial condition sufficiently close to it, the solution converges exponentially fast to it), but the stability of the family of von Mises--Fisher distributions~$ρ_*M_{κ_*Ω}$, for~$Ω∈𝕊$, is unknown.
\item[(ii)] At the critical point~$ρ=n$, the family of von Mises--Fisher distributions of the form~$n M_{κ_cΩ}$, for~$Ω∈𝕊$ and where~$κ_c$ is the unique positive solution of~\eqref{eq-compatibility}, is stable.
For any initial condition sufficiently close to~$n M_{κ_cΩ}$ for some~$Ω∈𝕊$, there exists~$Ω_∞$ such that the solution converges exponentially fast to~$n M_{κ_cΩ_∞}$.
However, in this case, the stability of the uniform distribution~$f=n$ is unknown.
\item[(iii)] As previously, in the special case~$J_{f_0}=0$, the equation reduces to the heat equation and the solution converges to the uniform equilibrium.
\end{itemize}
\end{remark}

The order parameter~$c_1$ as a function of~$ρ$ (i.e.~$c_1(ρ)=c(κ)$ with~$ρ=\frac{j(κ)}{c(κ)}$) is depicted in Figure~\ref{fig-phase-diagram} for dimension~$2$ or~$3$.
The dashed lines corresponds to branches of equilibria which are unstable.

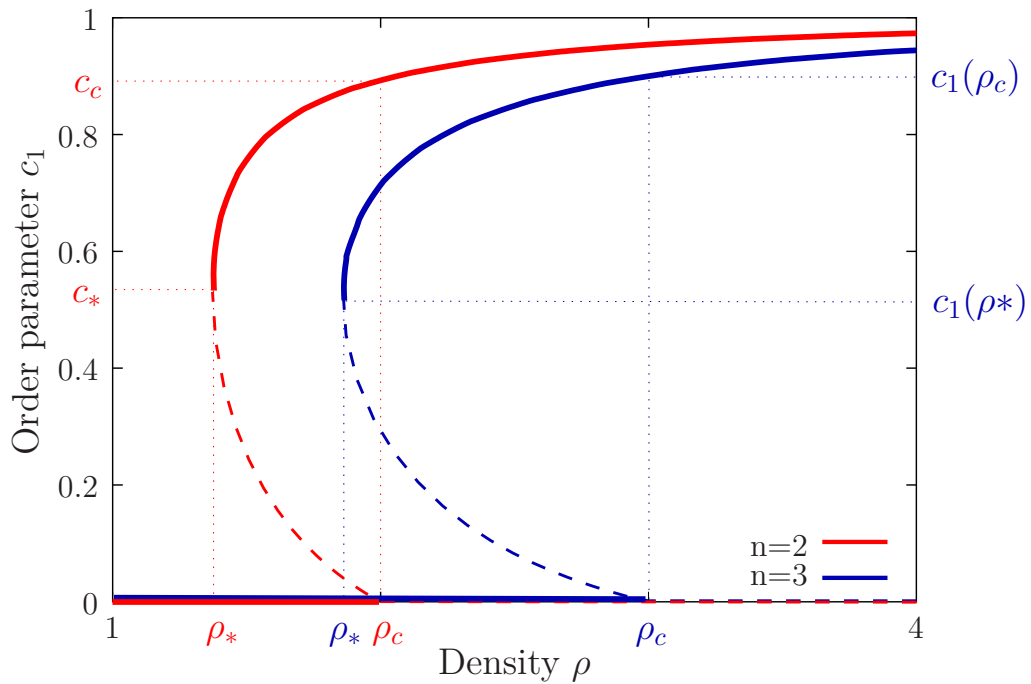
\begin{figure}[h]
\centering
\input{c_rho_hysteresis.pspdftex}
\caption{Phase diagram of the model with hysteresis, in dimensions~$2$ and~$3$.}
\label{fig-phase-diagram}
\end{figure}

The hysteresis phenomenon can be described by the hysteresis loop.
If the parameter~$ρ$ starts from a value less than~$ρ_*$, and increases slowly, the only stable distribution is initially the uniform distribution and it remains stable, until~$ρ$ reaches the critical value~$ρ_c$.
For~$ρ>ρ_c$, the only stable equilibria are the von Mises--Fisher distributions.
The order parameter then jumps from~$0$ to~$c_1(ρ_c)$.
If then the density~$ρ$ is further decreased slowly, the von Mises--Fisher distributions are stable until~$ρ$ reaches~$ρ_*$ back.
For~$ρ<ρ_*$, the only stable equilibrium is the uniform distribution, and the order parameter jumps from~$c_1(ρ_*)$ to~$0$.
The order parameter spans an oriented loop called hysteresis loop.

This hysteresis loop can be observed numerically at the kinetic level or at the particle level.
The plots of the order parameter for such numerical simulations are given by Figures~\ref{fig-numerics-kinetic-2d} and~\ref{fig-numerics-particle-2d}.
The details of the numerical simulations are provided in appendix~\ref{sec-numeric}. The key point to be able to perform these numerical simulations while varying the parameter~$ρ$ in time is to rescale the equation in order to see the parameter~$ρ$ as a coefficient of this new equation, and not anymore as the mass of the initial condition (normalized to be a probability measure).

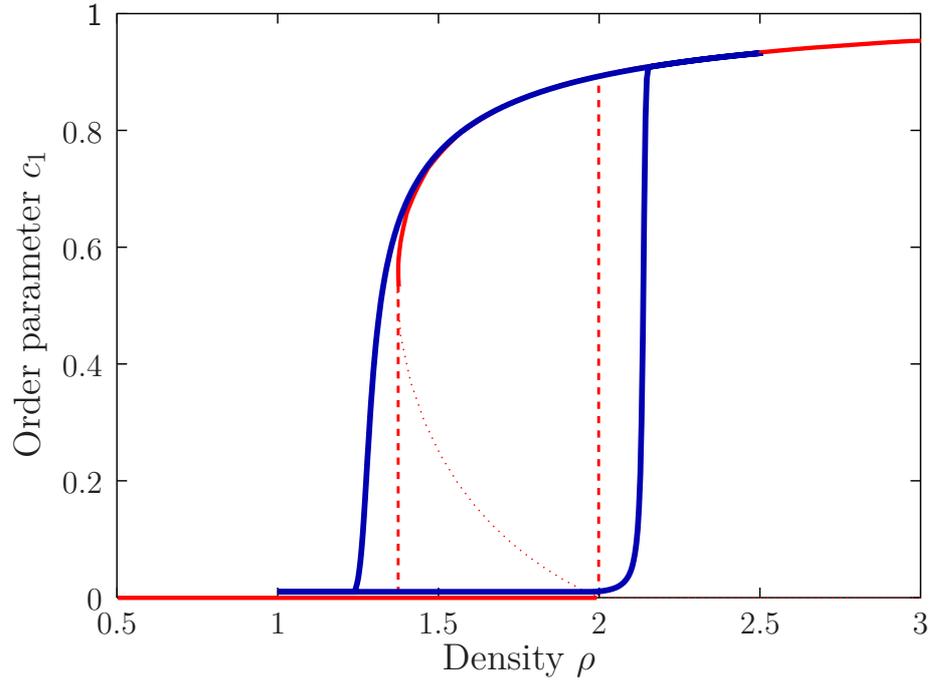
\begin{figure}[p]
\centering
\input{c_kinetic_500.pspdftex}
\caption{Hysteresis loop for the order parameter~$c_1$ in a numerical simulation of the homogeneous kinetic equation with time varying~$ρ$ (see~\eqref{eq-KFP-rho}), in dimension~$2$.
The red curve is the theoretical curve, the blue one corresponds to the simulation.}
\label{fig-numerics-kinetic-2d}
\end{figure}

\begin{figure}[p]
\centering
\input{c_particle_500.pspdftex}
\caption{Hysteresis loop for the order parameter~$c_1$ in a numerical simulation of the homogeneous particle model with varying~$ρ$ (see~\eqref{particle-hom-domega}-\eqref{particle-hom-Omega-J}), in dimension~$2$.}
\label{fig-numerics-particle-2d}
\end{figure}
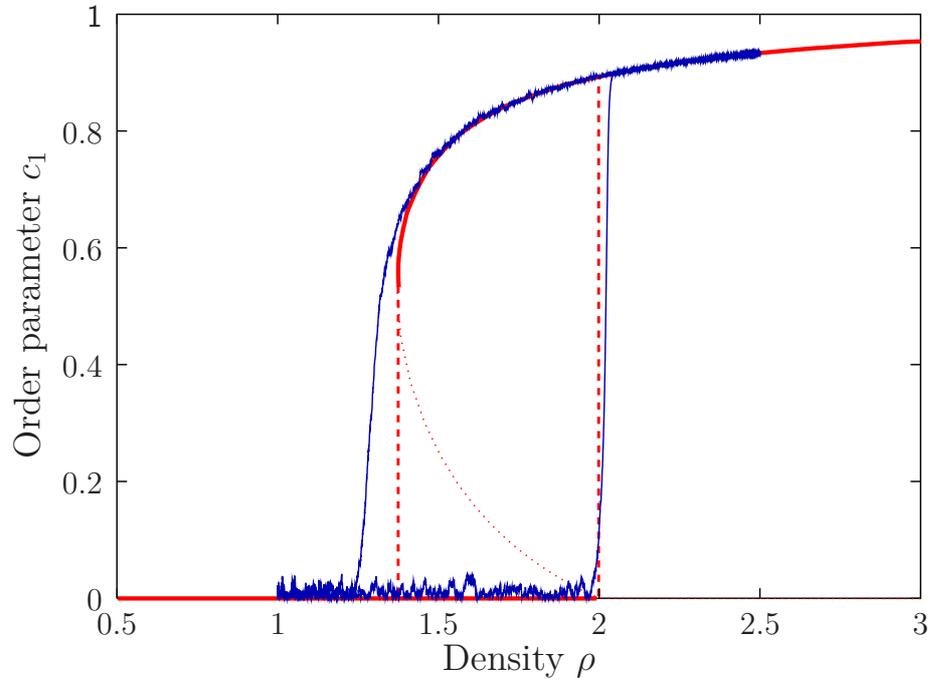

We can also obtain the theoretical diagrams for the free energy and the rates of convergences.
For this particular example, the free energies~$\mathcal F(ρ)$ and~$\mathcal F_κ$ (corresponding respectively to the uniform distribution and to the von Mises--Fisher distribution~$ρ M_{κΩ}$ for a positive solution~$κ$ of the compatibility equation~\eqref{eq-compatibility}) are given by
\begin{gather*}
\mathcal F(ρ)=ρ\ln ρ,\\
\begin{split}
\mathcal F_κ&=ρ\ln ρ + ⟨ρ\ln M_{κΩ}⟩_M - \frac12j^2-\frac13j^3\\
&=ρ\ln ρ  - ρ \ln ∫ e^{κ\cosθ}\d ω -\frac16(κ-j)+\frac23{jκ}.
\end{split}
\end{gather*}
The plots of these functions are depicted in dimensions~$2$ and~$3$ in the left plot of Figure~\ref{fig-free-energy}.
Since the functions are very close in some range, we magnify the difference~$\mathcal F_κ-\mathcal F(ρ)$ in the right plot of Figure~\ref{fig-free-energy}.
The dashed lines correspond to unstable branches of equilibria.

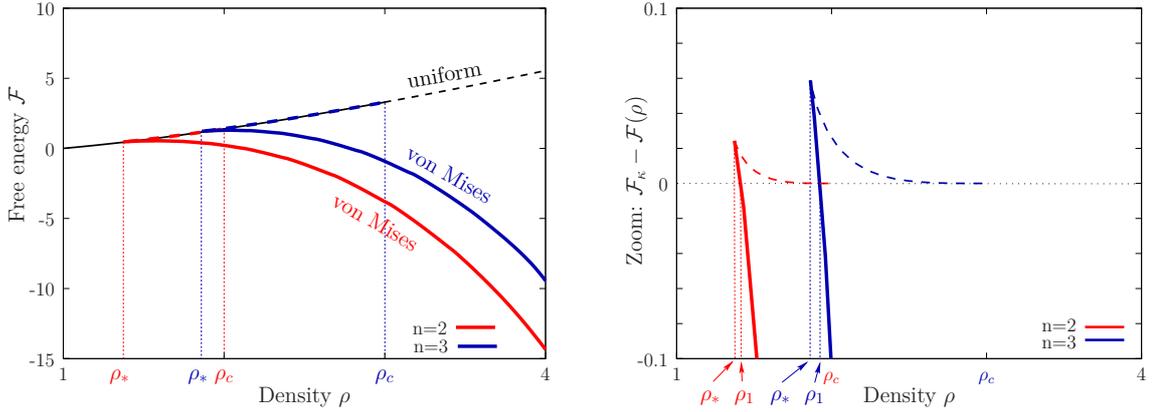
\begin{figure}[h]
\centering
\scalebox{.6}{\input{f_rho_hysteresis.pspdftex}}\hfill\scalebox{.6}{\input{diff_f_rho_hysteresis.pspdftex}}
\caption{Free energy levels of the different equilibria (left), and difference of free energies of anisotropic and uniform equilibria (right), as functions of the density, in dimensions~$2$ and~$3$. The dashed lines in the right picture corresponds to unstable equilibria. At the density~$ρ_1$, the free energies of the stable anisotropic and the uniform equilibria are the same.}
\label{fig-free-energy}
\end{figure}

We observe that the free energy of the unstable non-isotropic equilibria (in dashed line) is always greater than the one of the uniform distribution.
There exist~$ρ_1∈(ρ_*,ρ_c)$ and a corresponding solution~$κ_1$ of the compatibility solution~\eqref{eq-compatibility} (with~$κ_1>κ_*$, corresponding to a stable family of non-isotropic equilibria) such that~$\mathcal F_{κ_1}=\mathcal F(ρ_1)$.
If~$ρ<ρ_1$, the global minimizer of the free energy is the uniform distribution, while if~$ρ>ρ_1$, then the global minimum is reached for the family of stable von Mises--Fisher equilibria.
However, there is no easy way to assess the value of~$ρ_1$ numerically.
We observe that the stable von Mises--Fisher distribution has larger free energy than the uniform distribution if~$ρ<ρ_1$ and therefore consists of a metastable state.
On the contrary, the uniform distribution has larger free energy than the stable von Mises--Fisher distributions if~$ρ>ρ_1$ and now, consists of a metastable state.

The rates of convergence to the stable equilibria, following Proposition~\ref{prop-stability-uniform} and Theorem~\ref{thm-strong-stability-anisotropic}, are given by
\begin{gather*}
λ_0=(n-1)(1-\frac{ρ}n),\text{ for } ρ<ρ_c=n,\\
λ_κ=\frac{1}{1+j}Λ_κ\Big(1-\big(\frac1c-c-\frac{n-1}{κ}\big)j(1+2j)\Big),\text{ for } ρ>ρ_*,
\end{gather*}
where~$λ_0$ is the rate of convergence to the uniform distribution~$ρ$, and~$λ_κ$ is the rate of convergence to the stable family of von Mises--Fisher distributions~$ρ M_{κΩ}$, where~$κ$ is the unique solution of the compatibility condition~\eqref{eq-compatibility} such that~$κ>κ_*$.
Details for the numerical computation of the Poincaré constant~$Λ_κ$ are given in the appendix of~\cite{degond2012macroscopic}.
The computations in dimensions~$2$ and~$3$ are depicted in Figure~\ref{fig-rates}.
We observe that the rate of convergence to a given equilibrium is close to zero when~$ρ$ is close to the corresponding threshold for the stability of this equilibrium, and large when~$ρ$ is far from this threshold.
Moreover, the rate~$λ_κ$ of convergence to a von Mises--Fisher distribution is unbounded as~$ρ→∞$, while the rate~$λ_0$ of convergence to the uniform distribution is bounded by~$n-1$.

\begin{figure}[h]
\centering
\scalebox{.8}{\input{lambda_rho_hysteresis.pspdftex}}
\caption{Rates of convergence to both types of stable equilibria, as functions of the density~$ρ$, in dimensions~$2$ and~$3$.}
\label{fig-rates}
\end{figure}
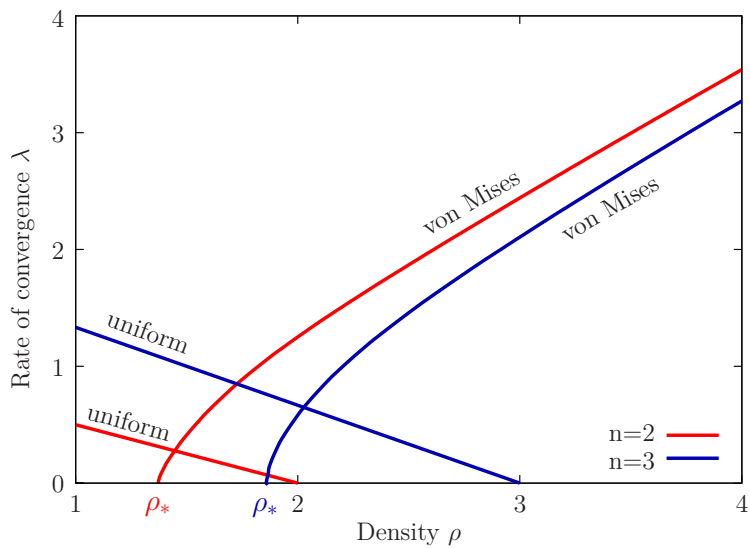

\section{Macroscopic models, hyperbolicity, and diffusivity}
\label{sec-macro-models}

We now go back to the spatially inhomogeneous system.
We want to investigate the hydrodynamic models that we can derive from the kinetic equation~\eqref{KFP-meanfield}.

\subsection{Scalings}
\label{subsec-hydro-scaling}

In order to understand the roles of the various terms, it is useful to introduce dimensionless quantities.
We set~$t_0$ the time unit and~$x_0=a\, t_0$ the space unit.
We assume that the range of the interaction kernel~$K$ is~$R$, meaning that we can write~$K(|x|) = \frac1{R^n}\widetilde K(\frac{|x|}{R})$ (we recall that~$K$ is normalized to~$1$, so we still have~$∫\widetilde K(|ξ|) \, \d ξ = 1$).
We also assume that~$\widetilde K$ has second moment of order~$1$, i.e.~$\widetilde K_2=\mathcal O(1)$, where
\begin{equation}
\label{def-K2}
\widetilde K_2=\frac1{2n}∫_{ℝ^n}\widetilde K(|ξ|)|ξ|^2 \, \d ξ.
\end{equation} 
We now introduce dimensionless variables~$\widetilde x = x/x_0$,~$\widetilde t = t/t_0$, and we make the change of variables~$\widetilde f (\widetilde x, ω, \widetilde t) = x_0^n  \, f(x_0 \widetilde x, ω, t_0 \widetilde t)$,~$\widetilde{\mathcal J}_{\widetilde{f}}(\widetilde x, \widetilde t) = x_0^n \mathcal J_f(x_0 \widetilde x, t_0 \widetilde t)/a$.
Finally, we introduce the dimensionless quantities: 
\[η = \frac{R}{x_0}, \quad  \hat{ν}(|\widetilde{\mathcal J}_{\widetilde{f}}|) = ν(|\mathcal J_f|)\,t_0, \quad \hat{τ}(|\widetilde{\mathcal J}_{\widetilde{f}}|) = τ(|\mathcal J_f|)\,t_0.
\label{eq:dimpar}\]
In this new system of coordinates, the system~\eqref{KFP-meanfield} is written as follows (we still use the notation~$J_{f}(x,t) = ∫_{𝕊}  f(x,ω,t) \,ω \, \dω$): 
\begin{equation}
\begin{cases}
\label{KFP-meanfield-dimless}
∂_t f +  ω · ∇_x f + \hat{ν}(|\mathcal J_f|) ∇_ω · (P_{ω^⊥} \bar{ω}_f \, f) = \hat{τ}(|\mathcal J_f|) Δ_ω f\\
\mathcal J_f(x,t) = (K_η* J_f)(x,t)\, \d y \, ,\quad \bar{ω}_f=\frac{\mathcal J_f}{|\mathcal J_f|} ,
\end{cases}
\end{equation}
where we have dropped the tildes for the sake of clarity, and where~$K_η$ is the rescaling of~$K$ given by
\begin{equation}
\label{def-K-eta}
K_η(x)=\frac{1}{η^n}K\big(\frac{x}{η} \big).
\end{equation}
Now, by fixing the relations between the three dimensionless quantities~\eqref{eq:dimpar}, we define the regime we are interested in.
We suppose that the diffusion and social forces are simultaneously large, while the range of the social interaction~$η$ tends to zero.
More specifically, we let~$ε ≪ 1$ be a small parameter and we assume that~$\hat{τ} = {\mathcal O}(1/ε)$ (large diffusion),~$\hat{ν} = {\mathcal O}(1/ε)$ (large social force).
In order to highlight these scaling assumptions, we define~$τ^\sharp$,~$ν^\sharp$, which are all~${\mathcal O}(1)$ and such that 
\begin{equation} \hat{τ} = \frac{1}{ε} τ^\sharp, \quad \quad \hat{ν} = \frac{1}{ε} ν^\sharp.
\label{eq:parscaled}
\end{equation}
Since~$η$ is supposed to be small, using the fact that~$K$ is isotropic, we can first get the Taylor expansion of~$\mathcal J_f$ with respect to~$η$, using~\eqref{def-K-eta} and~\eqref{def-K2}, when~$J_f$ is sufficiently smooth with respect to the space variable~$x$:
\begin{equation}
\label{expansion-mathcalJ}
\mathcal J_f= J_f + η^2 K_2 Δ_x J_f + \mathcal O(η^4).
\end{equation}
Inserting this expansion into~\eqref{KFP-meanfield-dimless}, and dropping all ``hats'' and ``sharps'', we are lead to:
\begin{equation}
\label{KFP-meanfield-eps}
ε(∂_t f +  ω · ∇_x f)  +K_2 η^2\big[∇_ω·(P_{ω^⊥} ℓ_{f} f) - m_fΔ_ω f\big]  =Q(f)+ \mathcal O(η^4),
\end{equation}
with
\begin{gather*}
Q(f)= - ν(|J_f|)∇_ω · (P_{ω^⊥} Ω_{f} \, f) + τ(|J_f|)Δ_ω f,\\
J_{f}(x,t) = ∫_{𝕊}  f(x,ω,t) \,ω \, \dω,\quad Ω_f=\frac{J_f}{|J_f|}\\
ℓ_f=\frac{ν(|J_f|)}{|J_f|} P_{Ω_f^⊥}\,Δ_x J_f + (Ω_f\cdotΔ_x J_f) \,ν'(|J_f|)Ω_f \,,\\
m_f=(Ω_f\cdotΔ_x J_f)\,τ'(|J_f|),
\end{gather*}
where the primes denote derivatives with respect to~$|J|$.
We recover the same definition of~$Q$ as in the spatially homogeneous setting~\eqref{def-Q}, and the additional terms~$ℓ_f$ and~$m_f$ do not depend on the velocity variable~$ω$ (they only depend on~$J_f$ and its Laplacian~$Δ_xJ_f$).

Our plan is now to investigate the hydrodynamic limit~$ε → 0$ in this model, within two different regimes for the range of the social interaction~$η$: firstly,~$η=\mathcal O(ε)$, and secondly,~$η=\mathcal O(√{ε})$.
We have seen in~\cite{degond2012hydrodynamic} that the second scaling allows us to retain some of the nonlocality of the social force in the macroscopic model, while the first one does not.
Indeed,~$ε$ corresponds to the characteristic distance needed by an individual to react to the social force, while~$η$ is the typical distance at which agents are able to detect their congeners.
The first scaling assumes that these two distances are of the same order of magnitude.
The second one corresponds to a large detection region compared to the reaction distance.
Which one of these two regimes is biologically relevant depends on the situation.
For instance, we can imagine that the first scaling will be more relevant in denser swarms because in such systems, far agents are concealed by closer ones.

In both cases, we will write~$f$ as~$f^ε$ to insist on the dependence on~$ε$.
The limiting behavior of the function~$f^ε$ as~$ε→0$ is supposed to be a local equilibrium for the operator~$Q$, as can be seen in~\eqref{KFP-meanfield-eps}.
Keeping in mind the results of the previous section on the spatial homogeneous version, we will assume that~$f^ε$ converges to a stable equilibrium of a given type in a given region.
Depending on the type of equilibrium (uniform distribution or von Mises--Fisher distribution), we will observe different behaviors.

\subsection{Disordered region: diffusion model}
We consider a region where~$f^ε$ converges as~$ε→0$ to a uniform equilibrium~$ρ(x,t)$ which is stable.
Therefore we must have~$ρ<ρ_c$.

We first remark that we can integrate~\eqref{KFP-meanfield-dimless} on the sphere to get the following conservation law (conservation of mass):
\begin{equation}
\label{mass-conservation-eps}
∂_tρ_{f^ε}+∇_x·J_{f^ε}=0.
\end{equation}
Therefore, if we suppose that the convergence is sufficiently strong,~$J_{f^ε}$ converges to~$0$, and we get~$∂_tρ=0$.

To obtain more precise information, we are then looking at the next order in~$ε$ in the Chapman--Enskog expansion method, in the same spirit as in the case of rarefied gas dynamics (see~\cite{degond2004macroscopic} for a review).
We obtain exactly the same model as in~\cite{degond2012macroscopic}.
We prove the following theorem:

\begin{theorem}
With both scalings~$η=\mathcal O(\sqrtε)$ and~$η=\mathcal O(ε)$, when~$ε$ tends to zero, the (formal) first order approximation to the solution of the rescaled mean-field model~\eqref{KFP-meanfield-eps} in a ``disordered region'' (where the solution locally converges to a stable uniform distribution) is given by
\begin{equation}
f^ε(x,ω,t)= ρ^ε(x,t)-ε\,\frac{n\,ω·∇_xρ^ε(x,t)}{(n-1)nτ_0\big(1-\frac{ρ^ε(x,t)}{ρ_c}\big)},
\label{eq:diff_correction}
\end{equation}
where the density~$ρ^ε$ satisfies the following diffusion equation
\begin{equation}
∂_t ρ^ε = \frac{ε}{(n-1)nτ_0}∇_x·\Big(\frac1{1-\frac{ρ^ε}{ρ_c}} \, ∇_xρ^ε\Big).
\label{diffusion_disorder}
\end{equation}
\label{thm:limit_disorder}
\end{theorem}

\begin{proof} We let~$ρ^ε=ρ_{f^ε}$ and write~$f^ε=ρ^ε+ε f_1^ε(x,ω,t)$ (so we have~$∫_𝕊f_1^ε\d ω=0$).
The assumption is that~$f_1^ε$ is a~$\mathcal O(1)$ quantity as~$ε→0$.
We then get
\begin{equation}
\label{J-f1}
J_{f^ε}=ε J_{f_1^ε},
\end{equation}
and the model~\eqref{KFP-meanfield-eps} becomes:
\begin{equation}
\label{KFP-diffusion-expansion}
ε(∂_tρ^ε+ω·∇_xρ^ε) = - ε ν'(0)∇_ω·(P_{ω^⊥}J_{f^ε_1}ρ^ε) + ε τ_0 Δ_ω f_1^ε +\mathcal O(ε^2)+\mathcal O(η^2ε).
\end{equation}
Additionally, using~\eqref{mass-conservation-eps} and~\eqref{J-f1}, we get that~$∂_tρ^ε=\mathcal O(ε)$.
Therefore we can put~$ε∂_tρ^ε$ into the~$O(ε^2)$ terms of~\eqref{KFP-diffusion-expansion} and get, in both scalings, 
\[Δ_ω f_1^ε=\frac1{τ_0}(∇_xρ^ε-ν'(0)(n-1)ρ^ε J_{f^ε_1})\cdotω + \mathcal O(ε).\]
We can solve this equation for~$f_1^ε$ and, together with the fact that~$ν'(0)=\frac{τ_0n}{ρ_c}$ (thanks to the definition~\eqref{def-rho-c} of~$ρ_c$), we get
\[f_1^ε=\big(-\frac1{τ_0(n-1)}∇_xρ^ε+\frac{nρ^ε}{ρ_c}J_{f^ε_1}\big)\cdotω + \mathcal O(ε).\]
This gives us, using~\eqref{eq-omega-omega}, that
\[J_{f_1^ε}=-\frac1{τ_0n(n-1)}∇_xρ^ε+\frac{ρ^ε}{ρ_c}J_{f^ε_1} + \mathcal O(ε),\]
which implies that~$f_1^ε=n J_{f_1^ε}\cdotω+ \mathcal O(ε)$ and that we have
\[J_{f_1^ε}=\frac{-1}{(n-1)nτ_0\big(1-\frac{ρ^ε}{ρ_c}\big)}∇_xρ^ε + \mathcal O(ε).\]
Therefore we obtain the expression~\eqref{eq:diff_correction} of~$f_1^ε$.
Moreover, inserting this expression of~$J_{f_1^ε}$ into the conservation of mass~\eqref{mass-conservation-eps} gives the diffusion model~\eqref{diffusion_disorder}.
\end{proof}

\begin{remark} The same remark was made in~\cite{degond2012macroscopic} (see Remark~3.1 therein): the expression of~$f^ε_1$, which is given by the~$\mathcal O(ε)$ term of~\eqref{eq:diff_correction} confirms that the approximation is only valid in the region where~$ρ_c-ρ^ε ≫ ε$.
The diffusion coefficient is only positive in the case where the uniform distribution is stable for the homogeneous model ($ρ^ε<ρ_c$) and it blows up as~$ρ^ε$ tends to~$ρ_c$, showing that the Chapman-Enskog expansion loses its validity.
\end{remark}

\subsection{Ordered region: hydrodynamic model and hyperbolicity}

\label{sec-qualitative}

We now turn to the derivation of a macroscopic model in a region where the local equilibria follow a given branch of stable von Mises--Fisher equilibria.
More precisely, we suppose that the function~$f^ε$ converges towards~$ρ(x,t)M_{κ(ρ(x,t))Ω(x,t)}$ in a given region, where~$κ(ρ)$ is a branch of solutions of the compatibility equation~\eqref{eq-compat2} defined for a given range of positive values of~$ρ$, and which correspond to stable equilibria (in the sense of Theorem~\ref{thm-strong-stability-anisotropic}).
This implies that~$κ$ is an increasing function of~$ρ$.
The goal is to prove the following theorem, which gives the evolution equations for~$ρ(x,t)$ and~$Ω(x,t)$, assuming that the convergence of~$f^ε$ is as smooth as needed.
\begin{theorem}
We suppose that~$f^ε$ converges as~$ε→0$ towards~$ρ(x,t)M_{κ(ρ(x,t))Ω(x,t)}$, for a positive density~$ρ(x,t)$ and an orientation~$Ω(x,t)∈𝕊$, and where~$ρ\mapstoκ(ρ)$ is a branch of solutions of the compatibility equation~\eqref{eq-compat2}.
We also suppose that the convergence of~$f^ε$ and of all its needed derivatives is sufficiently strong.
Then~$ρ$ and~$Ω$ satisfy the following system of partial differential equations:
\begin{gather}
\label{thm-evolution-rho}
∂_tρ+∇_x·(ρ c_1 Ω)=0,\\
\label{thm-evolution-omega}
ρ \, \left( ∂_t Ω + c_2\, (Ω · ∇_x) Ω \right) + Θ \, P_{Ω^⊥} ∇_x ρ=\mathcal K_2\,δ\,P_{Ω^⊥}Δ_x(ρ c_1 Ω),
\end{gather}
where~$\mathcal K_2$ is the scaling parameter corresponding to the limit of~$K_2\frac{η^2}{ε}$ as~$ε→0$, and where the coefficients~$c_1$,~$c_2$,~$Θ$ and~$δ$ are the following functions of~$ρ(x,t)$ (where the dependence on~$ρ$ for~$κ$ or on~$κ(ρ)$ for~$c$,~$\widetilde c$ or~$j$ is omitted when no confusion is possible):
\begin{align}
\label{def-c1}c_1(ρ)&=c(κ)\\
c_2(ρ)&=\widetilde c(κ)=\frac{⟨\cosθ\, h_κ(\cosθ)\sin^2θ⟩_{M}}{⟨h_κ(\cosθ)\sin^2θ⟩_{M}}=
 \frac{∫_0^π \cos θ\, h_κ(\cosθ)\, e^{κ\cosθ} \, 
\sin^{n} θ \, \dθ}{∫_0^π h_κ(\cosθ)\,e^{κ\cosθ} \, \sin^{n} θ \, \dθ}\, ,\label{def-c-tilda} \\
Θ(ρ)&=\frac1{κ}+ (\widetilde c -c)\,\frac{ρ}{κ}\,\frac{\d κ}{\d ρ}=\frac{n-\frac{κ}{c} + κ \widetilde{c} - 1 + \frac{κ}{j}\frac{\d j}{\d κ}}{κ\big(n-\frac{κ}{c} + κ c - 1 + \frac{κ}{j}\frac{\d j}{\d κ}\big)}\, ,
\label{def-Theta-hyperbolicity}\\
\label{def-delta}
δ(ρ)&=\frac{ν(j)}{c}\,\big(\frac{n-1}{κ}+\widetilde c\big).
\end{align}
The function~$h_κ$ is defined below at Proposition~\ref{prop-GCI}.
\end{theorem}
\begin{proof}
The first equation~\eqref{thm-evolution-rho} (for the time evolution of~$ρ$) can easily be derived from the conservation of mass~\eqref{mass-conservation-eps}, since the hypotheses imply that~$J_{f^ε}$ converges to~$ρ c(κ(ρ))Ω$, and thanks to~\eqref{expansion-mathcalJ}, we obtain:
\[∂_tρ+∇_x·(ρ c(κ(ρ))Ω)=0.\]
The main difficulty is the derivation of an equation of evolution for~$Ω$, since we do not have any conservation relation related to this quantity.
To this aim, the main tool consists in the determination of the so-called generalized collisional invariants, introduced by Degond and Motsch~\cite{degond2008continuum} to study the model corresponding to the case~$ν=τ=1$ and the scaling~$η=\mathcal O(ε)$ in our setting.
These generalized collisional invariants were then used successfully to derive the same kind of evolution equation for some variants of the model we are studying (see~\cite{degond2012macroscopic} when~$ν(|J|)=|J|$ and~$τ=1$,~\cite{frouvelle2012continuum} for the case where~$ν$ is a function of~$ρ$, and where the interaction is anisotropic,~\cite{degond2010macroscopic} for another type of alignment, based on the curvature of the trajectories, and~\cite{degond2012hydrodynamic} in the case of the second scaling~$η=\mathcal O(√{ε})$ when~$ν=τ=1$).

The idea is to introduce, for a given~$κ>0$ and~$Ω∈𝕊$, the operator~$L_{κΩ}$ (linearized operator of~$Q$):
\[L_{κΩ}(f) = Δ_ω f - κ∇_ω·(f P_{ω^⊥}Ω)=∇_ω · \left[ M_{κΩ} ∇_ω \left( \frac{f}{M_{κΩ}} \right) \right],\]
so that we have~$Q(f)=τ(|J_f|) L_{k(|J_f|)Ω_f}(f)$.
And we define the set of generalized collisional invariants~${\mathcal C}_{κΩ}$:
\begin{definition}
The set~${\mathcal C}_{κΩ}$ of generalized collisional invariants associated to~$κ∈ℝ$ and~$Ω∈𝕊$, is the following vector space:
\[{\mathcal C}_{κΩ}=\left\{ψ|∫_{ω ∈ 𝕊} L_{κΩ}(f) \, ψ \, \dω = 0 , \, ∀ f \text{ such that } \, P_{Ω^⊥}J_f=0 \right\}.\]
\end{definition}
Hence, if~$ψ$ is a collisional invariant in~${\mathcal C}_{κΩ}$, we have~$∫_{ω ∈ 𝕊} Q(f) \, ψ \, \dω = 0$ for any function~$f$ such that~$k(|J_f|)=κ$ and~$Ω_f=Ω$.

The determination of~${\mathcal C}_{κΩ}$ has been done in~\cite{frouvelle2012continuum}.
We recall the result here:
\begin{proposition}
There exists a positive function~$h_κ:[-1,1]→ℝ$ such that
\[{\mathcal C}_{κΩ}=\{ω↦h_κ(ω\cdotΩ)A\cdotω+C,\, C∈ℝ, A∈ℝ^n,\text{ with }A\cdotΩ=0\}.\]
\end{proposition}
\label{prop-GCI} 
More precisely,~$h_κ(\cosθ)=\frac{g_κ(θ)}{\sinθ}$, where~$g_κ$ is the unique solution in the space~$V$ of the elliptic problem~$\widetilde L_κ^*g(θ)=\sinθ$, where
\begin{gather*}
\widetilde L_κ^*g(θ)=-\sin^{2-n}θ e^{-κ\cosθ}\tfrac{\d}{\dθ}(\sin^{n-2}θ e^{κ\cosθ}g'(θ))+\tfrac{n-2}{\sin^2θ}g(θ),\\
V = \{ g \, | \,(n-2)(\sinθ)^{\frac n2-2} g ∈ L^2(0,π), \, (\sinθ)^{\frac n2-1}g ∈ H^1_0(0,π) \}.
\end{gather*}
We now define the vector-valued generalized collisional invariant associated to~$κ$ and~$Ω$ as
\[\vec{ψ}_{κΩ}(ω)= h_κ(ω\cdotΩ) P_{Ω^⊥}ω,\]
and we have the following useful property:
\[∀ f\text{ such that }k(|J_f|)=κ\text{ and }Ω_f=Ω,\quad ∫_{ω ∈ 𝕊} Q(f) \, \vec{ψ}_{κΩ} \, \dω = 0.\]
The next step consists in multiplying the rescaled kinetic model~\eqref{KFP-meanfield-eps} by~$\frac1{ε}\vec{ψ}_{κ^ε Ω_{f^ε}}$, with~$κ^ε=k(|J_{f^ε}|)$, and to integrate it on the sphere.
We get, 
\[P_{(Ω^ε)^⊥}\, \big( X^ε+K_2\frac{η^2}{ε} [Y^ε+Z^ε]\big)=\mathcal O\big(\frac{η^4}{ε}\big),\]
where
\begin{gather*} 
X^ε = ∫_{ω ∈ 𝕊} ( ∂_t f^ε + ω · ∇_x f^ε)\, h_{κ^ε}(ω · Ω^ε) \, ω \, \dω\,,\\
Y^ε = ∫_{ω ∈ 𝕊} ∇_ω·(P_{ω^⊥}ℓ_{f^ε}f^ε)\, h_{κ^ε}(ω · Ω^ε) \, ω \, \dω\, ,\\
Z^ε = ∫_{ω ∈ 𝕊} m_{f^ε}Δ_ω f^ε\, h_{κ^ε}(ω · Ω^ε) \, ω \, \dω\,.
\end{gather*}
Now we can pass to the limit~$ε→0$.
We denote by~$\mathcal K_2$ the limit of~$K_2\frac{η^2}{ε}$, which makes sense in both scalings (either~$η=\mathcal O(ε)$, and~$\mathcal K_2=0$, or~$η=\mathcal O(\sqrtε)$ and we suppose that~$\mathcal K_2$ is a positive quantity), and we obtain 
\begin{equation}
\label{p-x-y-z}
P_{Ω^⊥}\,( X+\mathcal K_2[Y+Z])=0,
\end{equation}
where, since we suppose that~$f^ε→ρ(x,t)M_{κ(ρ(x,t))Ω(x,t))}$, we have
\begin{gather*} 
X = ∫_{ω ∈ 𝕊} ( ∂_t ( ρ M_{κΩ}) + ω · ∇_x (ρ M_{κΩ}))\, h_κ(ω · Ω) \, ω \, \dω\,,\\ 
Y = ∫_{ω ∈ 𝕊} ∇_ω·(P_{ω^⊥}ℓ_{ρ M_{κΩ}}ρ M_{κΩ})\, h_{κ}(ω · Ω) \, ω \, \dω\,,\\
Z = ∫_{ω ∈ 𝕊} m_{ρ M_{κΩ}}Δ_ω(ρ M_{κΩ})\, h_{κ}(ω · Ω) \, ω \, \dω\,.
\end{gather*}
The computation of~$P_{Ω^⊥}\, X$ has been done in~\cite{frouvelle2012continuum}: we get
\begin{equation}
\label{PX}
P_{Ω^⊥}\, X = ⟨h_κ(\cosθ)\sin^2θ⟩_{M}\,ρ\, \frac{κ}{n-1} \, \left( ∂_t Ω + \widetilde c \, (Ω · ∇_x) Ω \right) + Θ \, P_{Ω^⊥} ∇_x ρ,
\end{equation}
where~$\widetilde c$ and~$Θ$ are given by~\eqref{def-c-tilda} and the first expression of~\eqref{def-Theta-hyperbolicity}.
We now compute~$P_{Ω^⊥}Y$ and~$P_{Ω^⊥}Z$.
Since~$∇_ω·(P_{ω^⊥}A)=-(n-1)A\cdotω$ for any vector~$A∈ℝ^n$, we get
\[Y=ρ∫_{ω ∈ 𝕊} [-(n-1+κ\,ω·Ω)\,ℓ_{ρ M_{κΩ}}\cdotω +κ\, ℓ_{ρ M_{κΩ}}\cdotΩ]  \, h_{κ}(ω · Ω) \, ω \,M_{κΩ}\, \dω\,.\]
Writing~$ω=\cosθ\, Ω+\sin θ\,v$ (orthogonal decomposition with~$v∈𝕊_{n-2}$), and using the fact that~$∫_{𝕊_{n-2}}v\,\d v=0$ and~$∫_{𝕊_{n-2}}v⊗ v\,\d v=\frac1{n-1}P_{Ω^⊥}$, we obtain
\begin{align}
\nonumber P_{Ω^⊥}\, Y&=-⟨h_κ(\cosθ)\sin^2θ⟩_{M}\,ρ\,\frac{n-1+κ\widetilde c}{n-1}\,P_{Ω^⊥}ℓ_{ρ M_{κΩ}}\\
\label{PY}&=-⟨h_κ(\cosθ)\sin^2θ⟩_{M}\,ν(ρ\,c)\,\frac{n-1+κ\widetilde c}{(n-1)c}P_{Ω^⊥}Δ_x(ρ cΩ).
\end{align}
Finally, since~$Δ_ω(M_{κΩ})$ is a function of~$\cosθ$, the same decomposition and argument shows that we have~$P_{Ω^⊥}Z=0$.
Inserting~\eqref{PX} and~\eqref{PY} into~\eqref{p-x-y-z} and dividing by~$\frac{κ}{n-1}⟨h_κ(\cosθ)\sin^2θ⟩_{M}$ ends the derivation of~\eqref{thm-evolution-omega}, with~$δ$ given by~\eqref{def-delta}.

We finally derive the expression of~$Θ$ given by the right-hand side of~\eqref{def-Theta-hyperbolicity}.
We differentiate the compatibility condition~$ρ c=j$ with respect to~$κ$ (in a given local branch of solutions), and we get~$c\frac{\d ρ}{\d κ}+ρ\frac{\d c}{\d κ}=\frac{\d j}{\d κ}$.
As was shown in~\cite{degond2012macroscopic}, we have~$\frac{\d c}{\d κ}=1-(n-1)\frac{c}{κ}-c^2$, therefore we get
\[\frac{κ}{ρ}\,\frac{\d ρ}{\d κ}=c\,\frac{κ}{j}\,\frac{\d ρ}{\d κ}=\frac{κ}{j}\frac{\d j}{\d κ}-ρ\,\frac{κ}{j}(1-(n-1)\frac{c}{κ}-c^2)=(n-1 +\frac{κ}{j}\frac{\d j}{\d κ}-\frac{κ}{c}+κ c),\]
and finally, thanks to the first expression of~\eqref{def-Theta-hyperbolicity}, we have
\begin{equation}
\label{Theta-kappa}
Θ=\frac1{κ}+\frac{\widetilde c-c}{n-1 +\frac{κ}{j}\frac{\d j}{\d κ}-\frac{κ}{c}+κ c}=
\frac{n-\frac{κ}{c} + κ \widetilde{c} - 1 + \frac{κ}{j}\frac{\d j}{\d κ}}{κ\big(n-\frac{κ}{c} + κ c - 1 + \frac{κ}{j}\frac{\d j}{\d κ}\big)},
\end{equation}
which gives an expression of~$Θ$ in terms of~$κ$ and the functions~$c$,~$\widetilde c$ and~$j$ only.
\end{proof}

\subsection{Hyperbolicity}
\label{sec-examples}
As shown in~\cite{degond2008continuum,degond2012hydrodynamic,frouvelle2012continuum}, and more precisely in~\cite{degond2012macroscopic} when the coefficients~$c_1$,~$c_2$ and~$Θ$ depend on~$ρ$, we have the following
\begin{proposition}
The self-organized hydrodynamic (SOH) model~\eqref{thm-evolution-rho}-\eqref{thm-evolution-omega} is hyperbolic if and only if~$Θ>0$.
\end{proposition}
In that case it has been proved in~\cite{degond2012hydrodynamic} that the SOH model is locally well-posed in dimension~$2$ (provided~$δ⩾0$) and in dimension~$3$ (for the scaling where~$\mathcal K_2=0$, with an additional condition for the orientation of the initial data).
Therefore, in this section, we study the sign of these coefficients in some generic situations.
\begin{conjecture}
\label{conj-c-ctild}
For all~$κ>0$, we have~$0<\widetilde c(κ)<c(κ)$.
Consequently, in the SOH model, we have~$δ>0$, and the SOH model is well-posed if it is hyperbolic.
\end{conjecture}
Numerically, this conjecture is clear, at least in dimension~$2$ and~$3$, as can be seen in Figure~$3$ of~\cite{degond2012macroscopic}.
We know it is true when~$κ$ is small or large, thanks to the asymptotics of~$c$ and~$\widetilde c$ given in~\cite{frouvelle2012continuum}:
\begin{align}
c&=
\begin{cases}
\tfrac1n κ-\tfrac1{n^2(n+2)} κ^3+O(κ^5)&\text{as }κ→0,\\
1- \frac{n-1}{2κ} +\frac{(n-1)(n-3)}{8κ^2} + O(κ^{-3})&\text{as }κ→∞,
\label{exp-c}
\end{cases}\\
\widetilde c&=
\begin{cases}
\frac{2n-1}{2n(n+2)}κ+O(κ^2)&\text{as }κ→0,\\
1-\frac{n+1}{2κ}+\frac{(n+1)(3n-7)}{24κ^2}+O(κ^{-3})&\text{as }κ→∞.
\end{cases}
\label{exp-ctild}
\end{align}
These asymptotics can also help us to know if the system is hyperbolic in various regimes.
In the next four propositions, we provide different cases where we can determine the hyperbolicity of the SOH model with simple assumptions on the behavior of the function~$k$.
The first result is about non-hyperbolicity in the neighborhood of the critical threshold~$ρ_*$ for a first order phase transition.
\begin{proposition}
\label{prop-hyperbolic-first}
Suppose that there is a first order phase transition with hysteresis as described by Proposition~\ref{prop-hysteresis}.
If Conjecture~\ref{conj-c-ctild} is true, then the SOH model associated to the branch of stable von Mises--Fisher equilibria (for~$ρ>ρ_*$) satisfies~$Θ(ρ)<0$ if~$ρ$ is sufficiently close to~$ρ_*$.
The SOH model is not hyperbolic.
\end{proposition}
\begin{proof}
  We have~$(\frac{j}{c})'(κ_*)=0$ and~$(\frac{j}{c})'(κ)>0$ for~$κ>κ_*$.
This gives that~$\frac{\d κ}{\d ρ}→+∞$ as~$κ→κ_*$, and then we use~\eqref{def-Theta-hyperbolicity} and Conjecture~\ref{conj-c-ctild} to get that~$Θ→-∞$ as~$κ→κ_*$ (for~$κ>κ_*$).
\end{proof}

We now provide the same type of proposition in the neighborhood of the critical threshold~$ρ_c$ in the case of a second order phase transition.
The following proposition gives a strong link between hyperbolicity and the critical exponent of a second order phase transition: it is hyperbolic when the critical exponent~$β$ is greater than~$\frac12$, and not hyperbolic when~$β<\frac12$ (this threshold value~$\frac12$ also corresponds to the lowest possible critical exponent which can appear in the case of the simple criterion given by Lemma~\ref{lemma-2nd-order}).
\begin{proposition}
\label{prop-hyperbolic-second}
We suppose that there is a second order phase transition as described by Proposition~\ref{prop-second-order}, and we consider the SOH model associated to the von Mises--Fisher equilibria (for~$ρ>ρ_c$).
We suppose furthermore that there is a critical exponent~$β$ as stated in Definition~\ref{def-critical-exponent}, and we assume that this estimation is also true at the level of the derivative:
\[\frac{\d κ}{\d ρ}\sim n α_0β(ρ-ρ_c)^{β-1},\text{ as }ρ\overset{>}{→}ρ_c.\]
Then
\begin{itemize}
\item[(i)] If~$β<\frac12$, then~$Θ(ρ)<0$ if~$ρ$ is sufficiently close to~$ρ_c$.
The SOH model is not hyperbolic.
\item[(ii)] If~$β>\frac12$, then~$Θ(ρ)>0$ if~$ρ$ is sufficiently close to~$ρ_c$.
The SOH model is hyperbolic.
\item[(iii)] If~$β=\frac12$ and~$α_0≠√{\frac{4(n+2)}{5n\,ρ_c}}$, then when~$ρ$ is sufficiently close to~$ρ_c$,~$Θ(ρ)$ is of the sign of~$√{\frac{4(n+2)}{5n\,ρ_c}}-α_0$.
\end{itemize}
\end{proposition}

\begin{proof}

We have~$κ(ρ)\sim n α_0(ρ-ρ_c)^β$, as~$ρ\overset{>}{→}ρ_c$.
So we get~$\frac{ρ}{κ}\frac{\d κ}{\d ρ}\sim β\frac{ρ_c}{ρ-ρ_c}$.
Finally, using~\eqref{exp-c}-\eqref{exp-ctild}, we get~$\widetilde c-c\sim -\frac{5}{2n(n+2)}κ$ as~$κ→0$.
We can then obtain an equivalent of~$Θ$ as~$ρ\overset{>}{→}ρ_c$, with~\eqref{def-Theta-hyperbolicity}:
\[Θ(ρ)\sim
\begin{cases}
\frac1{n\,α_0}(ρ-ρ_c)^{-β}&\text{if }β>\frac12\\
-\frac{5\,ρ_c\, α_0\, β}{2(n+2)}(ρ-ρ_c)^{β-1}&\text{if }β<\frac12\\
\big(\frac1{n\,α_0}-\frac{5\,ρ_c \,α_0}{4(n+2)}\big)\frac1{√{ρ-ρ_c}}&\text{if }β=\tfrac12,
\end{cases}\]
where the last expression is valid only if~$\frac1{nα_0}≠\frac{5ρ_c α_0}{4(n+2)}$.
The sign of~$Θ$ is then directly given by these equivalents, and this ends the proof.
\end{proof}
It is possible to refine Proposition~\ref{prop-critical-exponent} in order to have the critical exponent estimation on the level of the derivative, and then express the hyperbolicity of the system with the help of the expansion of~$k$ only.
In summary, we get the following proposition, the proof of which is left to the reader:
\begin{proposition}
\label{prop-refined-hyperbolicity-second}
If~$k$ satisfies:
\[k'(|J|)=\frac{n}{ρ_c}-a(q+1)|J|^q+o(|J|^q)\text{ as }|J|→∞,\]
then we have
\begin{itemize}
\item[(i)] if~$q<2$ and~$a>0$, the critical exponent is given by~$β=\frac1q$ and the SOH model is hyperbolic when~$ρ$ is sufficiently close to~$ρ_c$.
\item[(ii)] if~$q=2$ and~$a>\frac{n^2}{4\,ρ_c^3(n+2)}$, then~$β=\frac12$ and the SOH model is hyperbolic when~$ρ$ is sufficiently close to~$ρ_c$.
\item[(iii)] if~$q=2$ and~$-\frac{n^2}{ρ_c^3(n+2)}<a<\frac{n^2}{4\,ρ_c^3(n+2)}$, then~$β=\frac12$ and the SOH model is not hyperbolic for~$ρ$ close to~$ρ_c$.
\end{itemize}
\end{proposition}
We finally give a result about hyperbolicity when~$ρ$ is large, depending on the behavior of~$k$ as~$|J|→∞$.
\begin{proposition}
\label{prop-hyperbolic-large}
We suppose that~$k(|J|)\sim a\, |J|^b$ as~$|J|→∞$ (with~$a,b>0$), and that this equivalent is also true at the level of the derivative:~$k'(|J|)\sim a \,b\,|J|^{b-1}$.
We consider the SOH model associated to a branch of stable von Mises--Fisher equilibria.
\begin{itemize}
\item[(i)] If~$0<b<1$, then for~$ρ$ sufficiently large,~$Θ(ρ)<0$ and the SOH model is not hyperbolic.
\item[(ii)] If~$b>1$, then for~$ρ$ sufficiently large,~$Θ(ρ)>0$ and the SOH model is hyperbolic.
\item[(iii)] If~$b=1$, we have to make stronger hypotheses on~$k$.
For example, if we suppose that, as~$|J|→∞$, we have~$k(|J|)=a|J|+r+o(1)$ and~$k'(|J|)=a+o(|J|^{-1})$ with~$r≠\frac{n+1}6$, then for~$ρ$ sufficiently large,~$Θ(ρ)$ is of the sign of~$r-\frac{n+1}6$.
\end{itemize}
\end{proposition}
\begin{proof}
We first use the expansion~\eqref{exp-c} to get that~$n-1-\frac{κ}{c}+κ c\sim\frac{-n+1}{2κ}$ as~${κ→∞}$, and that~$\widetilde c-c=-\frac1{κ}+\frac{n-2}{3κ^2}+o(κ^{-2})$.
Using Hypothesis~\ref{hyp-increasing}, the assumptions become~$j(κ)\sim(\frac{κ}{a})^{\frac1b}$ and~$\frac{\d j}{\d κ}=[k'(j(κ))]^{-1}\sim (a\,b)^{-1}(\frac{κ}{a})^{\frac1b-1}$ as~$κ→∞$.
This gives~$ρ=\frac{j(κ)}{c(κ)}\sim(\frac{κ}{a})^{\frac1b}$ as~$κ→∞$, which can be inverted to get~$κ\sim a\,ρ^b$ as~$ρ→∞$.

Finally, for~$b≠1$, we get, with the left part of~\eqref{Theta-kappa}:
\[Θ(ρ)\sim \big(1-\frac1b)\frac1{a\,ρ^b}\text{ as }ρ→+∞.\]
This proves the first two points.
In the case where~$b=1$, we suppose that we have the expansions~$k(|J|)=a|J|+r+o(1)$ and~$k'(|J|)=a+o(|J|^{-1})$ as~$|J|→∞$.
Then~$j(κ)=\frac1a(κ-r)+o(1)$ and~$\frac{\d j}{\d κ}=\frac1a+o(κ^{-1})$.
And we finally get, using the left part of~\eqref{Theta-kappa}, as~$κ→+∞$:
\begin{align*}
Θ&= \frac1{κ}+\frac{-κ^{-1}+\frac{n-2}{3}κ^{-2}+o(κ^{-2})}{1+(r-\frac{n-1}2)κ^{-1}+o(κ^{-1})}\\
 &=(r-\tfrac{n+1}6)κ^{-2}+o(κ^{-2}),
\end{align*}
Since~$κ\sim a\, ρ$ as~$ρ → ∞$, we have~$Θ(ρ)\sim \frac1{a^2}(r-\tfrac{n+1}6)ρ^{-2}$ as~$ρ → ∞$ and this proves point~(iii).
\end{proof}
\begin{remark}
\label{remark-h-bounded}
The case~$b=0$ can also be treated if we assume Conjecture~\ref{conj-c-ctild}.
This corresponds to the case where~$k$ takes values on~$[0,κ_{max})$ with~$κ_{max}<∞$.
If furthermore we assume that its derivative satisfies~$k'\sim a|J|^{-b}$ (with~$b>1$ and~$a>0$) as~$|J|→∞$, then after the same kind of computations we get that~$Θ(ρ)→+∞$ as~$ρ→+∞$, and the system is hyperbolic.
\end{remark}

Let us now comment these results in the case of specific examples.
The case where~$ν(|J|)=|J|$ and~$τ(|J|)=1$ corresponds to the model studied in~\cite{degond2012macroscopic}.
It was shown to be non hyperbolic (numerically for all~$ρ>ρ_c=n$, and theoretically for~$ρ$ large or close to~$ρ_c$).
We now see that it corresponds to points~(iii) of Proposition~\ref{prop-hyperbolic-large} and Proposition~\ref{prop-refined-hyperbolicity-second}, which are the special cases separating hyperbolicity to non-hyperbolicity.
A really slight change in the function~$k$ in this model could easily lead to hyperbolicity, while nearly keeping the same phase transition phenomena, from the point of view of equilibria.

The case studied in Section~\ref{subsec-first} and leading to a first order phase transition corresponds to the function~$k(|J|)=|J|+|J|^2$.
Thanks to Proposition~\ref{prop-hyperbolic-first} and to~\ref{prop-hyperbolic-large}, we get that the corresponding SOH model is not hyperbolic in both regimes: when~$ρ$ is close to~$ρ_*$ and when~$ρ$ is sufficiently large.
Numerical computations of the coefficient~$Θ$ suggest that this is the case for all the values of~$ρ>ρ_*$ (at least in dimensions~$2$ and~$3$).

Finally, we are interested in the original model presented in~\cite{degond2008continuum}, where~$ν$ and~$τ$ are constant.
We remark that hypotheses~\ref{hyp-prop-chaos} and~\ref{hyp-increasing} do not cover this model, but we can see it as a limiting case of a regularized~$ν$ satisfying such hypotheses, such as~$ν^ε(|J|)=\frac{|J|}{ε+|J|}$.
In that case, we have~$ρ_c^ε=n\,ε\,τ_0$, and by Lemma~\ref{lemma-2nd-order} and Proposition~\ref{prop-critical-exponent}, we get that there is a second order phase transition with critical exponent~$1$.
Furthermore, with Remark~\ref{remark-h-bounded} and Proposition~\ref{prop-hyperbolic-second}, we get that the corresponding SOH model is hyperbolic when~$ρ$ is large or close to~$ρ_c$.
Figures~\ref{fig-phase-diagram-epsilon} and~\ref{fig-Theta-epsilon} correspond to the plots of the phase diagram (the order parameter~$c$) and of the function~$Θ$ for three different values of~$ε$, with~$τ_0=\frac13$ and~$n=2$.
We observe that the system is always hyperbolic.

\begin{figure}[h]
\centering
\input{c_epsilon.pspdftex}
\caption{Order parameter~$c_1$, as function of the density~$ρ$, in dimension~$2$, for the regularized model.}
\label{fig-phase-diagram-epsilon}
\end{figure}
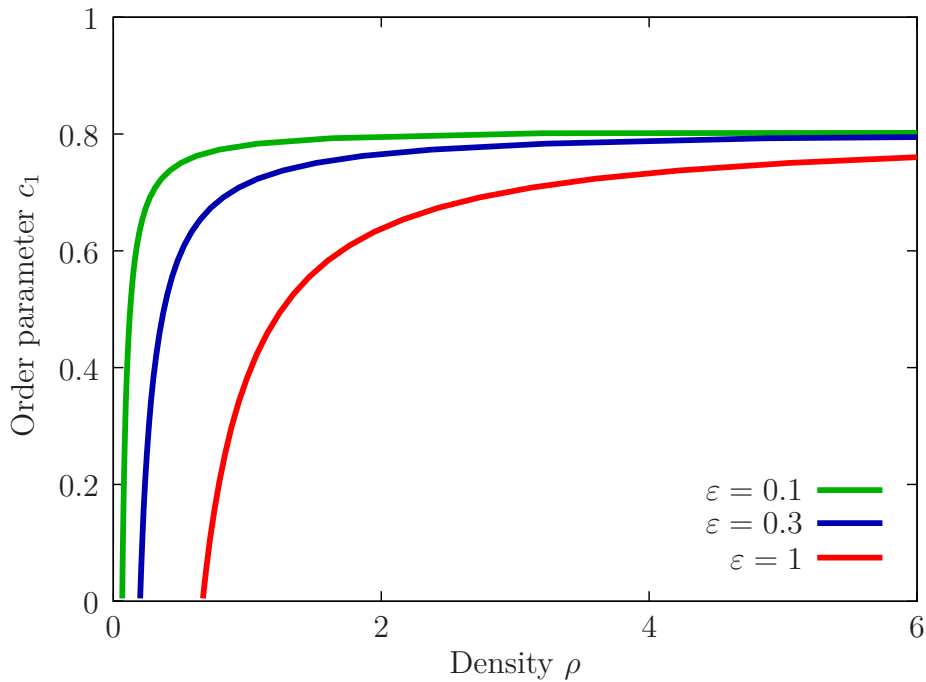 
\begin{figure}[!h]
\centering
\input{Theta_epsilon.pspdftex}
\caption{Coefficient~$Θ$, as function of the density~$ρ$, in dimension~$2$, for the regularized model.}
\label{fig-Theta-epsilon}
\end{figure}
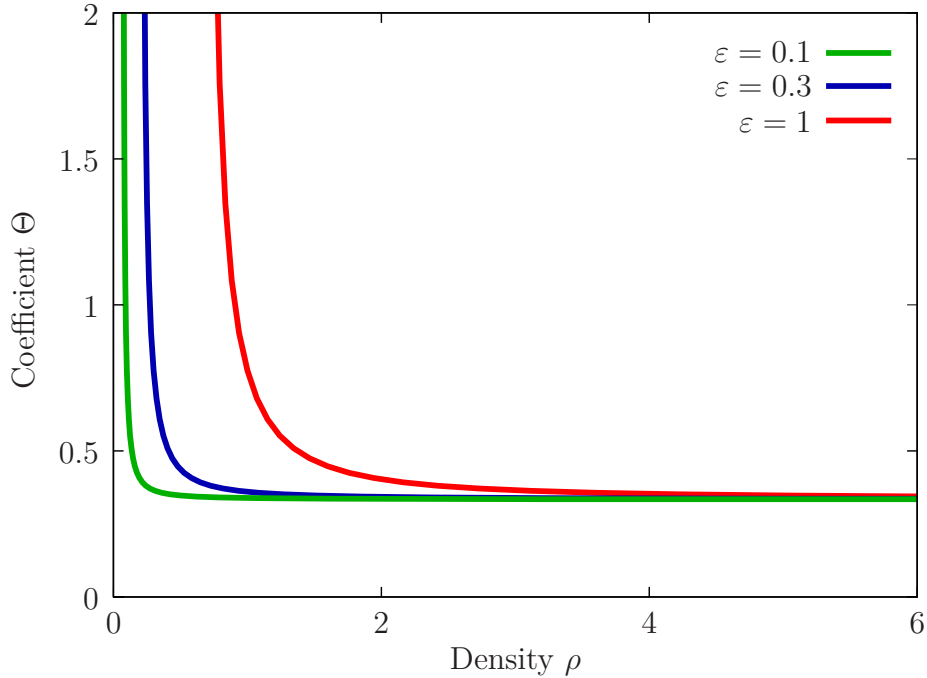 

We get the same conclusion for a regularization given by~$ν^ε(|J|)=\frac{|J|}{√{ε^2+|J|^2}}$, with a critical exponent~$β=\frac12$ this time, and~$k$ satisfies the condition~(ii) of Proposition~\eqref{prop-refined-hyperbolicity-second} if~$τ_0>\frac1{√{2n(n+2)}}$.
This gives a practical example of a second order phase transition with the minimal critical exponent such that the associated SOH model is hyperbolic in the neighborhood of the threshold~$ρ_c$ (indeed, in that case, thanks to Proposition~\eqref{prop-hyperbolic-second}, we must have~$β⩾\frac12$).

\section{Conclusion}
\label{sec-conclu}

In this work, we have provided a comprehensive and rigorous description of phase transitions for kinetic models describing self-propelled particles interacting through alignment.
We have highlighted how their behavior results from the  competition between alignment and noise.
We have considered a general framework, where both the alignment frequency and noise intensity depend on a measure of the local alignment.
We have shown that, in the spatially homogeneous case, the phase transition features (number and nature of equilibria, stability, convergence rate, phase diagram, hysteresis) are totally encoded in the function obtained by taking the quotient of the alignment and noise intensities as functions of the local alignment.
The phase transitions dealt with in this paper belong to the class of spontaneous symmetry-breaking phase transitions that also appear in many physics systems such as ferromagnetism, liquid crystals, polymers, etc.
We have also provided the derivation of the macroscopic models (of hydrodynamic or diffusion types) that can be obtained from the knowledge of the stable equilibria and classified their hyperbolicity.
In particular, we have provided a strong link between the critical exponent in the second order phase transition and the hyperbolicity of the hydrodynamic model.
In the future, we will investigate how the hydrodynamic and diffusion regimes can be spatially connected through domain walls and find the dynamic of these domain walls.

\appendix
\section{Appendix: numerical methodology for the hysteresis simulation}
\label{sec-numeric}

In this appendix, we give more details on the computation of the hysteresis loop provided in section~\ref{subsec-first}.
In order to highlight the role of the density~$ρ$ as the key parameter for the phase transition, we introduce the probability measure~$\widetilde f=\frac{f}{ρ}$ and we rewrite the homogeneous kinetic equation~\eqref{KFP-homogeneous} in terms of~$\widetilde f$.
We get
\begin{equation}
\label{eq-KFP-rho}
∂_t\widetilde f = τ(ρ|J_{\widetilde f}|)Δ_ω\widetilde f - ν(ρ|J_{\widetilde f}|)∇_ω·(P_{ω^⊥}Ω_{\widetilde f}\widetilde f).
\end{equation}
When~$ρ$ is constant, this equation is equivalent to~\eqref{KFP-homogeneous}.
We will now consider~$ρ$ as a parameter of the equation (and not anymore a parameter for the mass of initial condition, since~$\widetilde f$ is now a probability measure), but the long time behavior (equilibria, stability, convergence) is still given by this parameter~$ρ$.
Finally, we let~$ρ$ vary slowly with time (compared to the time scale of convergence to equilibrium, given by Figure~\ref{fig-rates}), as we expect it would be the case in the spatial inhomogeneous framework given by the kinetic equation~\eqref{KFP-meanfield}.

\subsection{Simulation at the kinetic level.}

Let us now present how the numerical simulations of the system~\eqref{eq-KFP-rho} in dimension~$n=2$, depicted in Figure~\ref{fig-numerics-kinetic-2d}, have been obtained.
We start with an initial condition which is a small perturbation of the uniform distribution, and we take a varying parameter of the form~$ρ=1.75-0.75\cos(\frac{π}T t)$, with~$T=500$.
We use a standard central finite different scheme (with~$100$ discretization points), implicit in time (with a time step of~$0.01$).
The only problem with this approach is that the solution converges strongly to the uniform distribution for~$ρ<ρ_c$. 
So after passing~$ρ_c$, the linear rate of increase for~$J_{\widetilde f}$ is given by~$\frac{ρ}{ρ_c}-1$, by virtue of~\eqref{ODE-J-reduced}, and is very slow when~$ρ$ is close to~$ρ_c$.
So since~$J_{\widetilde f}$ is initially very small when passing the threshold~$ρ=ρ_c$, the convergence to the stable von Mises--Fisher distribution is very slow.
Two ideas can be used to overcome this problem: either injecting noise in the system, or more efficiently, adding a threshold~$ε$ and strengthening~$|J_{\widetilde f}|$ when~$∥\widetilde f-1∥_∞⩽ε$, replacing~$\widetilde f$ at the end of such a step by
\[ \widetilde f + \max(0,ε-∥\widetilde f-1∥_∞)\,Ω_{\widetilde f}\cdotω.\]
We note that after this transformation, we still have~$∥\widetilde f-1∥_∞⩽ε$ if it was the case before applying the transformation.

Figure~\ref{fig-numerics-kinetic-2d} depicts the result of a numerical simulation with a threshold~$ε=0.02$.
We clearly see this hysteresis cycle, which agrees very well with the theoretical diagram.
The jumps at~$ρ=ρ_*$ and~$ρ=ρ_c$ are closer to the theoretical jumps when~$T$ is very large.

\subsection{Simulations at the particle level.}

Now, since the kinetic equation~\eqref{KFP-meanfield} comes from a limit of a particle system, we are interested in observing this hysteresis phenomenon numerically at the level of the particle system, where noise is already present in the model, since it is a system of stochastic differential equations.

As for~\eqref{KFP-meanfield}, it is easy to derive the mean-field equation~\eqref{eq-KFP-rho}, in the spirit of Proposition~\ref{prop-meanfield-limit}, from the following system:
\begin{align}
&\d ω_i=ν(ρ|J|) P_{ω_i^⊥}\,Ω \,\d t + √{2τ(ρ|J|)}P_{ω_i^⊥}∘\d B^i_t,\label{particle-hom-domega}\\
&Ω=\frac{J}{|J|}, \quad J = \frac{1}N\sum_{i=1}^N \,ω_i.
\label{particle-hom-Omega-J}
\end{align}
Here, once again, the parameter~$ρ$ is a parameter of the equation, which can be variable in time. We perform numerical simulations of this system for a large number of particles, with~$ρ$ varying as in the numerical simulation of the kinetic model.
As before, we start with a initial condition which consists of~$N=10000$ particles uniformly distributed on~$𝕊_1$, and we take~$ρ=1.75-0.75\cos(\frac{π}T t)$, with~$T=500$.
We use a splitting method for the random and the deterministic parts of this equation (with a time step of~$0.01$).
We then plot the order parameter~$c$, given by~$|J|$.
The result is given in Figure~\ref{fig-numerics-particle-2d}.

Let us remark that, thanks to the central limit theorem, the mean~$J$ of~$N$ vectors uniformly distributed on the circle has a law equivalent to a multivariate normal distribution in~$ℝ^2$ centered at~$0$, and with covariance matrix~$\frac{1}{2N}{\mathrm I}_2$.
Therefore~$|J|$ is equivalent to a Rayleigh distribution of parameter~$\frac{1}{√{2N}}$, and so the mean of~$|J|$ is equivalent to~$\frac{√π}{2√N}$.
In our case, that gives a mean of~$|J|$ of approximately~$0.009$, of the same order as in the previous section, since the threshold~$ε$ ensures that, when~$\widetilde f$ is close to the uniform distribution~$|J_{\widetilde f}|\approx\frac{ε}2$ with~$ε=0.02$.


\input{PhaseTransHystHyper.bbl}
\end{document}

%% file: shapes_rho_kappa.pspdftex
\begin{picture}(0,0)%
\includegraphics{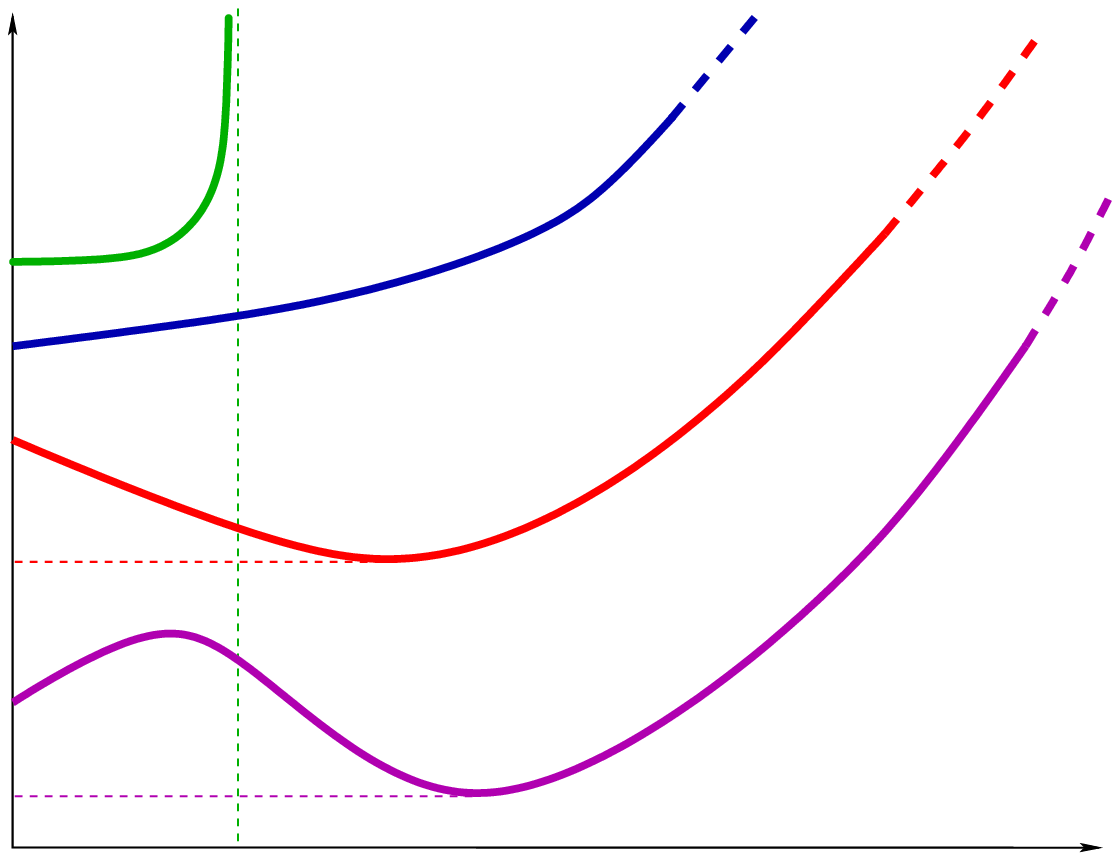}%
\end{picture}%
\setlength{\unitlength}{2368sp}%
\begingroup\makeatletter\ifx\SetFigFont\undefined%
\gdef\SetFigFont#1#2#3#4#5{%
  \reset@font\fontsize{#1}{#2pt}%
  \fontfamily{#3}\fontseries{#4}\fontshape{#5}%
  \selectfont}%
\fi\endgroup%
\begin{picture}(9164,7317)(1691,-7318)
\put(2026,-7136){\makebox(0,0)[b]{\smash{{\SetFigFont{12}{14.4}{\familydefault}{\mddefault}{\updefault}0}}}}
\put(1876,-6886){\makebox(0,0)[rb]{\smash{{\SetFigFont{12}{14.4}{\familydefault}{\mddefault}{\updefault}0}}}}
\put(3826,-7186){\makebox(0,0)[b]{\smash{{\SetFigFont{14}{16.8}{\familydefault}{\mddefault}{\updefault}{\color[rgb]{0,.69,0}$\kappa_{max}$}%
}}}}
\put(10201,-7186){\makebox(0,0)[b]{\smash{{\SetFigFont{14}{16.8}{\familydefault}{\mddefault}{\updefault}$\kappa$}}}}
\put(1951,-2236){\makebox(0,0)[rb]{\smash{{\SetFigFont{14}{16.8}{\familydefault}{\mddefault}{\updefault}{\color[rgb]{0,.69,0}$\rho_c=\rho_*$}%
}}}}
\put(1951,-2911){\makebox(0,0)[rb]{\smash{{\SetFigFont{14}{16.8}{\familydefault}{\mddefault}{\updefault}{\color[rgb]{0,0,.69}$\rho_c=\rho_*$}%
}}}}
\put(1951,-3661){\makebox(0,0)[rb]{\smash{{\SetFigFont{14}{16.8}{\familydefault}{\mddefault}{\updefault}{\color[rgb]{1,0,0}$\rho_c$}%
}}}}
\put(1951,-4636){\makebox(0,0)[rb]{\smash{{\SetFigFont{14}{16.8}{\familydefault}{\mddefault}{\updefault}{\color[rgb]{1,0,0}$\rho_*$}%
}}}}
\put(1951,-5761){\makebox(0,0)[rb]{\smash{{\SetFigFont{14}{16.8}{\familydefault}{\mddefault}{\updefault}{\color[rgb]{.69,0,.69}$\rho_c$}%
}}}}
\put(1951,-6511){\makebox(0,0)[rb]{\smash{{\SetFigFont{14}{16.8}{\familydefault}{\mddefault}{\updefault}{\color[rgb]{.69,0,.69}$\rho_*$}%
}}}}
\put(1918,-694){\makebox(0,0)[rb]{\smash{{\SetFigFont{14}{16.8}{\familydefault}{\mddefault}{\updefault}$\frac{j(\kappa)}{c(\kappa)}$}}}}
\end{picture}%

%% file: rho_kappa_hysteresis.pspdftex
\begin{picture}(0,0)%
\includegraphics{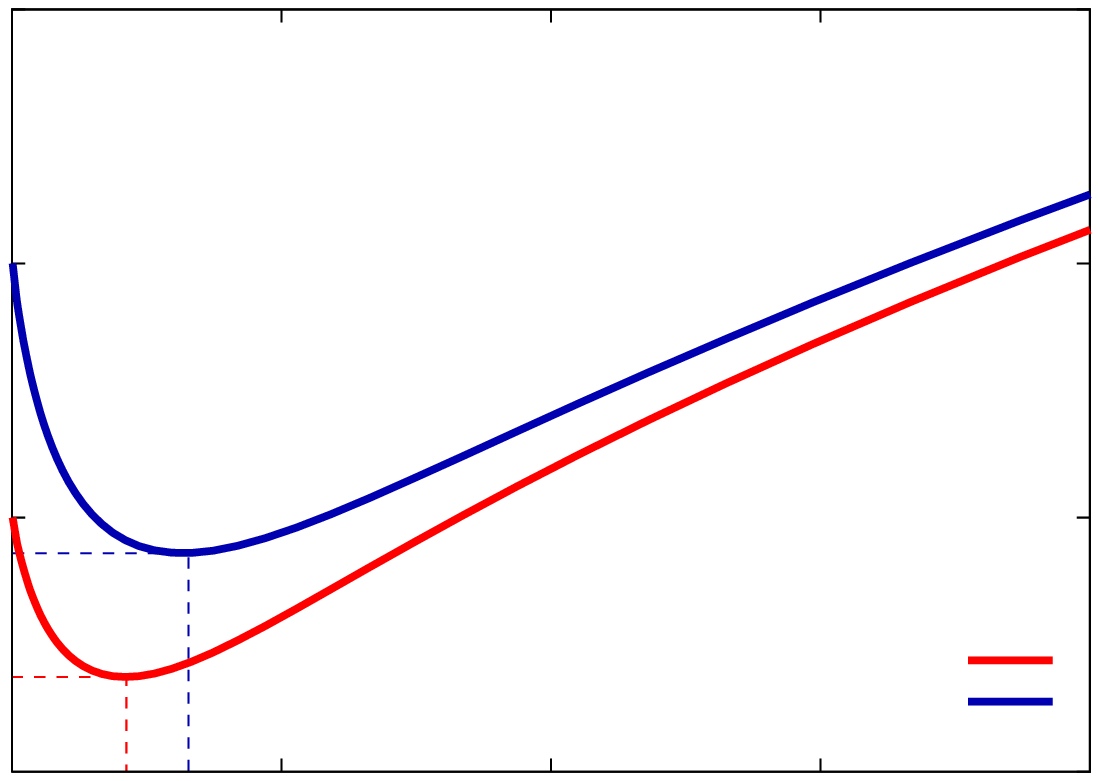}%
\end{picture}%
\setlength{\unitlength}{2368sp}%
\begingroup\makeatletter\ifx\SetFigFont\undefined%
\gdef\SetFigFont#1#2#3#4#5{%
  \reset@font\fontsize{#1}{#2pt}%
  \fontfamily{#3}\fontseries{#4}\fontshape{#5}%
  \selectfont}%
\fi\endgroup%
\begin{picture}(9143,7073)(1391,-7620)
\put(9225,-5971){\makebox(0,0)[rb]{\smash{{\SetFigFont{12}{14.4}{\familydefault}{\mddefault}{\updefault}n=2}}}}
\put(9225,-6326){\makebox(0,0)[rb]{\smash{{\SetFigFont{12}{14.4}{\familydefault}{\mddefault}{\updefault}n=3}}}}
\put(6038,-7511){\makebox(0,0)[b]{\smash{{\SetFigFont{12}{14.4}{\familydefault}{\mddefault}{\updefault}$\kappa$}}}}
\put(6038,-7136){\makebox(0,0)[b]{\smash{{\SetFigFont{12}{14.4}{\familydefault}{\mddefault}{\updefault}6}}}}
\put(8194,-7136){\makebox(0,0)[b]{\smash{{\SetFigFont{12}{14.4}{\familydefault}{\mddefault}{\updefault}9}}}}
\put(10350,-7136){\makebox(0,0)[b]{\smash{{\SetFigFont{12}{14.4}{\familydefault}{\mddefault}{\updefault}12}}}}
\put(3882,-7136){\makebox(0,0)[b]{\smash{{\SetFigFont{12}{14.4}{\familydefault}{\mddefault}{\updefault}3}}}}
\put(1726,-7136){\makebox(0,0)[b]{\smash{{\SetFigFont{12}{14.4}{\familydefault}{\mddefault}{\updefault}0}}}}
\put(1576,-6886){\makebox(0,0)[rb]{\smash{{\SetFigFont{12}{14.4}{\familydefault}{\mddefault}{\updefault}1}}}}
\put(1576,-4853){\makebox(0,0)[rb]{\smash{{\SetFigFont{12}{14.4}{\familydefault}{\mddefault}{\updefault}2}}}}
\put(1576,-2820){\makebox(0,0)[rb]{\smash{{\SetFigFont{12}{14.4}{\familydefault}{\mddefault}{\updefault}3}}}}
\put(1576,-787){\makebox(0,0)[rb]{\smash{{\SetFigFont{12}{14.4}{\familydefault}{\mddefault}{\updefault}4}}}}
\put(2558,-7111){\makebox(0,0)[b]{\smash{{\SetFigFont{12}{14.4}{\familydefault}{\mddefault}{\updefault}{\color[rgb]{1,0,0}$\kappa_*$}%
}}}}
\put(3188,-7111){\makebox(0,0)[b]{\smash{{\SetFigFont{12}{14.4}{\familydefault}{\mddefault}{\updefault}{\color[rgb]{0,0,.69}$\kappa_*$}%
}}}}
\put(1651,-5161){\makebox(0,0)[rb]{\smash{{\SetFigFont{12}{14.4}{\familydefault}{\mddefault}{\updefault}{\color[rgb]{0,0,.69}$\rho_*$}%
}}}}
\put(1651,-6136){\makebox(0,0)[rb]{\smash{{\SetFigFont{12}{14.4}{\familydefault}{\mddefault}{\updefault}{\color[rgb]{1,0,0}$\rho_*$}%
}}}}
\put(1501,-3811){\makebox(0,0)[rb]{\smash{{\SetFigFont{12}{14.4}{\familydefault}{\mddefault}{\updefault}{\color[rgb]{0,0,0}$\dfrac{j(\kappa)}{c(\kappa)}$}%
}}}}
\end{picture}%

%% file: c_rho_hysteresis.pspdftex
\begin{picture}(0,0)%
\includegraphics{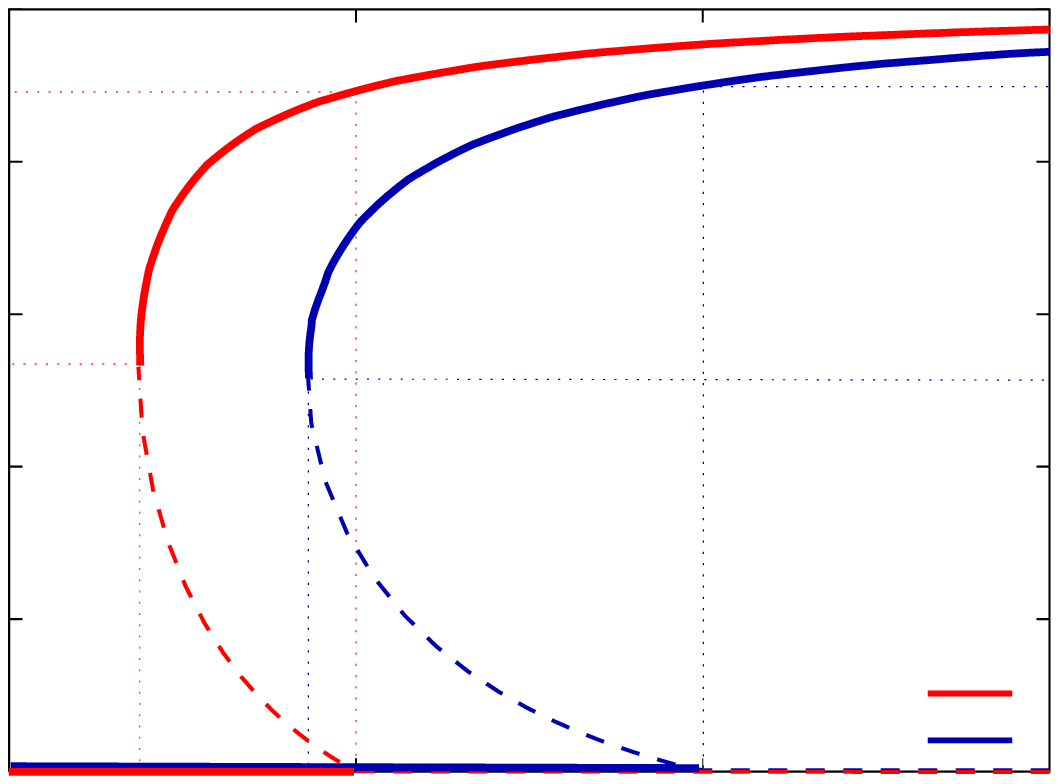}%
\end{picture}%
\setlength{\unitlength}{2368sp}%
\begingroup\makeatletter\ifx\SetFigFont\undefined%
\gdef\SetFigFont#1#2#3#4#5{%
  \reset@font\fontsize{#1}{#2pt}%
  \fontfamily{#3}\fontseries{#4}\fontshape{#5}%
  \selectfont}%
\fi\endgroup%
\begin{picture}(9601,7096)(930,-7643)
\put(2026,-7136){\makebox(0,0)[b]{\smash{{\SetFigFont{12}{14.4}{\familydefault}{\mddefault}{\updefault}1}}}}
\put(10350,-7136){\makebox(0,0)[b]{\smash{{\SetFigFont{12}{14.4}{\familydefault}{\mddefault}{\updefault}4}}}}
\put(9225,-6586){\makebox(0,0)[rb]{\smash{{\SetFigFont{12}{14.4}{\familydefault}{\mddefault}{\updefault}n=3}}}}
\put(9225,-6286){\makebox(0,0)[rb]{\smash{{\SetFigFont{12}{14.4}{\familydefault}{\mddefault}{\updefault}n=2}}}}
\put(1876,-5666){\makebox(0,0)[rb]{\smash{{\SetFigFont{12}{14.4}{\familydefault}{\mddefault}{\updefault}0.2}}}}
\put(1876,-6886){\makebox(0,0)[rb]{\smash{{\SetFigFont{12}{14.4}{\familydefault}{\mddefault}{\updefault}0}}}}
\put(1876,-4446){\makebox(0,0)[rb]{\smash{{\SetFigFont{12}{14.4}{\familydefault}{\mddefault}{\updefault}0.4}}}}
\put(1876,-3227){\makebox(0,0)[rb]{\smash{{\SetFigFont{12}{14.4}{\familydefault}{\mddefault}{\updefault}0.6}}}}
\put(1876,-2007){\makebox(0,0)[rb]{\smash{{\SetFigFont{12}{14.4}{\familydefault}{\mddefault}{\updefault}0.8}}}}
\put(4441,-7156){\makebox(0,0)[b]{\smash{{\SetFigFont{14}{16.8}{\familydefault}{\mddefault}{\updefault}{\color[rgb]{0,0,.69}$\rho_*$}%
}}}}
\put(3166,-7141){\makebox(0,0)[b]{\smash{{\SetFigFont{14}{16.8}{\familydefault}{\mddefault}{\updefault}{\color[rgb]{1,0,0}$\rho_*$}%
}}}}
\put(7606,-7156){\makebox(0,0)[b]{\smash{{\SetFigFont{14}{16.8}{\familydefault}{\mddefault}{\updefault}{\color[rgb]{0,0,.69}$\rho_c$}%
}}}}
\put(4876,-7141){\makebox(0,0)[b]{\smash{{\SetFigFont{14}{16.8}{\familydefault}{\mddefault}{\updefault}{\color[rgb]{1,0,0}$\rho_c$}%
}}}}
\put(6188,-7511){\makebox(0,0)[b]{\smash{{\SetFigFont{14}{16.8}{\familydefault}{\mddefault}{\updefault}Density $\rho$}}}}
\put(1921,-1441){\makebox(0,0)[rb]{\smash{{\SetFigFont{14}{16.8}{\familydefault}{\mddefault}{\updefault}{\color[rgb]{1,0,0}$c_c$}%
}}}}
\put(1921,-3631){\makebox(0,0)[rb]{\smash{{\SetFigFont{14}{16.8}{\familydefault}{\mddefault}{\updefault}{\color[rgb]{1,0,0}$c_*$}%
}}}}
\put(1876,-787){\makebox(0,0)[rb]{\smash{{\SetFigFont{12}{14.4}{\familydefault}{\mddefault}{\updefault}1}}}}
\put(1257,-3619){\rotatebox{90.0}{\makebox(0,0)[b]{\smash{{\SetFigFont{14}{16.8}{\familydefault}{\mddefault}{\updefault}Order parameter $c_1$}}}}}
\put(10501,-1381){\makebox(0,0)[lb]{\smash{{\SetFigFont{14}{16.8}{\familydefault}{\mddefault}{\updefault}{\color[rgb]{0,0,.69}$c_1(\rho_c)$}%
}}}}
\put(10516,-3736){\makebox(0,0)[lb]{\smash{{\SetFigFont{14}{16.8}{\familydefault}{\mddefault}{\updefault}{\color[rgb]{0,0,.69}$c_1(\rho*)$}%
}}}}
\end{picture}%

%% file: c_kinetic_500.pspdftex
\begin{picture}(0,0)%
\includegraphics{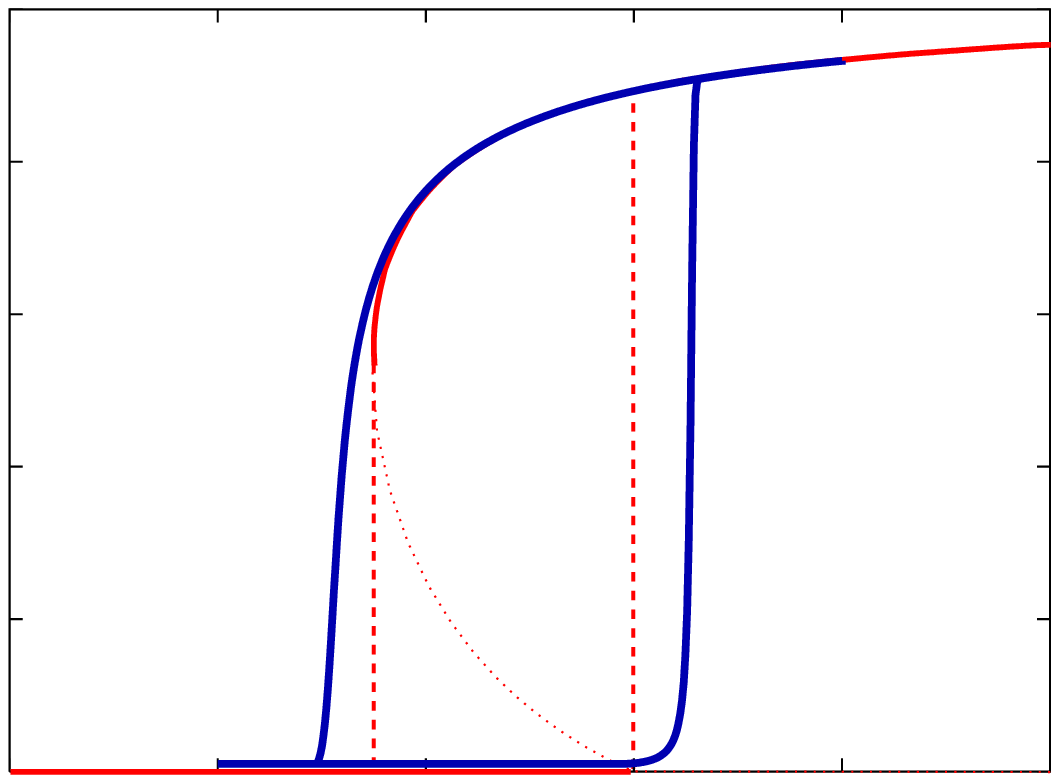}%
\end{picture}%
\setlength{\unitlength}{2368sp}%
\begingroup\makeatletter\ifx\SetFigFont\undefined%
\gdef\SetFigFont#1#2#3#4#5{%
  \reset@font\fontsize{#1}{#2pt}%
  \fontfamily{#3}\fontseries{#4}\fontshape{#5}%
  \selectfont}%
\fi\endgroup%
\begin{picture}(9549,7096)(900,-7643)
\put(1876,-787){\makebox(0,0)[rb]{\smash{{\SetFigFont{12}{14.4}{\sfdefault}{\mddefault}{\updefault}1}}}}
\put(6188,-7511){\makebox(0,0)[b]{\smash{{\SetFigFont{14}{16.8}{\familydefault}{\mddefault}{\updefault}Density $\rho$}}}}
\put(10350,-7136){\makebox(0,0)[b]{\smash{{\SetFigFont{12}{14.4}{\familydefault}{\mddefault}{\updefault}3}}}}
\put(8685,-7136){\makebox(0,0)[b]{\smash{{\SetFigFont{12}{14.4}{\familydefault}{\mddefault}{\updefault}2.5}}}}
\put(7020,-7136){\makebox(0,0)[b]{\smash{{\SetFigFont{12}{14.4}{\familydefault}{\mddefault}{\updefault}2}}}}
\put(5356,-7136){\makebox(0,0)[b]{\smash{{\SetFigFont{12}{14.4}{\familydefault}{\mddefault}{\updefault}1.5}}}}
\put(3691,-7136){\makebox(0,0)[b]{\smash{{\SetFigFont{12}{14.4}{\familydefault}{\mddefault}{\updefault}1}}}}
\put(2026,-7136){\makebox(0,0)[b]{\smash{{\SetFigFont{12}{14.4}{\familydefault}{\mddefault}{\updefault}0.5}}}}
\put(1876,-6886){\makebox(0,0)[rb]{\smash{{\SetFigFont{12}{14.4}{\familydefault}{\mddefault}{\updefault}0}}}}
\put(1876,-5666){\makebox(0,0)[rb]{\smash{{\SetFigFont{12}{14.4}{\familydefault}{\mddefault}{\updefault}0.2}}}}
\put(1876,-4446){\makebox(0,0)[rb]{\smash{{\SetFigFont{12}{14.4}{\familydefault}{\mddefault}{\updefault}0.4}}}}
\put(1876,-3227){\makebox(0,0)[rb]{\smash{{\SetFigFont{12}{14.4}{\familydefault}{\mddefault}{\updefault}0.6}}}}
\put(1876,-2007){\makebox(0,0)[rb]{\smash{{\SetFigFont{12}{14.4}{\familydefault}{\mddefault}{\updefault}0.8}}}}
\put(1227,-3687){\rotatebox{90.0}{\makebox(0,0)[b]{\smash{{\SetFigFont{14}{16.8}{\familydefault}{\mddefault}{\updefault}Order parameter $c_1$}}}}}
\end{picture}%

%% file: c_particle_500.pspdftex
\begin{picture}(0,0)%
\includegraphics{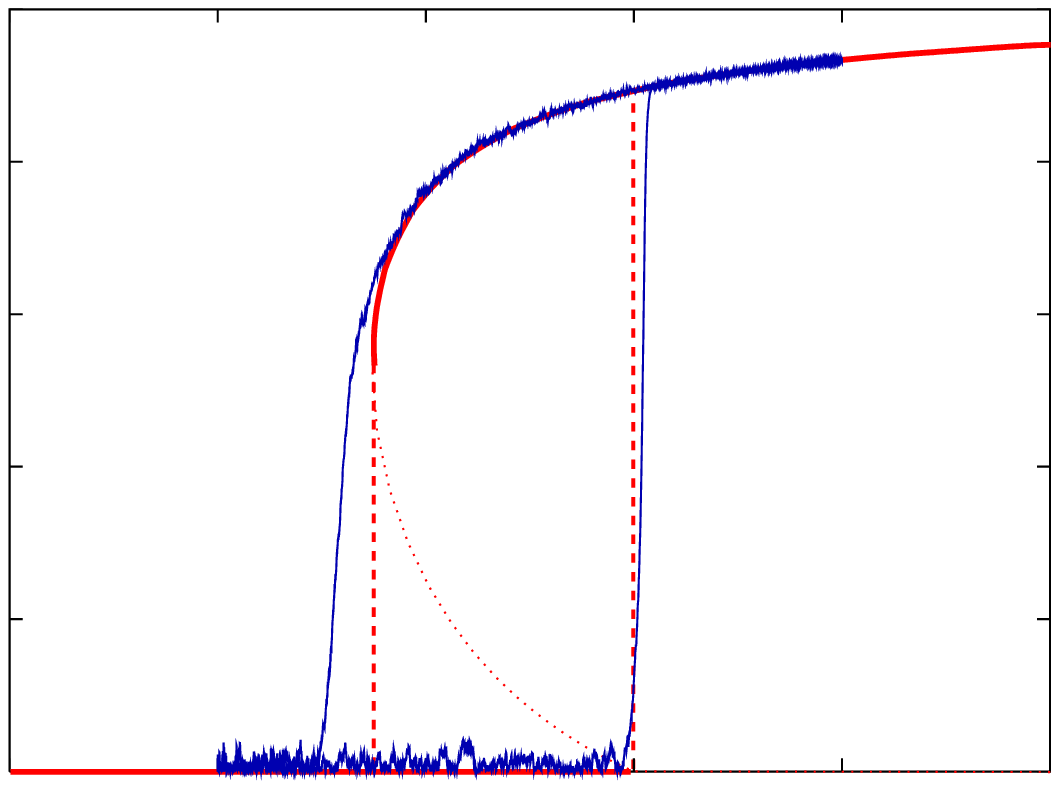}%
\end{picture}%
\setlength{\unitlength}{2368sp}%
\begingroup\makeatletter\ifx\SetFigFont\undefined%
\gdef\SetFigFont#1#2#3#4#5{%
  \reset@font\fontsize{#1}{#2pt}%
  \fontfamily{#3}\fontseries{#4}\fontshape{#5}%
  \selectfont}%
\fi\endgroup%
\begin{picture}(9549,7096)(900,-7643)
\put(1876,-787){\makebox(0,0)[rb]{\smash{{\SetFigFont{12}{14.4}{\sfdefault}{\mddefault}{\updefault}1}}}}
\put(6188,-7511){\makebox(0,0)[b]{\smash{{\SetFigFont{14}{16.8}{\familydefault}{\mddefault}{\updefault}Density $\rho$}}}}
\put(10350,-7136){\makebox(0,0)[b]{\smash{{\SetFigFont{12}{14.4}{\familydefault}{\mddefault}{\updefault}3}}}}
\put(8685,-7136){\makebox(0,0)[b]{\smash{{\SetFigFont{12}{14.4}{\familydefault}{\mddefault}{\updefault}2.5}}}}
\put(7020,-7136){\makebox(0,0)[b]{\smash{{\SetFigFont{12}{14.4}{\familydefault}{\mddefault}{\updefault}2}}}}
\put(5356,-7136){\makebox(0,0)[b]{\smash{{\SetFigFont{12}{14.4}{\familydefault}{\mddefault}{\updefault}1.5}}}}
\put(3691,-7136){\makebox(0,0)[b]{\smash{{\SetFigFont{12}{14.4}{\familydefault}{\mddefault}{\updefault}1}}}}
\put(2026,-7136){\makebox(0,0)[b]{\smash{{\SetFigFont{12}{14.4}{\familydefault}{\mddefault}{\updefault}0.5}}}}
\put(1876,-6886){\makebox(0,0)[rb]{\smash{{\SetFigFont{12}{14.4}{\familydefault}{\mddefault}{\updefault}0}}}}
\put(1876,-5666){\makebox(0,0)[rb]{\smash{{\SetFigFont{12}{14.4}{\familydefault}{\mddefault}{\updefault}0.2}}}}
\put(1876,-4446){\makebox(0,0)[rb]{\smash{{\SetFigFont{12}{14.4}{\familydefault}{\mddefault}{\updefault}0.4}}}}
\put(1876,-3227){\makebox(0,0)[rb]{\smash{{\SetFigFont{12}{14.4}{\familydefault}{\mddefault}{\updefault}0.6}}}}
\put(1876,-2007){\makebox(0,0)[rb]{\smash{{\SetFigFont{12}{14.4}{\familydefault}{\mddefault}{\updefault}0.8}}}}
\put(1227,-3687){\rotatebox{90.0}{\makebox(0,0)[b]{\smash{{\SetFigFont{14}{16.8}{\familydefault}{\mddefault}{\updefault}Order parameter $c_1$}}}}}
\end{picture}%

%% file: f_rho_hysteresis.pspdftex
\begin{picture}(0,0)%
\includegraphics{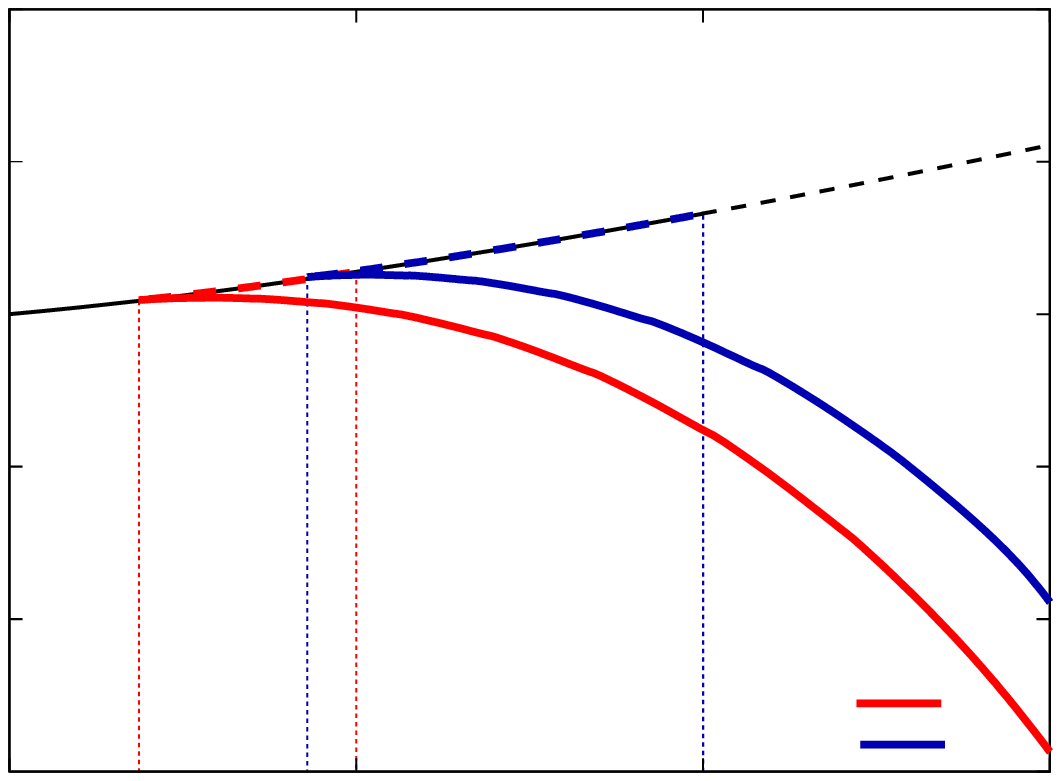}%
\end{picture}%
\setlength{\unitlength}{2368sp}%
\begingroup\makeatletter\ifx\SetFigFont\undefined%
\gdef\SetFigFont#1#2#3#4#5{%
  \reset@font\fontsize{#1}{#2pt}%
  \fontfamily{#3}\fontseries{#4}\fontshape{#5}%
  \selectfont}%
\fi\endgroup%
\begin{picture}(9425,7096)(1024,-7643)
\put(1876,-6886){\makebox(0,0)[rb]{\smash{{\SetFigFont{12}{14.4}{\familydefault}{\mddefault}{\updefault}-15}}}}
\put(1876,-2007){\makebox(0,0)[rb]{\smash{{\SetFigFont{12}{14.4}{\familydefault}{\mddefault}{\updefault}5}}}}
\put(1876,-5666){\makebox(0,0)[rb]{\smash{{\SetFigFont{12}{14.4}{\familydefault}{\mddefault}{\updefault}-10}}}}
\put(1876,-4446){\makebox(0,0)[rb]{\smash{{\SetFigFont{12}{14.4}{\familydefault}{\mddefault}{\updefault}-5}}}}
\put(1876,-3227){\makebox(0,0)[rb]{\smash{{\SetFigFont{12}{14.4}{\familydefault}{\mddefault}{\updefault}0}}}}
\put(1876,-787){\makebox(0,0)[rb]{\smash{{\SetFigFont{12}{14.4}{\familydefault}{\mddefault}{\updefault}10}}}}
\put(8671,-6646){\makebox(0,0)[rb]{\smash{{\SetFigFont{12}{14.4}{\familydefault}{\mddefault}{\updefault}n=3}}}}
\put(8656,-6331){\makebox(0,0)[rb]{\smash{{\SetFigFont{12}{14.4}{\familydefault}{\mddefault}{\updefault}n=2}}}}
\put(3001,-7111){\makebox(0,0)[b]{\smash{{\SetFigFont{14}{16.8}{\familydefault}{\mddefault}{\updefault}{\color[rgb]{1,0,0}$\rho_*$}%
}}}}
\put(7351,-4486){\rotatebox{330.0}{\makebox(0,0)[b]{\smash{{\SetFigFont{14}{16.8}{\familydefault}{\mddefault}{\updefault}{\color[rgb]{1,0,0}von Mises}%
}}}}}
\put(8626,-3661){\rotatebox{330.0}{\makebox(0,0)[b]{\smash{{\SetFigFont{14}{16.8}{\familydefault}{\mddefault}{\updefault}{\color[rgb]{0,0,.69}von Mises}%
}}}}}
\put(8626,-1936){\rotatebox{10.0}{\makebox(0,0)[b]{\smash{{\SetFigFont{14}{16.8}{\familydefault}{\mddefault}{\updefault}{\color[rgb]{0,0,0}uniform}%
}}}}}
\put(1351,-3211){\rotatebox{90.0}{\makebox(0,0)[b]{\smash{{\SetFigFont{14}{16.8}{\familydefault}{\mddefault}{\updefault}Free energy $\mathcal F$}}}}}
\put(6188,-7511){\makebox(0,0)[b]{\smash{{\SetFigFont{14}{16.8}{\familydefault}{\mddefault}{\updefault}Density $\rho$}}}}
\put(4351,-7111){\makebox(0,0)[b]{\smash{{\SetFigFont{14}{16.8}{\familydefault}{\mddefault}{\updefault}{\color[rgb]{0,0,.69}$\rho_*$}%
}}}}
\put(4801,-7111){\makebox(0,0)[b]{\smash{{\SetFigFont{14}{16.8}{\familydefault}{\mddefault}{\updefault}{\color[rgb]{1,0,0}$\rho_c$}%
}}}}
\put(7576,-7111){\makebox(0,0)[b]{\smash{{\SetFigFont{14}{16.8}{\familydefault}{\mddefault}{\updefault}{\color[rgb]{0,0,.69}$\rho_c$}%
}}}}
\put(2026,-7136){\makebox(0,0)[b]{\smash{{\SetFigFont{12}{14.4}{\rmdefault}{\mddefault}{\updefault}1}}}}
\put(10350,-7136){\makebox(0,0)[b]{\smash{{\SetFigFont{12}{14.4}{\familydefault}{\mddefault}{\updefault}4}}}}
\end{picture}%

%% file: diff_f_rho_hysteresis.pspdftex
\begin{picture}(0,0)%
\includegraphics{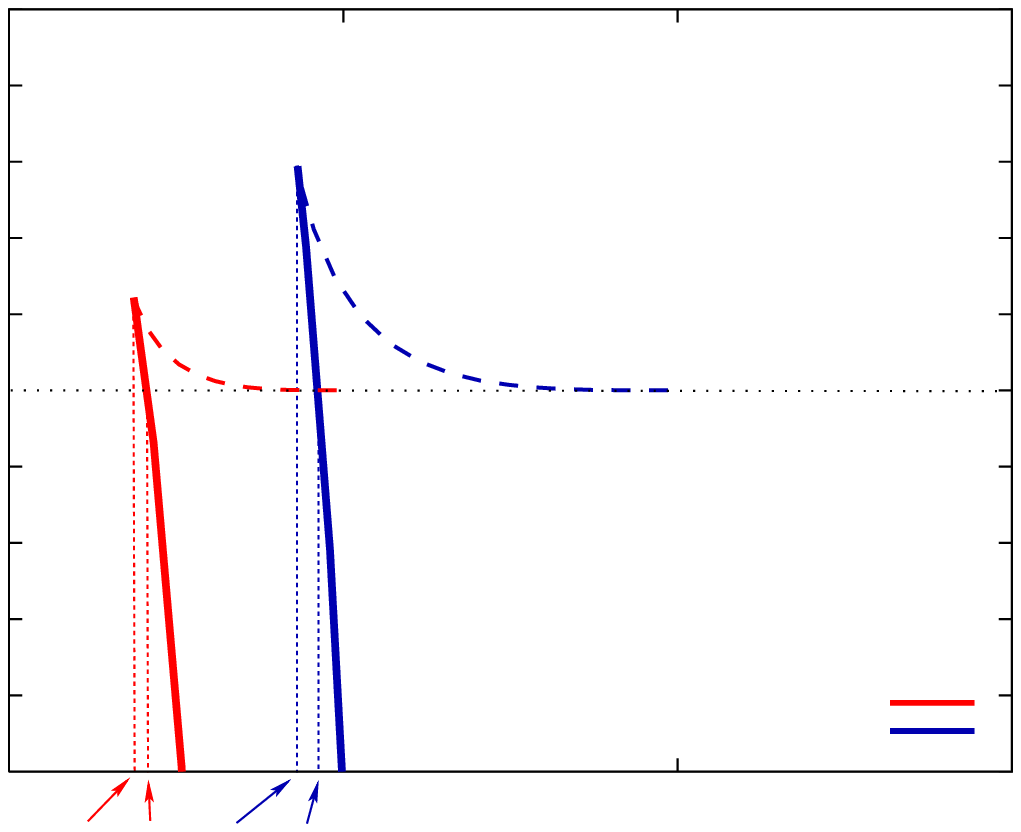}%
\end{picture}%
\setlength{\unitlength}{2368sp}%
\begingroup\makeatletter\ifx\SetFigFont\undefined%
\gdef\SetFigFont#1#2#3#4#5{%
  \reset@font\fontsize{#1}{#2pt}%
  \fontfamily{#3}\fontseries{#4}\fontshape{#5}%
  \selectfont}%
\fi\endgroup%
\begin{picture}(9050,7096)(1399,-7643)
\put(10350,-7136){\makebox(0,0)[b]{\smash{{\SetFigFont{12}{14.4}{\familydefault}{\mddefault}{\updefault}4}}}}
\put(2176,-6886){\makebox(0,0)[rb]{\smash{{\SetFigFont{12}{14.4}{\familydefault}{\mddefault}{\updefault}-0.1}}}}
\put(2176,-787){\makebox(0,0)[rb]{\smash{{\SetFigFont{12}{14.4}{\familydefault}{\mddefault}{\updefault}0.1}}}}
\put(9225,-6586){\makebox(0,0)[rb]{\smash{{\SetFigFont{12}{14.4}{\familydefault}{\mddefault}{\updefault}n=3}}}}
\put(2176,-3836){\makebox(0,0)[rb]{\smash{{\SetFigFont{12}{14.4}{\familydefault}{\mddefault}{\updefault}0}}}}
\put(4113,-7479){\makebox(0,0)[b]{\smash{{\SetFigFont{14}{16.8}{\familydefault}{\mddefault}{\updefault}{\color[rgb]{0,0,.69}$\rho_*$}%
}}}}
\put(4707,-7479){\makebox(0,0)[b]{\smash{{\SetFigFont{14}{16.8}{\familydefault}{\mddefault}{\updefault}{\color[rgb]{0,0,.69}$\rho_1$}%
}}}}
\put(2902,-7483){\makebox(0,0)[b]{\smash{{\SetFigFont{14}{16.8}{\familydefault}{\mddefault}{\updefault}{\color[rgb]{1,0,0}$\rho_*$}%
}}}}
\put(3497,-7483){\makebox(0,0)[b]{\smash{{\SetFigFont{14}{16.8}{\familydefault}{\mddefault}{\updefault}{\color[rgb]{1,0,0}$\rho_1$}%
}}}}
\put(9225,-6286){\makebox(0,0)[rb]{\smash{{\SetFigFont{12}{14.4}{\familydefault}{\mddefault}{\updefault}n=2}}}}
\put(1726,-3661){\rotatebox{90.0}{\makebox(0,0)[b]{\smash{{\SetFigFont{14}{16.8}{\familydefault}{\mddefault}{\updefault}Zoom: $\mathcal{F}_\kappa-\mathcal{F}(\rho)$}}}}}
\put(6338,-7511){\makebox(0,0)[b]{\smash{{\SetFigFont{14}{16.8}{\familydefault}{\mddefault}{\updefault}Density $\rho$}}}}
\put(5001,-7136){\makebox(0,0)[b]{\smash{{\SetFigFont{12}{14.4}{\familydefault}{\mddefault}{\updefault}{\color[rgb]{1,0,0}$\rho_c$}%
}}}}
\put(7675,-7136){\makebox(0,0)[b]{\smash{{\SetFigFont{12}{14.4}{\familydefault}{\mddefault}{\updefault}{\color[rgb]{0,0,.69}$\rho_c$}%
}}}}
\put(2326,-7136){\makebox(0,0)[b]{\smash{{\SetFigFont{12}{14.4}{\familydefault}{\mddefault}{\updefault}1}}}}
\end{picture}%

%% file: lambda_rho_hysteresis.pspdftex
\begin{picture}(0,0)%
\includegraphics{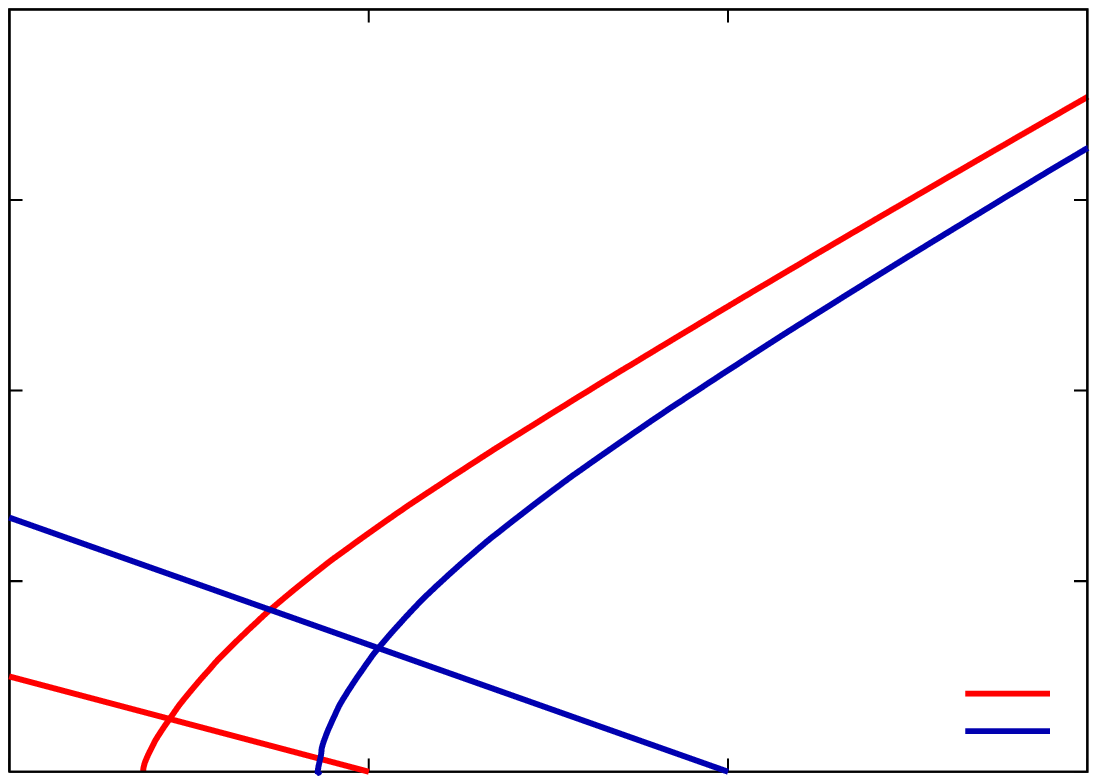}%
\end{picture}%
\setlength{\unitlength}{2368sp}%
\begingroup\makeatletter\ifx\SetFigFont\undefined%
\gdef\SetFigFont#1#2#3#4#5{%
  \reset@font\fontsize{#1}{#2pt}%
  \fontfamily{#3}\fontseries{#4}\fontshape{#5}%
  \selectfont}%
\fi\endgroup%
\begin{picture}(9590,7049)(859,-7620)
\put(6038,-7511){\makebox(0,0)[b]{\smash{{\SetFigFont{12}{14.4}{\familydefault}{\mddefault}{\updefault}Density $\rho$}}}}
\put(7475,-7136){\makebox(0,0)[b]{\smash{{\SetFigFont{12}{14.4}{\familydefault}{\mddefault}{\updefault}3}}}}
\put(10350,-7136){\makebox(0,0)[b]{\smash{{\SetFigFont{12}{14.4}{\familydefault}{\mddefault}{\updefault}4}}}}
\put(4601,-7136){\makebox(0,0)[b]{\smash{{\SetFigFont{12}{14.4}{\familydefault}{\mddefault}{\updefault}2}}}}
\put(1726,-7136){\makebox(0,0)[b]{\smash{{\SetFigFont{12}{14.4}{\familydefault}{\mddefault}{\updefault}1}}}}
\put(1576,-6886){\makebox(0,0)[rb]{\smash{{\SetFigFont{12}{14.4}{\familydefault}{\mddefault}{\updefault}0}}}}
\put(1576,-2312){\makebox(0,0)[rb]{\smash{{\SetFigFont{12}{14.4}{\familydefault}{\mddefault}{\updefault}3}}}}
\put(1126,-3961){\rotatebox{90.0}{\makebox(0,0)[b]{\smash{{\SetFigFont{12}{14.4}{\familydefault}{\mddefault}{\updefault}Rate of convergence $\lambda$}}}}}
\put(3151,-5086){\rotatebox{340.0}{\makebox(0,0)[rb]{\smash{{\SetFigFont{12}{14.4}{\familydefault}{\mddefault}{\updefault}uniform}}}}}
\put(2926,-6211){\rotatebox{345.0}{\makebox(0,0)[rb]{\smash{{\SetFigFont{12}{14.4}{\familydefault}{\mddefault}{\updefault}uniform}}}}}
\put(7501,-2836){\rotatebox{30.0}{\makebox(0,0)[rb]{\smash{{\SetFigFont{12}{14.4}{\familydefault}{\mddefault}{\updefault}von Mises}}}}}
\put(9301,-2911){\rotatebox{30.0}{\makebox(0,0)[rb]{\smash{{\SetFigFont{12}{14.4}{\familydefault}{\mddefault}{\updefault}von Mises}}}}}
\put(9225,-6211){\makebox(0,0)[rb]{\smash{{\SetFigFont{12}{14.4}{\familydefault}{\mddefault}{\updefault}n=2}}}}
\put(9225,-6586){\makebox(0,0)[rb]{\smash{{\SetFigFont{12}{14.4}{\familydefault}{\mddefault}{\updefault}n=3}}}}
\put(2791,-7112){\makebox(0,0)[b]{\smash{{\SetFigFont{14}{16.8}{\familydefault}{\mddefault}{\updefault}{\color[rgb]{1,0,0}$\rho_*$}%
}}}}
\put(4181,-7124){\makebox(0,0)[b]{\smash{{\SetFigFont{14}{16.8}{\familydefault}{\mddefault}{\updefault}{\color[rgb]{0,0,.69}$\rho_*$}%
}}}}
\put(1576,-5361){\makebox(0,0)[rb]{\smash{{\SetFigFont{12}{14.4}{\familydefault}{\mddefault}{\updefault}1}}}}
\put(1576,-3836){\makebox(0,0)[rb]{\smash{{\SetFigFont{12}{14.4}{\familydefault}{\mddefault}{\updefault}2}}}}
\put(1576,-811){\makebox(0,0)[rb]{\smash{{\SetFigFont{12}{14.4}{\familydefault}{\mddefault}{\updefault}4}}}}
\end{picture}%

%% file: c_epsilon.pspdftex
\begin{picture}(0,0)%
\includegraphics{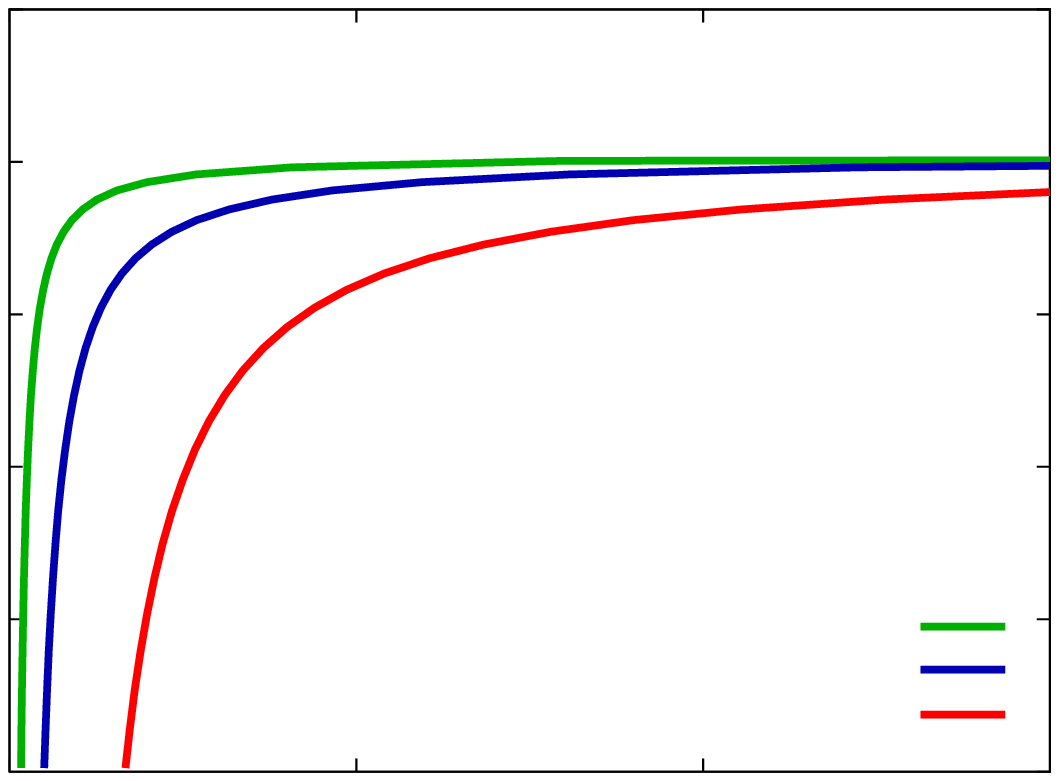}%
\end{picture}%
\setlength{\unitlength}{2368sp}%
\begingroup\makeatletter\ifx\SetFigFont\undefined%
\gdef\SetFigFont#1#2#3#4#5{%
  \reset@font\fontsize{#1}{#2pt}%
  \fontfamily{#3}\fontseries{#4}\fontshape{#5}%
  \selectfont}%
\fi\endgroup%
\begin{picture}(9515,7073)(934,-7620)
\put(9166,-5701){\makebox(0,0)[rb]{\smash{{\SetFigFont{12}{14.4}{\familydefault}{\mddefault}{\updefault}$\varepsilon=0.1$}}}}
\put(9167,-6050){\makebox(0,0)[rb]{\smash{{\SetFigFont{12}{14.4}{\familydefault}{\mddefault}{\updefault}$\varepsilon=0.3$}}}}
\put(9167,-6425){\makebox(0,0)[rb]{\smash{{\SetFigFont{12}{14.4}{\familydefault}{\mddefault}{\updefault}$\varepsilon=1$}}}}
\put(4801,-7136){\makebox(0,0)[b]{\smash{{\SetFigFont{12}{14.4}{\familydefault}{\mddefault}{\updefault}2}}}}
\put(7575,-7136){\makebox(0,0)[b]{\smash{{\SetFigFont{12}{14.4}{\familydefault}{\mddefault}{\updefault}4}}}}
\put(2026,-7136){\makebox(0,0)[b]{\smash{{\SetFigFont{12}{14.4}{\familydefault}{\mddefault}{\updefault}0}}}}
\put(1876,-6886){\makebox(0,0)[rb]{\smash{{\SetFigFont{12}{14.4}{\familydefault}{\mddefault}{\updefault}0}}}}
\put(1876,-5666){\makebox(0,0)[rb]{\smash{{\SetFigFont{12}{14.4}{\familydefault}{\mddefault}{\updefault}0.2}}}}
\put(1876,-4446){\makebox(0,0)[rb]{\smash{{\SetFigFont{12}{14.4}{\familydefault}{\mddefault}{\updefault}0.4}}}}
\put(1876,-3227){\makebox(0,0)[rb]{\smash{{\SetFigFont{12}{14.4}{\familydefault}{\mddefault}{\updefault}0.6}}}}
\put(1876,-2007){\makebox(0,0)[rb]{\smash{{\SetFigFont{12}{14.4}{\familydefault}{\mddefault}{\updefault}0.8}}}}
\put(1876,-787){\makebox(0,0)[rb]{\smash{{\SetFigFont{12}{14.4}{\familydefault}{\mddefault}{\updefault}1}}}}
\put(10350,-7136){\makebox(0,0)[b]{\smash{{\SetFigFont{12}{14.4}{\familydefault}{\mddefault}{\updefault}6}}}}
\put(6188,-7511){\makebox(0,0)[b]{\smash{{\SetFigFont{12}{14.4}{\familydefault}{\mddefault}{\updefault}Density $\rho$}}}}
\put(1201,-3661){\rotatebox{90.0}{\makebox(0,0)[b]{\smash{{\SetFigFont{12}{14.4}{\familydefault}{\mddefault}{\updefault}Order parameter $c_1$}}}}}
\end{picture}%

%% file: Theta_epsilon.pspdftex
\begin{picture}(0,0)%
\includegraphics{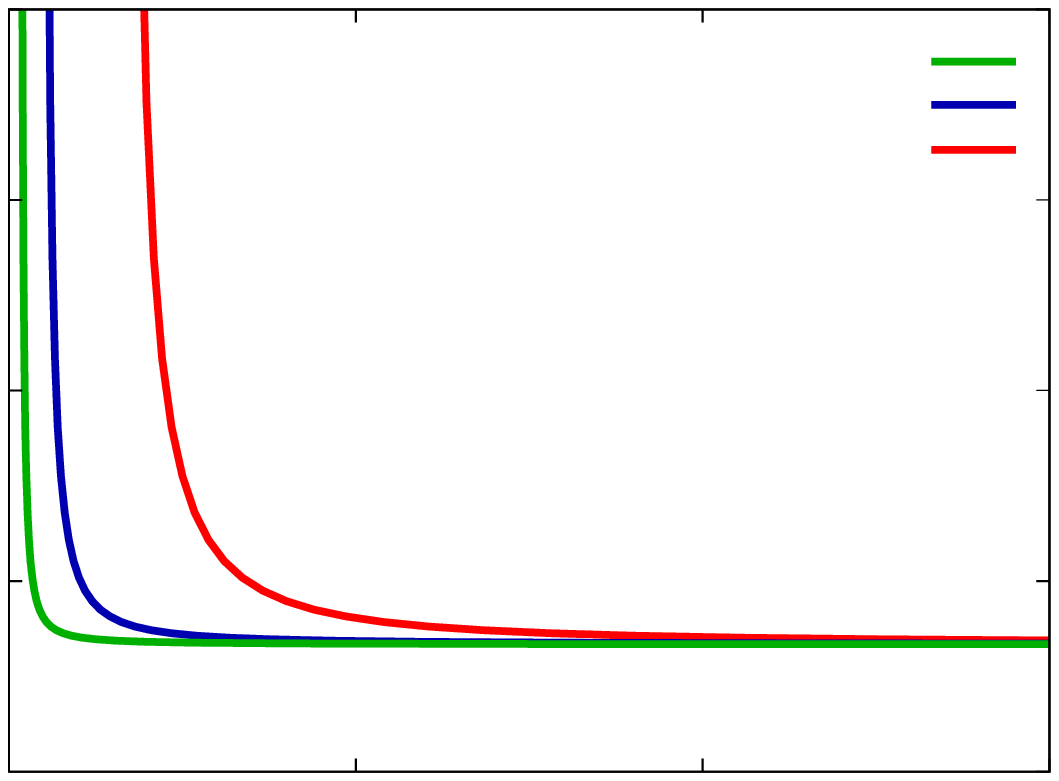}%
\end{picture}%
\setlength{\unitlength}{2368sp}%
\begingroup\makeatletter\ifx\SetFigFont\undefined%
\gdef\SetFigFont#1#2#3#4#5{%
  \reset@font\fontsize{#1}{#2pt}%
  \fontfamily{#3}\fontseries{#4}\fontshape{#5}%
  \selectfont}%
\fi\endgroup%
\begin{picture}(9479,7073)(970,-7620)
\put(9256,-1182){\makebox(0,0)[rb]{\smash{{\SetFigFont{12}{14.4}{\familydefault}{\mddefault}{\updefault}$\varepsilon=0.1$}}}}
\put(9257,-1531){\makebox(0,0)[rb]{\smash{{\SetFigFont{12}{14.4}{\familydefault}{\mddefault}{\updefault}$\varepsilon=0.3$}}}}
\put(9257,-1906){\makebox(0,0)[rb]{\smash{{\SetFigFont{12}{14.4}{\familydefault}{\mddefault}{\updefault}$\varepsilon=1$}}}}
\put(4801,-7136){\makebox(0,0)[b]{\smash{{\SetFigFont{12}{14.4}{\familydefault}{\mddefault}{\updefault}2}}}}
\put(7575,-7136){\makebox(0,0)[b]{\smash{{\SetFigFont{12}{14.4}{\familydefault}{\mddefault}{\updefault}4}}}}
\put(10350,-7136){\makebox(0,0)[b]{\smash{{\SetFigFont{12}{14.4}{\familydefault}{\mddefault}{\updefault}6}}}}
\put(2026,-7136){\makebox(0,0)[b]{\smash{{\SetFigFont{12}{14.4}{\familydefault}{\mddefault}{\updefault}0}}}}
\put(1876,-6886){\makebox(0,0)[rb]{\smash{{\SetFigFont{12}{14.4}{\familydefault}{\mddefault}{\updefault}0}}}}
\put(1876,-5361){\makebox(0,0)[rb]{\smash{{\SetFigFont{12}{14.4}{\familydefault}{\mddefault}{\updefault}0.5}}}}
\put(1876,-3836){\makebox(0,0)[rb]{\smash{{\SetFigFont{12}{14.4}{\familydefault}{\mddefault}{\updefault}1}}}}
\put(1876,-2312){\makebox(0,0)[rb]{\smash{{\SetFigFont{12}{14.4}{\familydefault}{\mddefault}{\updefault}1.5}}}}
\put(1876,-787){\makebox(0,0)[rb]{\smash{{\SetFigFont{12}{14.4}{\familydefault}{\mddefault}{\updefault}2}}}}
\put(6188,-7511){\makebox(0,0)[b]{\smash{{\SetFigFont{12}{14.4}{\familydefault}{\mddefault}{\updefault}Density $\rho$}}}}
\put(1201,-3661){\rotatebox{90.0}{\makebox(0,0)[b]{\smash{{\SetFigFont{12}{14.4}{\familydefault}{\mddefault}{\updefault}Coefficient $\Theta$}}}}}
\end{picture}%